\documentclass[11pt]{article}
\usepackage{lipsum}
\usepackage{bbm}
\usepackage[dvipsnames]{xcolor}
\usepackage{amssymb}
\usepackage{amsmath, amssymb, amsthm}
\usepackage{mathrsfs}
\usepackage{slashed}
\usepackage{mathtools}
\usepackage{fancyhdr}
\usepackage{xcolor}
\usepackage{blindtext}
\usepackage{tensor}
\usepackage{hyperref}
\usepackage{titlesec}
\usepackage{cancel}
\usepackage[title]{appendix}
\usepackage{bbm}
\usepackage{authblk}
\usepackage{appendix}
\usepackage{comment}

\newtheorem{theorem}{Theorem}[section]
\newtheorem{definition}{Definition}[section]
\newtheorem{lemma}{Lemma}[section]
\newtheorem{remark}{Remark}[section]
\newtheorem{proposition}{Proposition}[section]
\newtheorem{corollary}{Corollary}[section]
\numberwithin{equation}{section}
\newcommand{\enmo}[1]{{#1}^{[\textbf{n-1}]}}
\newcommand{\enmt}[1]{{#1}^{[\textbf{n-2}]}}

\newcommand{\hsp}{\hspace{.5mm}}
\newcommand{\be}{\begin{equation}}
\newcommand{\ee}{\end{equation}}
\newcommand{\eIan}{e_{Ia}^{[\bf n]}}

\newcommand{\en}[1]{{#1}^{[\bf n]}}
\newcommand{\ze}[1]{{#1}^{[\bf 0]}}

\newcommand{\uI}{\underline{I}}
\newcommand{\tr}{\text{tr} \hspace{.5mm}}
\allowdisplaybreaks[4]

\title{A localized construction of Kasner-like singularities}

\author{Nikolaos Athanasiou\footnote{Department of Mathematics \& Applied Mathematics, University of Crete, Voutes Campus, 70013 Heraklion, Greece, n.athanasiou@uoc.gr.}\,\, and Grigorios Fournodavlos\footnote{Department of Mathematics \& Applied Mathematics, University of Crete, Voutes Campus, 70013 Heraklion, Greece, gfournodavlos@uoc.gr.} \footnote{Institute of Applied and Computational Mathematics,
FORTH, 70013 Heraklion, Greece.}}

\date{}

\begin{document}

\maketitle

\begin{abstract}
We construct local, in spacetime, singular solutions to the Einstein vacuum equations that exhibit Kasner-like behavior in their past boundary. Our result can be viewed as a localization (in space) of the construction in \cite{FL}. We also prove a refined uniqueness statement and give a simple argument that generates general asymptotic data for Kasner-like singularities, enjoying all expected degrees of freedom, albeit only locally in space. The key difference of the present work with \cite{FL} is our use of a first order symmetric hyperbolic formulation of the Einstein vacuum equations, relative to the connection coefficients of a parallelly propagated orthonormal frame which is adapted to the Gaussian time foliation. This makes it easier to localize the construction, since elliptic estimates are no longer required to complete the energy argument.
\end{abstract}

\tableofcontents

\section{Introduction}

Ever since the discovery of the first explicit solutions to the Einstein equations containing a Big Bang singularity, by Kasner (1921) and Friedmann (1922), there have been attempts to understand the nature of the general cosmological singularity. Kasner-like singularities are a specific class of spacelike singularities whose leading order behavior toward the singularity resembles that of a Kasner solution at each spatial point on the singular hypersurface. More precisely, the first approximation of the spacetime metric reads: 
\begin{align}\label{Kasner-like}
{\bf g}\cong -dt^2+\sum_{i=1}^3  t^{2p_i(x)}\omega^i\otimes\omega^i,\qquad\omega^i=\sum_{j=1}^3c_{ij}(x)dx^j,
\end{align}
where $(t,x)\in (0,T]\times\Sigma$, for some closed spatial topology $\Sigma$. Here, the singularity is synchronized at the limiting hypersurface $t=0$.

The asymptotic behavior \eqref{Kasner-like} first appeared in the heuristic analysis of \cite{LK}, which gave the false conclusion that general solutions do not contain singularities. Instead, what the latter work indicated is that Kasner-like behavior is non-generic in 1+3 vacuum. In this setting, the Kasner exponents must verify the algebraic relations
\begin{align}\label{sum.pi}
\sum_{i=1}^3 p_i(x)=\sum_{i=1}^3 p_i^2(x)=1,\qquad \text{for all $x\in\Sigma$},
\end{align}
which forces one Kasner exponent to be negative, say $p_1(x)<0$, and the other two to be positive, $p_2(x),p_3(x)>0$. In \cite{LK}, they concluded that \eqref{Kasner-like} is consistent with the Einstein vacuum equations
\begin{align}\label{EVE}
{\bf R}_{\mu\nu}=0,
\end{align}
if and only if 
\begin{align}\label{omega1}
\omega^1(x)\wedge \text{d}\omega^1(x)=0,\qquad\text{for all $x\in\Sigma$}.
\end{align}
However, \eqref{omega1} eliminates one of the gravitational degrees of freedom and there is a priori no reason for it to be valid. In fact, condition \eqref{omega1} implies that we can make a change of spatial coordinates,\footnote{Notice that \eqref{omega1} is an integrability condition. Using the Frobenius theorem, we can assume that $\omega^1$ is a multiple of a coordinate 1-form.} such that \eqref{Kasner-like} becomes
\begin{align}\label{Kasner-like2}
{\bf g}\cong -dt^2+\sum_{i,j=1}^3  t^{2\max\{p_i(x),p_j(x)\}}c_{ij}(x)dx^idx^j.
\end{align}
Here, the degrees of freedom are interpreted in a function counting sense and correspond to the functions $p_i(x),c_{ij}(x)$. The latter functions can also be viewed as the asymptotic data for the Einstein vacuum equations at the singularity (see Definition \ref{defn1}).
 
Later, in the subsequent work \cite{BKL}, a more involved heuristic argument was put forth, which concluded that the general Big Bang singularity is oscillatory, in the sense that along a timelike curve of fixed $x$, as $t\rightarrow0$, there is an infinite number of ``bounces'' swapping the Kasner exponents in a specific, but chaotic, manner. Nevertheless, in the interval between two bounces the solution is still be modeled by a Kasner-like singularity, a different one in each interval. This is typically referred to as the BKL conjecture. 

The BKL heuristics have been extended to various settings, arriving to similar conclusions. We refer the interested reader to \cite{BH} for an overview of the subject. We should only note that, surprisingly, in the presence of a scalar field \cite{BK}, a stiff fluid \cite{Bar} or in sufficiently high spatial dimensions \cite{DHS}, the oscillations are silenced and Kasner-like behavior at the Big Bang singularity is expected to be generic. This is also called the sub-critical regime.

Given the complicated nature of the oscillatory scenario, the rigorous evidence in its favor is scarce, restricted to the homogeneous class of solutions \cite{BD,HUN,RT,Rin0}.\footnote{In specific symmetric settings with only one bounce, the recent works \cite{Li,Li2} are the first to go beyond homogeneity.}
On the other hand, Kasner-like singularities have by now been understood to a sufficient extent. The main types of results that exist are roughly divided into the following categories:
\begin{enumerate}
\item {\bf Gowdy symmetry.} Classification of generic solutions, in the polarized class \cite{CIM} and in the more general unpolarized setting \cite{Li,Rin,Rin2,Rin3}. The behavior is Kasner-like, apart from finitely many points (for unpolarized solutions) where the asymptotic data can form discontinuities at $t=0$, called spikes.\footnote{Constructions of singular solutions with spikes have been achieved in \cite{HUL,ML,RW}.}

\item {\bf Constructions.} Starting with the Kasner-like ansatz \eqref{Kasner-like}, the goal is to solve the Einstein equations for a remainder which is better behaved as $t\rightarrow0$. Numerous constructions have been achieved in various settings \cite{ABIL,ABIL2,AR,CBIM,DHRW,FL,Fr,IM,KR,Kl,Ren,St}. In the sub-critical regime, the constructed solutions enjoy all gravitational degrees of freedom, while in 1+3 vacuum the condition \eqref{omega1} is satisfied by assumption. 

\item {\bf Stability.} Given initial data away from the singularity, close to explicit Kasner data, it was first established in \cite{RS,RS2} that a Kasner-like singularity will form in the past, in the near isotropic case for the Einstein-scalar field and stiff fluid models. This result has been extended to $\mathbb{S}^3$ spatial topology \cite{Sp}, negatively curved spatial topology \cite{FU}, the Einstein-Euler-scalar field model \cite{BO2}, the
Einstein-Vlasov-scalar field model \cite{FU2}, and a localized stability result has been achieved in \cite{BO}. A moderate, but not isotropic, range of Kasner exponents was treated in \cite{RS3}, for sufficiently higher spatial dimensions. Stable Kasner-like singularity formation for the full range of Kasner exponents in the sub-critical regime was proven in \cite{FRS}. The latter result has been recently extended to general Kasner-like initial data in \cite{GPR}, always within the sub-critical regime. To have stable singularity formation outside of the sub-critical regime, a specific class of initial data is required which guarantees that an integrability condition, like \eqref{omega1}, is verified, hence, silencing any oscillations (instabilities). One such class is $U(1)$ polarized symmetry \cite{FRS} (see also \cite{AF}).

\item {\bf Conditional.} Kasner-like behavior is derived by assuming scale-invariant curvature bounds \cite{Lo,Lo2}. Alternatively, bounds on the renormalized Weingarten map result in a complete description of the geometry near the Kasner-like singularity \cite{Rin4,Rin5}.  
\end{enumerate}

\subsection{Goal and motivation}

Our goal in the present paper is to localize the construction of Kasner-like singularities in 1+3 vacuum that was achieved in \cite{FL}. Before the latter work, all previous constructions of Kasner-like singularities had been restricted to either symmetry or analyticity. One of the main new ingredients in \cite{FL} was a novel estimate of the second fundamental form of the level sets of Gaussian time, using a formulation of the Einstein equations at third order in the metric. Among other things, this required the derivation of elliptic estimates, due to a derivative loss issue. For this reason, the construction is not easily localizable. One would have to show that the overall sign of the boundary terms generated when localizing the energy estimates is favorable. Although this might be possible, we prefer to work instead with a first order symmetric hyperbolic formulation of the Einstein vacuum equations, at first order in the second fundamental form of the Gaussian time slices, which is more easily localizable, see Sections \ref{subsec:framework}, \ref{sec:ADM} for more details on our framework of preference.

Having a localized construction of Kasner-like singularities could prove useful in various ways. Firstly, there is no reason why the entire singularity should be Kasner-like. Since different spatial points on a Big Bang hypersurface have disjoint future cones, once we zoom sufficiently close to the singularity, it is entirely possible that part of the singularity is oscillatory, while in some other region Kasner-like. Hence, a local patch of a Kasner-like singularity could be seen as part of a solution with more complicated singular behavior. Moreover, in connection with black hole singularities, assuming that part of them is spacelike and non-oscillatory, one could envision that a local patch of a Kasner-like singularity could be attached to another part of the singularity which happens to be null, see \cite{Li,Mo} for spherically symmetric examples.

Another motivation for our work is related to the definition of asymptotic data for Kasner-like singularities and it is described in the next subsection.

\subsection{Asymptotic data in 1+3 vacuum}

We are working in 1+3 vacuum, where Kasner-like singularities are expected to be non-generic. Hence, condition \eqref{omega1} must be satisfied. In other words, the metric will have the leading order behavior \eqref{Kasner-like2}.
The following definition is taken from \cite{FL}, only localized in a subset of $\mathbb{T}^3$.
\begin{definition}[Gauge dependent data]\label{defn1}
Let $c_{ij},p_i : [0,\delta]^3\subset\mathbb{T}^3 \to \mathbb{R}$ be smooth functions, for $i,j=1,2,3$. We say that they form a vacuum initial data set on the singularity, if they satisfy the following conditions:
\begin{enumerate}
\item \label{def1.1} $c_{ii}(x)>0$ and $c_{ij}(x)=c_{ji}(x)$, for all $i,j= 1, 2, 3$ and $x\in[0,\delta]^3$.

\item \label{def1.2} $p_1(x)<p_2(x)<p_3(x)$ and $\sum_{i=1}^3 p_i(x) = \sum_{i=1}^3 p_i^2(x) =1$, for all $x\in[0,\delta]^3$. 

\item \label{def1.3} For $i=1,2,3$, there holds 
\begin{align}\label{momentum}
\sum_{l=1}^3 \begin{bmatrix} \frac{\partial_i c_{ll}}{c_{ll}}(p_l-p_i) + 2\partial_l {\kappa_i}^l+ \mathbbm{1}_{\{l>i\}} \frac{\partial_l (c_{11}c_{22}c_{33})}{c_{11}c_{22}c_{33}} {\kappa_i}^l\end{bmatrix}=0,
\end{align}
where ${\kappa_i}^i :=-p_i$ (no summation), ${\kappa_i}^l =0$ if $l <i$ and ${\kappa_1}^2 := (p_1-p_2)\frac{c_{12}}{c_{22}},  {\kappa_2}^3 := (p_2-p_3)\frac{c_{23}}{c_{33}},  {\kappa_1}^3 := (p_2-p_1) \frac{c_{12}c_{23}}{c_{22}c_{33}} +(p_1-p_3)\frac{c_{13}}{c_{33}}$.
\end{enumerate} 
\end{definition} 
Point \ref{def1.1} in Definition \ref{defn1} implies that a metric with the asymptotic profile \eqref{Kasner-like2} is indeed a Lorentzian metric for sufficiently small times. The distinct Kasner exponents assumption in point \ref{def1.2} should not be necessary to define asymptotic data for Kasner-like singularities. However, contact points between the two positive Kasner exponents introduce technical difficulties that we do not want to deal with here. There are five constraints in Definition \ref{defn1}, two algebraic 
and three differential, which correspond asymptotically to the Hamiltonian and momentum constraints:
\begin{align}
\label{Ham.const}R-|k|^2+(\text{tr}k)^2=&\,0,\\
\label{mom.const}\text{div}k-\nabla\text{tr}k=&\,0,
\end{align}
as well as the constant mean curvature condition.
\begin{remark}[Degrees of freedom]\label{rem:deg.free}
The degrees of freedom for general solutions to \eqref{EVE} are four. There are six $c_{ij}$'s, three $p_i$'s, and five constraints in Definition \ref{defn1}. This leaves roughly four free functions (see Section \ref{sec:asym}). However, the asymptotic profile \eqref{Kasner-like2} allows for a coordinate change $\widetilde{x}^3=f(x^1,x^2,x^3)$, without changing its form. This extra gauge freedom could for example fix $c_{33}=1$. Thus, we are left with three functional degrees of freedom. According to the heuristic analysis in \cite{LK}, this is the largest class of metrics with Kasner-like behavior we could expect. 
\end{remark}
Definition \ref{defn1} relies on the fact that there exists a certain coordinate system $(t,x^1,x^2,x^3)$, such that the asymptotic profile of the spacetime metric is given by \eqref{Kasner-like}, from which one can read off the asymptotic data $p_i,c_{ij}$ from that profile. Therefore, the previous data are gauge dependent.
A second, covariant, definition of asymptotic data for Kasner-like singularities was subsequently given by Ringstr\"om \cite{Rin6}, proving that it locally coincides with Definition \ref{defn1}. 
\begin{definition}[Covariant data]\label{defn2}
Let $(\overline{\mathcal{M}}, \overline{h})$ be a smooth, $3$-dimensional Riemannian manifold and let $\mathcal{K}$ be a smooth $(1,1)$-tensor field defined on $\overline{\mathcal{M}}$. The triplet $(\overline{\mathcal{M}},\overline{h},\mathcal{K})$ is called a non-degenerate, quiescent vacuum initial data set on the singularity if:
\begin{enumerate}
\item \label{def2.1} $\mathrm{tr}\mathcal{K}=(\mathrm{tr}\mathcal{K})^2=1$, $\mathcal{K}$ is symmetric with respect to $\overline{h}$,
and $\mathrm{div}_{\overline{h}}\mathcal{K} =0$.
\item \label{def2.2} The eigenvalues $p_1<p_2<p_3$ of $\mathcal{K}$ are distinct.
\item \label{def2.3} $[\mathcal{X}_2,\mathcal{X}_3]^1=0$ on $\overline{\mathcal{M}}$, where $\mathcal{X}_A$ are a basis of eigenvectors of $\mathcal{K}$ corresponding to the eigenvalues $p_A$, $A=1,2,3$.
\end{enumerate}
\end{definition}
The tensor $\mathcal{K}$ corresponds to the limit of the renormalized Weingarten map, as $t\rightarrow0$, while $\overline{h}$ corresponds to the spatial part of the metric \eqref{Kasner-like2} without the powers of $t$. Hence, the first part of point \ref{def2.1} in Definition \ref{defn2} corresponds to the algebraic Kasner relations \eqref{sum.pi}. The second part of point \ref{def2.1} corresponds \cite{Rin6} to the differential constraint \eqref{momentum}, while point \ref{def2.3} guarantees that condition \eqref{omega1} is satisfied, when the metric is assumed to have the form \eqref{Kasner-like}. Definition \ref{defn2} requires no coordinate system to make sense of asymptotic initial data sets. However, in order to relate the latter to Kasner-like singularities, it still relies on the choice of a time foliation for the definition the renormalized Weingarten map. 

Definition \ref{defn2}, together with our localized construction (see Section \ref{subsec:main.thm}), can be used to obtain an entire Kasner-like singularity, as in \cite{FL}, but without requiring the existence of a global, gauge-dependent (Definition \ref{defn1}) asymptotic data set, see Corollary \ref{cor:exist.entire}. Interestingly, a more general construction was performed in the recent work \cite{Fr}, starting with covariant data (in the spirit of Definition \ref{defn2}, see \cite{Rin6}), which includes a non-linear scalar field and a cosmological constant. The energy argument in \cite{Fr} is closely related to the one employed in \cite{FL}. In particular, elliptic estimates are still needed to complete the construction.

\subsection{Brief framework}\label{subsec:framework}

Consider a 1+3 splitting of the spacetime metric, relative to Gaussian time:
\begin{align}\label{metric}
{\bf g}=-dt^2+g_{ij}dx^idx^j,
\end{align}
where $(t,x)\in (0,T]\times[0,\delta]^3$, such that the singularity is synchronized at $t=0$. This makes it easier to compare our results to the usual Kasner-like ansatz \eqref{Kasner-like2}.

Also, we consider an orthonormal frame of the form 
\begin{align}\label{frame}
e_0=\partial_t,\qquad e_I=e_{Ia}\partial_a,
\end{align}
which is parallelly propagated along $e_0$:
\begin{align}\label{De0eI}
D_{e_0}e_0=D_{e_0}e_I=0,
\end{align}
where $D$ is the Levi-Civita connection of ${\bf g}$. We then formulate the Einstein vacuum equations as a first order symmetric hyperbolic system in the connection coefficients of the frame (see Section \ref{sec:ADM}):
\begin{align}\label{k.gamma}
k_{IJ}={\bf g}(D_{e_I}e_J,e_0)=k_{JI},\qquad
\gamma_{IJB}={\bf g}(D_{e_I}e_J,e_B)=-\gamma_{IBJ}.
\end{align}
Here, $k_{IJ}$ is the second fundamental form of the constant $t$-hypersurfaces, $\Sigma_t$, contracted against the frame components $e_I,e_J$. 


\subsection{Main results}\label{subsec:main.thm}

The main existence theorem that we prove is the following.
\begin{theorem}[Existence]\label{thm:exist}
Let $p_i(x),c_{ij}(x)$ be smooth asymptotic data satisfying Definition \ref{defn1} (see Proposition \ref{prop:asym} for an existence statement). There exists a Lorentzian metric ${\bf g}^{[\bf n]}=-dt^2+g^{[\bf n]}_{ij}dx^idx^j$,
defined by the variables $e_{Ia}^{[\bf n]},k^{[\bf n]}_{IJ},\gamma_{IJB}^{[\bf n]}$ of an iterative procedure (carried out in Section \ref{sec:approx.sol}), which satisfies:
\begin{align}\label{exist.approx}
\begin{split}
|\partial_x^\alpha(g^{[\bf n]}_{ij}-c_{ij}t^{2p_{\max\{i,j\}}})|\leq &\,C_{\alpha,n}t^{2p_{\max\{i,j\}}+\varepsilon},\qquad
|\partial_x^\alpha{\bf R}^{[\bf n]}_{\mu\nu}|\leq C_{\alpha,n}t^{-2+n\varepsilon},\\
\varepsilon=\min_{x\in[0,\delta]^3}&\{1-p_3(x)>0,\,p_3(x)-p_2(x)\}>0,
\end{split}
\end{align}
for any $n\in\mathbb{N}$, multi-index $\alpha$, indices $\mu,\nu$ (contracted relative to the frame), and $(t,x)\in (0,t_n]\times[0,\delta]^3$, for some $t_n$ sufficiently small. 

Moreover, for any $s\ge4$ and $N_0\in\mathbb{N}$, there exists $n=n_{N_0,s}$ sufficiently large and a $C^2$ Lorentzian metric ${\bf g}={\bf g}^{[\bf n]}+{\bf g}^{(d)}$, defined by the corresponding reduced variables 
\begin{align}\label{exist.d}
e_{Ia}=e^{[\bf n]}_{Ia}+e^{(d)}_{Ia},\qquad k_{IJ}=k^{[\bf n]}_{IJ}+k^{(d)}_{IJ},\qquad\gamma_{IJB}=\gamma^{[\bf n]}_{IJB}+\gamma_{IJB}^{(d)},
\end{align}
such that it solves the Einstein vacuum equations and satisfies the estimate (see Section \ref{subsec:exist} for the definition of the norms):
\begin{align}\label{exist.d.est}
\|e^{(d)}\|^2_{H^s(U_t)}+\|k^{(d)}\|^2_{H^s(U_t)}+\|\gamma^{(d)}\|^2_{H^s(U_t)}\leq t^{2N_0},
\end{align}
in a local domain $\{U_t\}_{t\in(0,T_{N_0,s,n}]}\subset(0,t_n]\times[0,\delta]^3$ with spacelike future boundary (see Section \ref{subsec:domain} for the precise definition).
\end{theorem}
\begin{proof}
The first part of the theorem regarding the approximate solution ${\bf g}^{[\bf n]}$ follows from Theorem \ref{thm:approx.sol} and Lemma \ref{lem:pre1}. The second part about the actual solution follows from Theorem \ref{thm:loc.exist} and Proposition \ref{prop:EVE}. 
\end{proof}
\begin{remark}[Choice of parameters and domain]
For convenience, we define the domain of definition $\{U_t\}_{t\in(0,T_{N_0,s,n}]}$ relative to the approximate solution ${\bf g}^{[\bf n]}$. It is strictly contained in the domain of dependence of $U_0$, having spacelike future boundary. This is true for ${\bf g}$ as well, since the remainder term ${\bf g}^{(d)}$ is lower order. Once we have fixed $N_0$, which is tied to the decay that we wish the remainder terms $e^{(d)}_{Ia},k^{(d)}_{IJ},\gamma_{IJB}^{(d)}$ to have, as $t\rightarrow0$, $n$ is taken sufficiently large and the domain is fixed. Increasing $s$ only shrinks the time of existence $T_{N_0,s,n}$, such that the estimate \eqref{exist.d.est} remains valid.
\end{remark}
\begin{remark}\label{rem:asym.reg}
The asymptotic data $p_i,c_{ij}$ need not be smooth for Theorem \ref{thm:exist} to hold. However, working with other regularity classes, like $C^r$, would require keeping track of the dependency of $r$ on the other parameters. For example, our method of proof would require that $r$ tends to infinity, as $\varepsilon
\rightarrow0$ or $n\rightarrow+\infty$.
\end{remark}
The next two theorems deal with the uniqueness and regularity of the solution in Theorem \ref{thm:exist}.
\begin{theorem}[Smoothness]\label{thm:smooth}
The solution ${\bf g}$ furnished by Theorem \ref{thm:exist} is smooth, provided $n,N_0$ are taken sufficiently large.
\end{theorem}
\begin{theorem}[Uniqueness]\label{thm:uniq}
$(i)$ Let ${\bf g},\widetilde{\bf g}$ be two solutions to \eqref{EVE} of the form \eqref{metric}, satisfying 
\begin{align}\label{uniq.cond.a}
\begin{split}
|\partial_x^{\alpha_1}(e_{Ia}-e^{[\bf 0]}_{Ia})|\leq C t^{p_I-2\min\{p_I,p_a\}+\varepsilon},\qquad|\partial_x^{\alpha_1}(\widetilde{e}_{Ia}-e^{[\bf 0]}_{Ia})|\leq C t^{p_I-2\min\{p_I,p_a\}+\varepsilon},\\
|\partial_x^{\alpha_1}(k_{IJ}-k^{[\bf 0]}_{IJ})|\leq C t^{-1+|p_I-p_J|+\varepsilon},\qquad|\partial_x^{\alpha_2}(\widetilde{k}_{IJ}-k^{[\bf 0]}_{IJ})|\leq C t^{-1+|p_I-p_J|+\varepsilon},
\end{split}
\end{align}
for all $|\alpha_1|\leq2$, $|\alpha_2|\leq 1$, and $(t,x)\in \{U_t\}_{t\in(0,T]}$, where the latter domain is defined relative to ${\bf g}^{\bf[ n]}$, as in Section \ref{subsec:domain}, for some $n$ sufficiently large. Also, assume that
\begin{align}\label{uniq.cond.b}
\|e-\widetilde{e}\|^2_{H^s(U_t)}+\|k-\widetilde{k}\|^2_{H^s(U_t)}+\|\gamma-\widetilde{\gamma}\|^2_{H^s(U_t)}\leq t^{2M_0},
\end{align}
for all $t\in (0,T]$, for some $s\ge4$, and an $M_0$ sufficiently large, how large depending only on the $L^\infty(U_0)$ norms of the $p_i$'s. Then the two solutions coincide.

$(ii)$ Let ${\bf g},\widetilde{\bf g}$ be two solutions to \eqref{EVE} of the form \eqref{metric}, satisfying 
\begin{align}\label{uniq.cond2}
\begin{split}
|\partial_x^\alpha(e_{Ia}-e^{[\bf 0]}_{Ia})|\leq C_\alpha t^{p_I-2\min\{p_I,p_a\}+\varepsilon},\qquad|\partial_x^\alpha(\widetilde{e}_{Ia}-e^{[\bf 0]}_{Ia})|\leq C_\alpha t^{p_I-2\min\{p_I,p_a\}+\varepsilon},\\
|\partial_x^\alpha(k_{IJ}-k^{[\bf 0]}_{IJ})|\leq C_\alpha t^{-1+|p_I-p_J|+\varepsilon},\qquad|\partial_x^\alpha(\widetilde{k}_{IJ}-k^{[\bf 0]}_{IJ})|\leq C_\alpha t^{-1+|p_I-p_J|+\varepsilon},
\end{split}
\end{align}
for all $(t,x)\in \{U_t\}_{t\in(0,T_{N_0,s,n}]}$, and $|\alpha|\leq M_1$, for some $M_1$ sufficiently large, how large depending only on $\varepsilon$. Then the two solutions coincide in $\{U_t\}_{t\in(0,T_{N_0,s,n}]}$.
\end{theorem}
%
%
Point $(i)$ in Theorem \ref{thm:uniq} is analogous to \cite[Theorem 1.7]{FL}. Point $(ii)$ does not require the two solutions to coincide to a large polynomial order. The idea here is that the latter property can be inferred from \eqref{uniq.cond2} by using the equations, see Section \ref{sec:smooth.uniq} for the proof of Theorem \ref{thm:uniq}.

Combining Theorems \ref{thm:exist}, \ref{thm:uniq} with the (local) equivalence \cite{Rin6} of Definitions \ref{defn1}, \ref{defn2}, we can obtain the existence of an entire Kasner-like singularity having covariant asymptotic data (see \cite{Fr} for a more general construction starting with covariant data).
\begin{corollary}\label{cor:exist.entire}
Let $(\mathbb{T}^3,\overline{h},\mathcal{K})$ be a triplet of smooth asymptotic data satisfying Definition \ref{defn2}. Then, there exists a solution ${\bf g}$ to \eqref{EVE}, defined in $(0,T]\times\mathbb{T}^3$, and a finite covering of local domains $\{U_t^\ell\}_{t\in(0,T]}$, such that in each one ${\bf g}$ satisfies Theorem \ref{thm:exist}, with gauge dependent asymptotic data $p_i,c_{ij}$ that are induced from $\overline{h},\mathcal{K}$.
\end{corollary}
\begin{proof}
For every $q\in\mathbb{T}^3$, there exist \cite{Rin6} local coordinates $x^1,x^2,x^3$ around $q$ and asymptotic data $p_i,c_{ij}$ satisfying Definition \ref{defn1}, which correspond to the given covariant data, ie. the $p_i$'s are eigenvalues of $\mathcal{K}$ and $h=\sum_{i,j=1}^3c_{ij}dx^idx^j$. Then, Theorem \ref{thm:exist} applies to yield a local patch of a Kasner-like singularity. By compactness, there exists $T>0$ sufficiently small and a finite number of local domains $\{U_t^\ell\}_{t\in(0,T]}$ covering $(0,T]\times\mathbb{T}^3$. In the overlapping regions the solutions coincide by Theorem \ref{thm:uniq} by taking $n,N_0$ sufficiently large to ensure that condition \eqref{uniq.cond.b} is satisfied.
\end{proof}

\subsection{Method of proof}

The basic idea in such constructions is simple. First, we construct an approximate solution ${\bf g}^{[\bf n]}$ to a suitable degree and then solve for a remainder ${\bf g}^{(d)}$ to upgrade it to an actual solution. What makes it hard to implement is the complicated nature of the Einstein vacuum equations. 

{\it Step 1. Constructing an approximate solution.} The evolution equations \eqref{e0.e}, \eqref{e0.k} satisfied by $e_{Ia},k_{IJ}$ (recall the notation from Section \ref{subsec:framework}) are useful for constructing an approximate solution, in a local domain of the form $(0,T]\times[0,\delta]^3$. We use them to define an iteration scheme in Section \ref{sec:approx.sol} that completely decouples all variables. The iterates $e^{[\bf n]}_{Ia},k_{IJ}^{[\bf n]}$ define a metric ${\bf g}^{[\bf n]}$ whose Ricci tensor we prove that it vanishes, as $t\rightarrow0$, to an increasing order in $n$, recall \eqref{exist.approx}. These derivations are analogous to those in \cite[Section 2]{FL}, constructing $g_{ij}^{[\bf n]},(k^{[\bf n]})_i{}^j$. One interesting difference concerning the second fundamental form is that $k_{IJ}^{[\bf n]}$ is asymptotically diagonal, instead of upper triangular for $(k^{[\bf n]})_i{}^j$ in \cite{FL}.
The asymptotic momentum constraint, expressed in terms of the frame coefficients, takes the form \eqref{frame.mom2}, which we show in Lemma \ref{lem:pre2} to be equivalent to \eqref{momentum}. It is used to show the approximate propagation of the constraints \eqref{Ham.const}-\eqref{mom.const}.

{\it Step 2. Solving for a remainder.} To construct an actual solution to the Einstein vacuum equations with the desired behavior, we make the ansatz 
\begin{align}\label{ansatz}
{\bf g}={\bf g}^{\bf [n]}+{\bf g}^{(d)}
\end{align}
and solve for the remainder term ${\bf g}^{(d)}$ in weighted norms which guarantee that the variables associated to ${\bf g}^{(d)}$ vanish, as $t\rightarrow0$, to a sufficiently high order. The main steps in this process are the following:
\begin{enumerate}
\item Express the Einstein vacuum equations relative to a first order symmetric hyperbolic system for $k_{IJ},\gamma_{IJB}$, see \eqref{e0.kIJ.mod}-\eqref{e0.gamma_IJB.mod}. Here, a modification of the original equations \eqref{e0.k}, \eqref{RIJ}, \eqref{gammaIJB.eq} is achieved following \cite{FSm}. 

\item Fix the local domain of definition of the solution, such that the overall boundary terms generated in an energy argument have a favorable sign (see Section \ref{subsec:domain} and the proof of Proposition \ref{prop:en.est}). 

\item Produce a solution to the modified evolution equations \eqref{e0.kIJ.mod}-\eqref{e0.gamma_IJB.mod} for $k_{IJ},\gamma_{IJB}$, by deriving weighted energy estimates for the remainder terms $k^{(d)}_{IJ},\gamma^{(d)}_{IJB}$ in a domain of the form $\{U_t\}_{t\in[\eta,T]}$, defined for convenience relative to ${\bf g}^{[\bf n]}$ (Section \ref{subsec:domain}), with trivial data. The use of large $t$ weights is necessary to absorb terms in the equations having $t^{-1}$ coefficients (see Sections \ref{subsec:rem.eq}, \ref{subsec:en.est}). Fortunately, the magnitude of these coefficients depends only on the asymptotic data, which allows for uniform estimates in $n,\eta$, provided $s,N_0$ are fixed (recall Theorem \ref{thm:exist}). Passing to the limit $\eta\rightarrow0$, we obtain a time interval of existence $(0,T]$, where $t=0$ corresponds to the singularity (Sections\ref{subsec:exist}).

\item Retrieve a solution to the full Einstein vacuum equations by propagating the vanishing of constraints off of the singularity (Section \ref{sec:EVE}). 
\end{enumerate}
The symmetrization of the first order system for $k_{IJ},\gamma_{IJB}$ and how to recover \eqref{EVE} are adopted from \cite{FSm}. This is the key difference from \cite{FL}, where instead a second order system for $k_i{}^j$ was used. Elliptic estimates are no longer required to complete the energy argument. Also, the weighted energy estimates are slightly simpler, because we define our norms $H^s$ norms relative to partial derivatives, not covariant ones, and all orders have the same $t$-weights. 

{\it Step 3. Uniqueness and smoothness.} If two solutions satisfy point $(i)$ of Theorem \ref{thm:uniq}, then by subtracting them and deriving an estimate similar to \eqref{exist.d.est} (with trivial initial data), it follows that they must be equal. Hence, to prove the second uniqueness statement in Theorem \ref{thm:uniq}, it remains to show that \eqref{uniq.cond2} implies that the two solutions satisfy \eqref{uniq.cond.b}. Condition \eqref{uniq.cond2} firstly implies that the two solutions have the same asymptotic data. Hence, they have the same approximate solution ${\bf g}^{\bf [n]}$. The idea now is to use the equations to iteratively improve the behavior of the remainder terms of each solution, see Section \ref{subsec:uniq}.
Every improvement by $t^\varepsilon$ costs two spatial derivatives, hence, the requirement of $M_1>0$ being sufficiently large in point $(ii)$ of Theorem \ref{thm:uniq}.

To obtain the smoothness of solutions (for smooth asymptotic data), we increase the number of derivatives $s$ in the existence norms. Some care is needed because increasing $s$ might also require to increase $n$, which in turn changes ${\bf g}^{[\bf n]}$, relative to which the local domain is defined. We overcome this issue by defining the domain relative to the minimum of $n$'s, which does not affect the overall estimates (see Section \ref{subsec:smooth}). Employing point $(i)$ of Theorem \ref{thm:uniq} and standard unique continuation criteria yields the corresponding increased regularity of the original solution. Since $s$ can be taken arbitrarily large, the smoothness property follows.

\subsection{Acknowledgements}

G.F. would like to thank Jonathan Luk and Hans Ringstr\"om for useful discussions. 
N.A. and G.F.
gratefully acknowledge the support of the ERC starting grant 101078061 SINGinGR, under the European Union's Horizon Europe program for research and innovation, and the H.F.R.I. grant 7126, under the 3rd call for H.F.R.I. research projects to support post-doctoral researchers.

\section{Local existence for the asymptotic constraint equations}\label{sec:asym}

Definitions \ref{defn1}, \ref{defn2} give conditions that characterize the asymptotic data of Kasner-like singularities. However, in both works \cite{FL,Rin} there is no mention as to whether such general asymptotic data exist. In this section, we present a simple argument that gives existence of localized asymptotic data sets. Also, from our argument one can infer the freedom one has in choosing such initial data sets on the singularity. Interestingly, constructing a global initial data set, in the sense of Definition \ref{defn2} or the analogue of Definition \ref{defn1} on the entire $\mathbb{T}^3$, is more intricate than it seems and we shall not discuss it here. 
%
%
%
\begin{proposition}\label{prop:asym}
Let $p_1,p_2,p_3:[0,\delta]^3\to \mathbb{R}$ be smooth functions satisfying $p_1(x)<p_2(x)<p_3(x)$, as well as $\sum_{i=1}^3 p_i(x) = \sum_{i=1}^3 p_i^2(x) =1$, for all $x\in[0,\delta]^3$. Then, for any freely prescribed smooth functions $\alpha \in \{c_{11},c_{22}>0\}$, $\beta \in \{c_{33}>0,\kappa_2{}^3\}$, $\gamma \in \{\kappa_1{}^2, \kappa_1{}^3\}:$ $[0,\delta]^3\to \mathbb{R}$, there exist smooth functions $\epsilon, \zeta,\eta: [0,\delta]^3\to \mathbb{R}$ such that 
\begin{align}
\begin{Bmatrix} \alpha, \beta,\gamma,\epsilon,\zeta,\eta \end{Bmatrix} = \begin{Bmatrix} c_{11},  c_{22}, c_{33}, \kappa_1{}^2,  \kappa_1{}^3,  \kappa_2{}^3 \end{Bmatrix} 
\end{align}
and such that Definition \ref{defn1} holds, where $c_{12},c_{13},c_{23}$ are uniquely determined by the values of $\kappa_1{}^2,\kappa_1{}^3,\kappa_2{}^3$. Moreover, the functions $\epsilon,\zeta,\eta$ are unique up to a choice of three 2-variable functions.
\end{proposition}
\begin{remark}
From the statement of the previous proposition, it would seem that we are free to prescribe four functions, e.g. $p_1,c_{11},c_{33}$, and $\kappa_1{}^2$ (which amounts to choosing $c_{12}$). However, $c_{33}$ can be fixed to 1 by choosing the coordinate function $x^3$ appropriately to begin with. Hence, the functional degrees of freedom are three. Also, the three 2-variable functions that we are free to prescribe come from suitably integrating \eqref{momentum}, for each $i=1,2,3$.
\end{remark}
\begin{proof}
We begin by writing out \eqref{momentum} for $i=3$:
\begin{align} \label{one}
 (p_3-p_1) \partial_3 \log(c_{11}) + (p_3-p_2) \partial_3 \log(c_{22})+2\partial_3p_3 =0,
\end{align}
The latter equation can be solved for either $c_{11}$ or $c_{22}$. Without loss of generality, let us freely prescribe $c_{22}>0$. Then $c_{11}$ is uniquely determined by \eqref{one}, up to a choice of a 2-variable function:
\begin{align}\label{one.sol}
\log c_{11}(x^1,x^2,x^3)=\log c_{11}(x^1,x^2,0)-\int^{x^3}_0\big\{\frac{p_3-p_2}{p_3-p_1}\partial_3\log c_{22}+\frac{2\partial_3p_3}{p_3-p_1} \big\}(x^1,x^2,s)ds
\end{align}
Next, we expand \eqref{momentum} for $i=2$:
\begin{align}\label{two}
(p_2-p_1)\partial_2 \log c_{11} + (p_2-p_3)\partial_2 \log c_{33} + 2 \partial_2 p_2 -2 \partial_3 \kappa_2{}^3 -\kappa_2{}^3 \sum_{l=1}^3 \partial_3 \log c_{ll}=0.
\end{align}
Since $c_{11},c_{22}$ are already fixed, we can either freely prescribe $c_{33}>0$ and solve for $\kappa_2{}^3$ or vice versa. In the latter case, solving for $\log c_{33}$ amounts to integrating a transport equation in the direction $(p_2-p_3)\partial_2-\kappa_2{}^3\partial_3$, which is possible thanks to the condition $p_2<p_3$. For convenience, we freely prescribe $c_{33}>0$ and rewrite \eqref{two} as follows:
\begin{align}\label{two2}
\partial_3 \kappa_2{}^3 +\frac{1}{2}[\partial_3\log(c_{11}c_{22}c_{33})]\kappa_2{}^3=\frac{1}{2}(p_2-p_1)\partial_2 \log c_{11} + \frac{1}{2}(p_2-p_3)\partial_2 \log c_{33} + \partial_2 p_2
\end{align}
Hence, we easily solve for $\kappa_2{}^3$ via integrating factors, in a unique manner up to a choice of a 2-variable function. 

Lastly, we expand \eqref{momentum} for $i=1$: 
\begin{align}\label{three}
\begin{split}
&(p_1-p_2)\partial_1\log c_{22}+(p_1-p_3)\partial_1\log c_{33} + 2 \partial_1 p_1 -2  \partial_2 \kappa_1{}^2 \\&- 2 \partial_3 \kappa_1{}^3- \kappa_1{}^2 \sum_{l=1}^3 \partial_2 \log c_{ll} -\kappa_1{}^3 \sum_{l=1}^3 \partial_3 \log c_{ll}=0.
\end{split}
\end{align}
It is clear that we can freely prescribe either $\kappa_1{}^2$ or $\kappa_1{}^3$ and solve for the other via integrating factors, using \eqref{three} as above. Once more, we have the freedom of choosing an initial condition, which amounts to a 2-variable function. This completes the proof of the proposition. 
\end{proof}

\section{The Einstein vacuum equations as a symmetric first order hyperbolic ADM-type system}\label{sec:ADM}

We have defined the frame and connection coefficients $e_{Ia},k_{IJ},\gamma_{IJB}$ in Section \ref{subsec:framework}.
Define also the inverse transformation 
\begin{align}\label{coframe}
\partial_b=\omega_{bC} e_C,\qquad \omega_{bC}e_{Ca}=\delta_{ba},\qquad e_{Ia}\omega_{aC}=\delta_{IC},
\end{align}
where repeated indices are summed, unless underlined or otherwise stated.
Then the following equations hold \cite{FSm}:\footnote{Note: In \cite{FSm}
the sign convention for $k_{IJ}$ is opposite than the one that we use.}
\begin{align}
\label{e0.e}\partial_te_{Ia}=&\,k_{IC}e_{Ca}\\
\label{e0.omega}\partial_t\omega_{bC}=&-k_{CD}\omega_{bD}\\
\label{e0.k}\partial_tk_{IJ}-\text{tr}k \, k_{IJ}=&\,R_{IJ}-{\bf R}_{IJ},
\end{align}
where $R_{IJ},{\bf R}_{IJ}$ are the Ricci curvatures contracted with $e_I,e_J$ of $g,{\bf g}$ respectively. 
These are analogous to the ADM equations for the first and second fundamental forms of $\Sigma_t$. The actual solution that we construct is in vacuum, hence, the last term in \eqref{e0.k} can be set to zero. We keep it in the RHS, however, to keep track of the extra terms coming from the approximate solution that is constructed in Section \ref{sec:approx.sol}.

On the other hand, the spatial Ricci can be expressed in terms of $\gamma_{IJB},e_I^a,\omega_b^C$ as follows:
\begin{align}\label{RIJ}
R_{IJ}=&\,e_C \gamma_{IJC}-e_I\gamma_{CJC}
-\gamma_{CID}\gamma_{DJC}-\gamma_{IJD}\gamma_{CCD},\\
\label{gammaIJB}\gamma_{IJB}=&\,\frac{1}{2}\big\{\omega_{aB}(e_Ie_{Ja}-e_Je_{Ia})
-\omega_{aI}(e_Je_{Ba}-e_Be_{Ja})
+\omega_{aJ}(e_Be_{Ia}-e_Ie_{Ba})\big\}.
\end{align}
We will use \eqref{e0.e}-\eqref{gammaIJB} in Section \ref{sec:approx.sol} to construct an approximate solution to the Einstein vacuum equations, as $t\rightarrow0$, via an iteration scheme. 

We will not exploit the formula \eqref{gammaIJB} for local well-posedness. Instead, we couple \eqref{e0.k} to the following equation satisfied by $\gamma_{IJB}$:
\begin{align}\label{gammaIJB.eq}
\partial_t\gamma_{IJB}-k_{IC}\gamma_{CJB}=&\,e_Bk_{JI}-e_Jk_{BI}\\
\notag&+\gamma_{JBC}k_{CI}+\gamma_{JIC}k_{BC}-\gamma_{BJC}k_{CI}-\gamma_{BIC}k_{JC}
\end{align}
To derive energy estimates in Section \ref{subsec:en.est}, we consider the symmetrized system satisfied by $k_{IJ},\gamma_{IJB}$, following \cite{FSm}:
\begin{align}
\label{e0.kIJ.mod}\partial_tk_{IJ}-\text{tr}k\,k_{IJ}=&\,\frac{1}{2}\bigg[e_C\gamma_{IJC}-e_I\gamma_{CJC}
+e_C\gamma_{JIC}-e_J\gamma_{CIC}\\
\notag&-\gamma_{CID}\gamma_{DJC}-\gamma_{CCD}\gamma_{IJD}
-\gamma_{CJD}\gamma_{DIC}-\gamma_{CCD}\gamma_{JID}\bigg]\\
\notag&+\frac{1}{2}\delta_{IJ}\bigg[2e_D\gamma_{CDC}+\gamma_{CED}\gamma_{DEC}
+\gamma_{CCD}\gamma_{EED}+k_{CD}k_{CD}-(k_{CC})^2\bigg]\\
\notag&-\frac{1}{2}({\bf R}_{IJ}+{\bf R}_{JI})
+\delta_{IJ}({\bf R}_{00}+\frac{1}{2}{\bf R}),\\
\label{e0.gamma_IJB.mod} \partial_t\gamma_{IJB}-k_{IC}\gamma_{CJB}
=&\,e_Bk_{JI}-e_Jk_{BI}\\
\notag&+\gamma_{JBC}k_{CI}+\gamma_{JIC}k_{BC}-\gamma_{BJC}k_{CI}-\gamma_{BIC}k_{JC}\\
\notag&-\delta_{IB}\bigg[e_Ck_{CJ}-\gamma_{CCD}k_{DJ}-\gamma_{CJD}k_{CD}-e_J\text{tr}k\bigg]\\
\notag&+\delta_{IJ}\bigg[e_Ck_{CB}-\gamma_{CCD}k_{DB}-\gamma_{CBD}k_{CD}-e_B\text{tr}k\bigg]\\
\notag&-\delta_{IB}{\bf R}_{0J}+\delta_{IJ}{\bf R}_{0B},
\end{align}
coupled to equation \eqref{e0.e}. 
\begin{remark}
The previous equations are valid for any metric of the form \eqref{metric} and its associated connection coefficients \eqref{k.gamma}. However, we will drop all ${\bf R}_{\mu\nu}$ terms in the derivations of the energy estimates below, since we are interested in producing a vacuum solution.
\end{remark}
%
%

\section{An approximate solution to the Einstein vacuum equations}\label{sec:approx.sol}

We define the following iteration scheme for the equations \eqref{e0.e}, \eqref{e0.k}:
\begin{align}
\label{eIi.it}\partial_t \eIan - k_{\underline{I} \underline{I}}^{[\bf{n}]}e_{\underline{I}a}^{[\bf{n}]} = &\sum_{C\neq I} k_{IC}^{[\bf{n-1}]}e_{Ca}^{[\bf{n-1}]},\\
\label{kIJ.it}\partial_tk_{IJ}^{\bf [n]}-k^{[\bf n-1]}_{CC}k_{IJ}^{[\bf n]}=&\,R^{[\bf n-1]}_{IJ},
\end{align}
for $n\ge1$ and $t\in(0,T]$, where underlined indices are not summed. Here, $R_{IJ}^{[\bf n-1]}$ is the Ricci curvature of the metric $g^{[\bf n-1]}$ for which the frame $e_I^{\bf [n-1]}=e_{Ia}^{[\bf n-1]}\partial_a$ is orthonormal. Similarly, we denote by ${\bf R}_{\mu\nu}^{[\bf n]}$ the Ricci curvature of the spacetime metric 
\begin{align}\label{bf.gn}
{\bf g}^{[\bf n]}=-dt^2+g^{[\bf n]}.
\end{align}
We set the zeroth iterates equal to: 
\begin{align}
\label{eIi.kIJ.0}e_{Ia}^{[\bf 0]}=
\left\{\begin{array}{ll}
  f_{\underline{I}a}(x) t^{-p_{\underline{I}}(x)}  & I\leq a \\
    0 & I>a
\end{array}\right.,\qquad\qquad
k_{IJ}^{[\bf 0]}
=-\delta_{\underline{I}J}\frac{p_{\underline{I}}(x)}{t},
\end{align}
To complete the iteration scheme, we require that the following asymptotic initial conditions hold: 
\begin{align}\label{asym.cond.it}
\lim_{t\rightarrow0^+} t^{p_{\underline{I}}(x)}e^{[\bf n]}_{\underline{I}a}=\left\{\begin{array}{ll}
  f_{Ia}(x)  & I\leq a \\
    0 & I>a
\end{array}\right.,
\qquad\qquad\lim_{t\rightarrow0^+}tk_{IJ}^{[\bf n]}=-\delta_{\underline{I}J}p_{\underline{I}}(x).
\end{align}
%
Also, define the inverse frame components through the relations 
\begin{align}\label{omega.it}
\partial_b=\omega^{\bf [n]}_{bC}e_C^{[\bf n]},\qquad
\omega_{bC}^{[\bf n]}e_{Ca}^{[\bf n]}=\delta_{ba},\qquad 
e_{Ia}^{[\bf n]}\omega_{aC}^{[\bf n]}=\delta_{IC}.
\end{align}
The equation \eqref{eIi.it} and relations \eqref{omega.it} imply the equation
\begin{align}\label{dt.omegabC.it}
   \partial_t\omega^{[\bf n]}_{bC}+k^{[\bf n]}_{\underline{C}\hsp\underline{C}}\omega^{[\bf n]}_{b\underline{C}}=- \sum_{D}\sum_{E\neq D} \omega_{bD}^{[\bf n]}  k_{DE}^{[\bf{n-1}]} \big(e_{Ea}^{[\bf{n-1}]} \hsp\omega_{aC}^{[\bf n]}\big),
\end{align}
for $n\ge1$. The zeroth iterates $\omega^{[\bf 0]}_{bC}$ are computed using \eqref{omega.it}, see \eqref{omega.0}. The condition \eqref{asym.cond.it} implies that  there exist functions $h_{bC}(x)$ such that
\begin{align}\label{omega.asym}
\lim_{t\rightarrow0^+}t^{-p_{\underline C}(x)}\omega^{[\bf n]}_{b\underline C}=\left\{\begin{array}{ll}
h_{bC}(x)& b\leq C\\
0&  b>C
\end{array}\right.,\qquad \omega^{[\bf 0]}_{bC}=\left\{\begin{array}{ll}
h_{b\underline{C}}(x)t^{p_{\underline{C}}}(x)& b\leq C\\
0&  b>C
\end{array}\right..
\end{align}
%
%
\begin{remark}\label{Asymptotic differential condition}
The frame coefficients $f_{Ia}$ are in one to one correspondence with the metric coefficients $c_{ij}$ in Definition \ref{defn1}, see Lemma \ref{lem:pre1}. Moreover, the asymptotic differential condition \eqref{momentum} is equivalent to 
\begin{align}\label{frame.mom}
E_{\underline{I}} p_{\underline{I}}+\sum_J(p_J-p_{\underline{I}})E_{\underline{I}}\log (f_{JJ})-\sum_J\sum_{I\leq a\leq J}(p_J-p_{\underline{I}})h_{aJ}E_Jf_{\underline{I}a}=0,
\end{align}
for $I=1,2,3$, where $E_I=\sum_{a\ge I} f_{Ia}\partial_a$, see Lemma \ref{lem:frame.mom}. By convention, the last sum in \eqref{frame.mom} does not appear when $I>J$.
\end{remark}
\begin{theorem}\label{thm:approx.sol}
Let $p_i(x),f_{Ia}(x)\in C^\infty([0,\delta]^3)$ and let 
\begin{align}\label{varepsilon}
\varepsilon:=\min_x\{1-p_3(x),\,p_3(x)-p_2(x) \}>0.
\end{align}
Then for every $n\ge1$ there exists a unique solution to the iteration scheme \eqref{eIi.it}-\eqref{asym.cond.it} in $(0,t_n]\times[0,\delta]^3$, for some $t_n=t_n(p_i,f_{Ia})$ sufficiently small, 
such that the following points hold: 
\begin{enumerate}
    \item \label{item.approx1} For every multi-index $\alpha$, every $I,J,C,a,b$, the functions $e^{[\bf n]}_{Ia},\omega^{[\bf n]}_{bC},k^{[\bf n]}_{IJ}$ satisfy:
    \begin{align}\label{eIa.n-0.est}
|\partial^\alpha_x( e^{[\bf n]}_{ Ia}- e^{\bf [0]}_{ Ia})|\leq &\left\{\begin{array}{ll}
 C_{\alpha,n}t^{-p_I+\varepsilon}  & I\leq a, \\
 C_{\alpha,n}t^{p_I-2p_a+\varepsilon}  & I>a,
 \end{array}\right.\\
  \label{omegabC.n-0.est}
|\partial^\alpha_x(
\omega^{[\bf n]}_{bC}- \omega^{\bf [0]}_{bC})|\leq &\left\{\begin{array}{ll}
 C_{\alpha,n}t^{p_C+\varepsilon}  & b\leq C, \\
    C_{\alpha,n}t^{2p_b-p_C+\varepsilon} & b>C,
\end{array}\right.\\
        \label{kIJ.n-1.est}
      |\partial^\alpha_x(k_{IJ}^{[\bf n]}-k_{IJ}^{[\bf 0]})|\leq&\, 
    C_{\alpha,n}t^{-1+\varepsilon+|p_I-p_J|},
    \end{align}
    for all $(t,x)\in(0,t_n]\times[0,\delta]^3$.
    \item  \label{item.approx2} For every multi-index $\alpha$, every $I,J$, the spatial Ricci curvature satisfies:
    \begin{align}\label{Ric.n.est}
       |\partial^\alpha_x R_{IJ}^{\bf [n]}|\leq 
    C_{\alpha,n}t^{-2+\varepsilon+|p_I-p_J|},
    \end{align}
     for all $(t,x)\in(0,t_n]\times[0,\delta]^3$.
    \item  \label{item.approx3} If in addition the constraint \eqref{frame.mom} is satisfied, then for every multi-index $\alpha$, every $\mu,\nu=0,I,J$, the spacetime Ricci curvature satisfies:  
    \begin{align}\label{Ric4.n.est}
|\partial^\alpha_x {\bf R}_{\mu\nu}^{{\bf [n]}}|\leq 
 C_{\alpha,n}t^{-2+n\varepsilon},
 \end{align}
  for all $(t,x)\in(0,t_n]\times[0,\delta]^3$.
    \end{enumerate}
\end{theorem}
\begin{proof}
The points \ref{item.approx1}, \ref{item.approx2} are contained in Proposition \ref{prop:ind1}. Point \ref{item.approx3} is contained in Proposition \ref{prop:approx4}, for spatial indices $\mu,\nu=I,J$, and for the rest of the indices $\mu=0,\nu=0,I$ in Proposition \ref{prop:approx5}.
\end{proof}
\subsection{The zeroth iterates and the asymptotic differential constraint}

The zeroth frame coefficients determine the zeroth co-frame and metric components.
\begin{lemma}\label{lem:pre1}
Let $e_{Ia}^{[\bf 0]}$ be defined as in \eqref{eIi.kIJ.0} and let $\omega^{[\bf 0]}_{bC}$ represent the components of the inverse frame transformation, as in \eqref{omega.it}. Also, let $g_{ij}^{[\bf 0]}$ be the metric for which the frame $e^{[\bf 0]}_I=e^{[\bf 0]}_{Ia}\partial_a$ is orthonormal. Then, the following formulas are valid:
\begin{align}\label{omega.0}
(\omega^{[\bf 0]}_{bC})=
\begin{bmatrix}
f_{11}^{-1}t^{p_1}& -f_{11}^{-1}f_{22}^{-1}f_{12}t^{p_2}& f_{11}^{-1}f_{33}^{-1}\big[f_{12}f_{23}f_{22}^{-1}-f_{13}\big]t^{p_3}\\
0 & f_{22}^{-1}t^{p_2} & -f_{22}^{-1}f_{33}^{-1}f_{23} t^{p_3}\\
0 & 0 & f_{33}^{-1} t^{p_3}
\end{bmatrix}
\end{align}
and 
\begin{align}\label{gij.0}
 g^{[\bf 0]}=\begin{bmatrix}
c_{11}t^{2p_1} & c_{12}t^{2p_2} & c_{13}t^{2p_3}\\
c_{12}t^{2p_2}&c_{22}t^{2p_2} & c_{23}t^{2p_3}\\
c_{13}t^{2p_3}& c_{23}t^{2p_3} & c_{33}t^{2p_3}
\end{bmatrix}+g^{[\bf 0]}_{error},
\end{align}
where $|\partial_x^\alpha (g^{[\bf 0]}_{error})_{ij}|\leq C_\alpha t^{2p_{\max\{i,j\}}+\varepsilon}$, and
\begin{align}\label{cij}
(c_{ij})=\begin{bmatrix}
f_{11}^{-2}&-f_{12}f_{11}^{-1}f_{22}^{-2} &f_{11}^{-1}f_{33}^{-2}\big[f_{12}f_{23}f_{22}^{-1}-f_{13}\big] \\
-f_{12}f_{11}^{-1}f_{22}^{-2}&f_{22}^{-2} &-f_{23}f_{22}^{-1}f_{33}^{-2} \\
f_{11}^{-1}f_{33}^{-2}\big[f_{12}f_{23}f_{22}^{-1}-f_{13}\big]&-f_{23}f_{22}^{-1}f_{33}^{-2} & f_{33}^{-2}
\end{bmatrix}
\end{align}
\end{lemma}
\begin{proof}
First, we expand the relations $e_{Ia}^{[\bf 0]}\omega^{[\bf 0]}_{aC}=\delta_{IC}$. For $I=3$:
\begin{align*}
 &e_{33}^{[\bf 0]}\omega^{[\bf 0]}_{33}=1,\qquad
e_{33}^{[\bf 0]}\omega^{[\bf 0]}_{32}=0,\qquad
e_{33}^{[\bf 0]}\omega^{[\bf 0]}_{31}=0\\
\Rightarrow&\qquad \omega^{[\bf 0]}_{33}=f_{33}^{-1}t^{p_3},\quad \omega^{[\bf 0]}_{32}=\omega^{[\bf 0]}_{31}=0,
\end{align*}
for $I=2$:
\begin{align*}
& e_{22}^{[\bf 0]}\omega^{[\bf 0]}_{21}=0,\qquad
e_{22}^{[\bf 0]}\omega^{[\bf 0]}_{22}=1,\qquad
e_{22}^{[\bf 0]}\omega^{[\bf 0]}_{23}+e_{23}^{[\bf 0]}\omega^{[\bf 0]}_{33}=0,\\
\Rightarrow&\qquad \omega^{[\bf 0]}_{21}=0,\quad
\omega^{[\bf 0]}_{22}=f_{22}^{-1}t^{p_2},\quad \omega_{23}^{[\bf 0]}=-f_{23}f_{22}^{-1}f_{33}^{-1}t^{p_3}
\end{align*}
and for $I=1$:
\begin{align*}
&e_{11}^{[\bf 0]}\omega^{[\bf 0]}_{11}=1,\qquad
e_{11}^{[\bf 0]}\omega^{[\bf 0]}_{12}+e_{12}^{[\bf 0]}\omega^{[\bf 0]}_{22}=0,\qquad
e_{11}^{[\bf 0]}\omega^{[\bf 0]}_{13}+e_{12}^{[\bf 0]}\omega^{[\bf 0]}_{23}+e_{13}^{[\bf 0]}\omega^{[\bf 0]}_{33}=0\\
\Rightarrow\qquad&\omega^{[\bf 0]}_{11}=f_{11}^{-1}t^{p_1},\quad \omega_{12}^{[\bf 0]}=-f_{11}^{-1}f_{22}^{-1}f_{12}t^{p_2},\\
&\omega_{13}^{[\bf 0]}=f_{11}^{-1}f_{33}^{-1}\big[f_{12}f_{23}f_{22}^{-1}-f_{13}\big]t^{p_3},
\end{align*}
which combined give \eqref{omega.0}. 

Hence, the metric components equal $g_{ij}^{[\bf 0]}=\omega_{iC}^{[\bf 0]}\omega_{jC}^{[\bf 0]}$, which expands to
\begin{align*}
g_{11}^{[\bf 0]}=&\,\omega_{11}^{[\bf 0]}\omega_{11}^{[\bf 0]}+(g^{[\bf 0]}_{error})_{11}
=f_{11}^{-2}t^{2p_1}+(g^{[\bf 0]}_{error})_{11},\\
g_{12}^{[\bf 0]}=&\,\omega_{12}^{[\bf 0]}\omega_{22}^{[\bf 0]}+(g^{[\bf 0]}_{error})_{12}
=-f_{11}^{-1}f_{12}f_{22}^{-2}t^{2p_2}+(g^{[\bf 0]}_{error})_{12},\\ 
g_{13}^{[\bf 0]}=&\,\omega_{13}^{[\bf 0]}\omega_{33}^{[\bf 0]}=f_{11}^{-1}f_{33}^{-2}\big[f_{12}f_{23}f_{22}^{-1}-f_{13}\big]t^{2p_3},\\
g_{22}^{[\bf 0]}=&\,\omega_{22}^{[\bf 0]}\omega_{22}^{[\bf 0]}+(g^{[\bf 0]}_{error})_{22}
=f_{22}^{-2}t^{2p_2}+(g^{[\bf 0]}_{error})_{22},\\ 
g_{23}^{[\bf 0]}=&\,\omega_{23}^{[\bf 0]}\omega_{33}^{[\bf 0]}=-f_{23}f_{22}^{-1}f_{33}^{-2}t^{2p_3},\\ 
g_{33}^{[\bf 0]}=&\,\omega_{33}^{[\bf 0]}\omega_{33}^{[\bf 0]}=f_{33}^{-2}t^{2p_3}
\end{align*}
The leading order metric coefficients $c_{ij}$ can be read from the previous formulas, giving \eqref{cij}, while the error terms satisfy
\begin{align*}
|\partial_x^\alpha (g^{\bf [0]}_{error})_{11}|\leq &\,C_\alpha \big[t^{2p_2-2p_1} |\log t|^{|\alpha|}\big]t^{2p_1}\leq C_\alpha t^{2p_1+2\varepsilon}|\log t|^\alpha,\\ 
|\partial_x^\alpha (g^{\bf [0]}_{error})_{12}|\leq &\,C_\alpha \big[t^{2p_3-2p_2} |\log t|^{|\alpha|}\big]t^{2p_2}\leq C_\alpha t^{2p_2+2\varepsilon}|\log t|^\alpha,\\
|\partial_x^\alpha (g^{\bf [0]}_{error})_{22}|\leq &\, C_\alpha \big[t^{2p_3-2p_2} |\log t|^{|\alpha|}\big]t^{2p_2}\leq C_\alpha t^{2p_2+2\varepsilon}|\log t|^\alpha,
\end{align*}
which implies the claimed estimate.
\end{proof}
Using the relations in the previous lemma, we can now phrase the asymptotic differential constraint in terms of the frame and co-frame coefficients, which will be used below to approximately propagate the constraints, see Lemma \ref{lem:constraint}.
\begin{lemma}\label{lem:frame.mom}
The asymptotic differential condition \eqref{momentum} is equivalent to the following differential set of equations:
\begin{align}\label{frame.mom2}
E_{\underline{I}} p_{\underline{I}}+\sum_J(p_J-p_{\underline{I}})E_{\underline{I}}\log (f_{JJ})-\sum_J\sum_{I\leq a\leq J}(p_J-p_{\underline{I}})h_{aJ}E_Jf_{\underline{I}a}=0,
\end{align}
for $I=1,2,3$, where $E_I=\sum_{a\ge I}f_{Ia}\partial_a$ and the last sum is zero by convention for $I>J$.
\end{lemma}
\begin{proof}
It is a straightforward, but tedious, computation.
For $I=3$, \eqref{frame.mom2} becomes 
\begin{align}
\notag 0=&\,E_3 p_3 +(p_1-p_3)E_3\log f_{11}+(p_2-p_3)E_3\log f_{22}\\
\label{equivalenceE.1} =&\, f_{33} \big[\partial_3 p_3 +(p_1-p_3)\partial_3\log f_{11}+(p_2-p_3)\partial_3\log f_{22}\big]\\
\notag=&\, f_{33}\big[\partial_3p_3 -\frac{1}{2}(p_1-p_3)\partial_3 \log c_{11}-\frac{1}{2}(p_2-p_3)\partial_3 \log c_{22}\big], 
\end{align}
since $c_{\underline{J}\hsp \underline{J}}=f_{\underline{J}\hsp \underline{J}}^{-2}$, for $J=1,2,3$. Given that $f_{33}$ is nowhere zero, \eqref{equivalenceE.1} is equivalent to \eqref{momentum} for $i=3$. 

Next, we expand \eqref{frame.mom2} for $I=2$:
\begin{align}
\begin{split}
&(f_{22}\partial_2+f_{23}\partial_3)p_2+(p_1-p_2)(f_{22}\partial_2+f_{23}\partial_3)\log f_{11}\\
&+(p_3-p_2)(f_{22}\partial_2+f_{23}\partial_3)\log f_{33}
-\sum_{a=2,3}(p_3-p_2)h_{a3}f_{33}\partial_3f_{2a}=0,
\label{frame.mom.I=2}
\end{split}
\end{align}
where from \eqref{omega.0}, \eqref{cij}, it follows that $f_{23}=-(c_{23}/c_{33})(c_{22})^{-\frac{1}{2}}$ and $h_{23}=-f_{22}^{-1}f_{33}^{-1}f_{23}$, $h_{33}=f_{33}^{-1}$. Multiplying \eqref{frame.mom.I=2} with $\sqrt{c_{22}}=f_{22}^{-1}$ and plugging in the formulas for $f_{23},f_{\underline{J}\underline{J}},h_{23},h_{33}$ in terms of $c_{ij}$, we compute: 
\begin{align}\label{equivalenceE.2}
\notag0=&\,(\partial_2-\frac{c_{23}}{c_{33}}\partial_3)p_2-\frac{1}{2}(p_1-p_2)(\partial_2-\frac{c_{23}}{c_{33}}\partial_3)\log c_{11}-\frac{1}{2}(p_3-p_2)(\partial_2-\frac{c_{23}}{c_{33}}\partial_3)\log c_{33}\\
\notag&+(p_3-p_2)\sqrt{c_{22}}f_{23}\partial_3\log f_{22}
-(p_3-p_2)\sqrt{c_{22}}\partial_3f_{23}\\
\notag=&\,\partial_2p_2-\frac{c_{23}}{c_{33}}\big[\partial_3p_2-\frac{1}{2}(p_1-p_2)\partial_3\log c_{11}-\frac{1}{2}(p_3-p_2)\partial_3\log c_{33}\big]\\
\notag& -\frac{1}{2}(p_1-p_2)\partial_2\log c_{11}-\frac{1}{2}(p_3-p_2)\partial_2\log c_{33}\\
&+\frac{1}{2}(p_3-p_2)\frac{c_{23}}{c_{33}} \partial_3\log c_{22}
+(p_3-p_2)\partial_3(\frac{c_{23}}{c_{33}})
-\frac{1}{2}(p_3-p_2)\frac{c_{23}}{c_{33}}\partial_3\log c_{22}\\
\tag{rearranging terms}=&\,\partial_2p_2
 -\frac{1}{2}(p_1-p_2)\partial_2\log c_{11}-\frac{1}{2}(p_3-p_2)\partial_2\log c_{33}\\
 \notag&-\frac{c_{23}}{c_{33}}\big[\partial_3p_2-\frac{1}{2}(p_1-p_3)\partial_3 \log c_{11}-\frac{1}{2}(p_2-p_3)\partial_3 \log c_{22}\big]+(p_3-p_2)\partial_3(\frac{c_{23}}{c_{33}})\\
\notag&-\frac{1}{2}(p_2-p_3)\frac{c_{23}}{c_{33}}\partial_3\log(c_{11}c_{22}c_{33})\\
\notag=&\,\partial_2p_2
 -\frac{1}{2}(p_1-p_2)\partial_2\log c_{11}-\frac{1}{2}(p_3-p_2)\partial_2\log c_{33}
-\frac{1}{2}(p_2-p_3)\frac{c_{23}}{c_{33}}\partial_3\log(c_{11}c_{22}c_{33})\\
\tag{by \eqref{equivalenceE.1}}&+\frac{c_{23}}{c_{33}}\partial_3(p_3-p_2)
+(p_3-p_2)\partial_3(\frac{c_{23}}{c_{33}})
\end{align}
which is indeed equivalent to \eqref{momentum} for $i=2$, after taking into account that ${\kappa_2}^3=(p_2-p_3)\frac{c_{23}}{c_{33}}$.

Finally, for $I=1$, we expand the first two terms in \eqref{frame.mom2} and use in addition the relations from \eqref{cij}, 
$f_{12}f_{11}^{-1}=-\frac{c_{12}}{c_{22}}$, $f_{13}f_{11}^{-1}=-\frac{c_{13}}{c_{33}}+\frac{c_{12}}{c_{22}}\frac{c_{23}}{c_{33}}$:
\begin{align}
\notag&E_1p_1+(p_2-p_1)E_1\log f_{22}+(p_3-p_1)E_1\log f_{33}\\
\notag=&\,(f_{11}\partial_1+f_{12}\partial_2+f_{13}\partial_3)p_1
+(p_2-p_1)(f_{11}\partial_1+f_{12}\partial_2+f_{13}\partial_3)\log f_{22}\\
\notag&+(p_3-p_1)(f_{11}\partial_1+f_{12}\partial_2+f_{13}\partial_3)\log f_{33}\\
\tag{by \eqref{cij}}=&\,f_{11}\bigg[
\partial_1p_1-\frac{1}{2}(p_2-p_1)\partial_1\log c_{22}-\frac{1}{2}(p_3-p_1)\partial_1\log c_{33}\\
\label{frame.mom.I=1}&-\frac{c_{12}}{c_{22}}\partial_2p_1+\frac{1}{2}(p_2-p_1)\frac{c_{12}}{c_{22}}\partial_2\log c_{22}+\frac{1}{2}(p_3-p_1)\frac{c_{12}}{c_{22}}\partial_2\log c_{33}\\
\notag&+(\frac{c_{12}}{c_{22}}\frac{c_{23}}{c_{33}}-\frac{c_{13}}{c_{33}})\partial_3p_1
-\frac{1}{2}(p_2-p_1)(\frac{c_{12}}{c_{22}}\frac{c_{23}}{c_{33}}-\frac{c_{13}}{c_{33}})\partial_3\log c_{22}\\
\notag&-\frac{1}{2}(p_3-p_1)(\frac{c_{12}}{c_{22}}\frac{c_{23}}{c_{33}}-\frac{c_{13}}{c_{33}})\partial_3\log c_{33}\bigg]
\end{align}
To compute the last sum in \eqref{frame.mom2}, for $I=1$, we use the relations from \eqref{omega.0}, \eqref{cij}:
\begin{align*}
h_{12}E_2=&-f_{11}^{-1}f_{22}^{-1}f_{12}(f_{22}\partial_2+f_{23}\partial_3)=\frac{c_{12}}{c_{22}}(\partial_2-\frac{c_{23}}{c_{33}}\partial_3)\\
h_{22}E_2=&\,\partial_2-\frac{c_{23}}{c_{33}}\partial_3\\
h_{13}E_3=&\,h_{13}f_{33}\partial_3=f_{11}^{-1}(f_{12}f_{23}f_{22}^{-1}-f_{13})\partial_3=\frac{c_{13}}{c_{33}}\partial_3\\
h_{23}E_3=&-f_{22}^{-1}f_{23}\partial_3=\frac{c_{23}}{c_{33}}\partial_3\\
h_{33}E_3=&\,\partial_3
\end{align*}
to obtain the identity
\begin{align}\label{frame.mom.I=1.2}
\notag&-\sum_J\sum_{1\leq a\leq J}(p_J-p_1)h_{aJ}E_Jf_{1a}\\
\notag=&-(p_2-p_1)h_{12}E_2f_{11}-(p_2-p_1)h_{22}E_2f_{12}\\
\notag&-(p_3-p_1)h_{13}E_3f_{11}
-(p_3-p_1)h_{23}E_3f_{12}-(p_3-p_1)h_{33}E_3f_{13}\\
=&-(p_2-p_1)f_{11}\frac{c_{12}}{c_{22}}(\partial_2-\frac{c_{23}}{c_{33}}\partial_3)\log f_{11}-(p_2-p_1)(\partial_2-\frac{c_{23}}{c_{33}}\partial_3)f_{12}\\
\notag&-(p_3-p_1)f_{11}\frac{c_{13}}{c_{33}}\partial_3\log f_{11}
-(p_3-p_1)\frac{c_{23}}{c_{33}}\partial_3f_{12}-(p_3-p_1)\partial_3f_{13} \\ 
\notag=&\,f_{11}\bigg[\frac{1}{2}(p_2-p_1)\frac{c_{12}}{c_{22}}\partial_2\log c_{11}+(p_2-p_1)c_{11}^{\frac{1}{2}}\partial_2\big[\frac{c_{12}}{c_{22}}c_{11}^{-\frac{1}{2}}\big]\\
\notag&-\frac{1}{2}\big[(p_2-p_1)\frac{c_{12}}{c_{22}}\frac{c_{23}}{c_{33}}
+(p_1-p_3)\frac{c_{13}}{c_{33}}\big]\partial_3\log c_{11}\\
\notag&+(p_3-p_2)\frac{c_{23}}{c_{33}}c_{11}^{\frac{1}{2}}\partial_3\big[\frac{c_{12}}{c_{22}}c_{11}^{-\frac{1}{2}}\big]-(p_3-p_1)c_{11}^{\frac{1}{2}}\partial_3\big[c_{11}^{-\frac{1}{2}}(\frac{c_{12}}{c_{22}}\frac{c_{23}}{c_{33}}-\frac{c_{13}}{c_{33}})\big]\bigg]
\end{align}
Multiplying the identities \eqref{frame.mom.I=1}, \eqref{frame.mom.I=1.2} by $f_{11}^{-1}$ and adding them together gives:
\begin{align}\label{equivalenceE.3}
\notag0=&\,
\partial_1p_1-\frac{1}{2}(p_2-p_1)\partial_1\log c_{22}-\frac{1}{2}(p_3-p_1)\partial_1\log c_{33}\\
\notag&-\frac{c_{12}}{c_{22}}\bigg[\partial_2p_1-\frac{1}{2}(p_2-p_1)\partial_2\log c_{22}-\frac{1}{2}(p_3-p_1)\partial_2\log c_{33}\bigg]
+(p_2-p_1)\partial_2(\frac{c_{12}}{c_{22}})\\
&+(\frac{c_{12}}{c_{22}}\frac{c_{23}}{c_{33}}-\frac{c_{13}}{c_{33}})\bigg[\partial_3p_1
-\frac{1}{2}(p_2-p_1)\partial_3\log c_{22}
-\frac{1}{2}(p_3-p_1)\partial_3\log c_{33}\\
\notag&+\frac{1}{2}(p_3-p_1)\partial_3\log c_{11}\bigg]-\frac{1}{2}\big[(p_2-p_1)\frac{c_{12}}{c_{22}}\frac{c_{23}}{c_{33}}
+(p_1-p_3)\frac{c_{13}}{c_{33}}\big]\partial_3\log c_{11}\\
\notag&-\frac{1}{2}(p_3-p_2)\frac{c_{23}}{c_{33}}\frac{c_{12}}{c_{22}}\partial_3\log c_{11}+(p_3-p_2)\frac{c_{23}}{c_{33}}\partial_3(\frac{c_{12}}{c_{22}})-(p_3-p_1)\partial_3(\frac{c_{12}}{c_{22}}\frac{c_{23}}{c_{33}}-\frac{c_{13}}{c_{33}})
\end{align}
From \eqref{equivalenceE.1}, \eqref{equivalenceE.2}, we have 
\begin{align}\label{equiv.plugin1}
\frac{1}{2}(p_3-p_1)\partial_3 \log c_{11}=&-\partial_3p_3+\frac{1}{2}(p_2-p_3)\partial_3 \log c_{22}\\
\label{equiv.plugin2}
\frac{1}{2}(p_2-p_3)\partial_2\log c_{33}=&-\partial_2p_2
+\frac{1}{2}(p_1-p_2)\partial_2\log c_{11}
-\partial_3\big[(p_3-p_2)\frac{c_{23}}{c_{33}}\big]\\
\notag&+\frac{1}{2}(p_2-p_3)\frac{c_{23}}{c_{33}}\partial_3\log(c_{11}c_{22}c_{33})
\end{align}
Plugging \eqref{equiv.plugin1}-\eqref{equiv.plugin2} into \eqref{equivalenceE.3} we obtain:
\begin{align}\label{equivalenceE.3.2}
\notag0=&\,
\partial_1p_1-\frac{1}{2}(p_2-p_1)\partial_1\log c_{22}-\frac{1}{2}(p_3-p_1)\partial_1\log c_{33}\\
\notag&-\frac{c_{12}}{c_{22}}\bigg[\partial_2(p_1-p_2)-\frac{1}{2}(p_2-p_1)\partial_2\log c_{22}+\frac{1}{2}(p_1-p_2)\partial_2\log c_{33}\\
\notag&+\frac{1}{2}(p_1-p_2)\partial_2\log c_{11}
-\partial_3\big[(p_3-p_2)\frac{c_{23}}{c_{33}}\big]
+\frac{1}{2}(p_2-p_3)\frac{c_{23}}{c_{33}}\partial_3\log(c_{11}c_{22}c_{33})
\bigg]\\
\notag&+(p_2-p_1)\partial_2(\frac{c_{12}}{c_{22}})
+(\frac{c_{12}}{c_{22}}\frac{c_{23}}{c_{33}}-\frac{c_{13}}{c_{33}})\bigg[\partial_3(p_1-p_3)
-\frac{1}{2}(p_3-p_1)\partial_3\log c_{22}\\
&\notag
-\frac{1}{2}(p_3-p_1)\partial_3\log c_{33}\bigg]
-\frac{1}{2}\big[(p_2-p_1)\frac{c_{12}}{c_{22}}\frac{c_{23}}{c_{33}}
+(p_1-p_3)\frac{c_{13}}{c_{33}}\big]\partial_3\log c_{11}\\
\notag&-\frac{1}{2}(p_3-p_2)\frac{c_{23}}{c_{33}}\frac{c_{12}}{c_{22}}\partial_3\log c_{11}+(p_3-p_2)\frac{c_{23}}{c_{33}}\partial_3(\frac{c_{12}}{c_{22}})-(p_3-p_1)\partial_3(\frac{c_{12}}{c_{22}}\frac{c_{23}}{c_{33}}-\frac{c_{13}}{c_{33}})\\
=&\,\partial_1p_1-\frac{1}{2}(p_2-p_1)\partial_1\log c_{22}-\frac{1}{2}(p_3-p_1)\partial_1\log c_{33}-\partial_2\big[(p_1-p_2)\frac{c_{12}}{c_{22}}\big]\\
\notag&-\frac{1}{2}(p_1-p_2)\frac{c_{12}}{c_{22}}\partial_2\log(c_{11}c_{22}c_{33})
-\frac{1}{2}(p_2-p_3)\frac{c_{12}}{c_{22}}\frac{c_{23}}{c_{33}}\partial_3\log(c_{11}c_{22}c_{33})\\
\notag&+\frac{1}{2}(\frac{c_{12}}{c_{22}}\frac{c_{23}}{c_{33}}-\frac{c_{13}}{c_{33}})(p_1-p_3)\partial_3\log (c_{22}c_{33})
-\frac{1}{2}\big[(p_3-p_1)\frac{c_{12}}{c_{22}}\frac{c_{23}}{c_{33}}
+(p_1-p_3)\frac{c_{13}}{c_{33}}\big]\partial_3\log c_{11}\\
\notag&+\partial_3\big[(p_3-p_2)\frac{c_{12}}{c_{22}}\frac{c_{23}}{c_{33}}\big]
+\partial_3\big[(p_1-p_3)(\frac{c_{12}}{c_{22}}\frac{c_{23}}{c_{33}}-\frac{c_{13}}{c_{33}})\big]\\
\notag=&\,\partial_1p_1-\frac{1}{2}(p_2-p_1)\partial_1\log c_{22}-\frac{1}{2}(p_3-p_1)\partial_1\log c_{33}-\partial_2\big[(p_1-p_2)\frac{c_{12}}{c_{22}}\big]\\
\notag&-\frac{1}{2}(p_1-p_2)\frac{c_{12}}{c_{22}}\partial_2\log(c_{11}c_{22}c_{33})
-\frac{1}{2}\big[(p_2-p_1)\frac{c_{12}}{c_{22}}\frac{c_{23}}{c_{33}}
+(p_1-p_3)\frac{c_{13}}{c_{33}}\big]\partial_3\log(c_{11}c_{22}c_{33})\\
\notag&-\partial_3\big[(p_2-p_1)\frac{c_{12}}{c_{22}}\frac{c_{23}}{c_{33}}
+(p_1-p_3)\frac{c_{13}}{c_{33}}\big],
\end{align}
which is \eqref{momentum} written explicitly for $i=1$, after plugging in $\kappa_1{}^2=(p_1-p_2)\frac{c_{12}}{c_{22}}$, $\kappa_2{}^3=(p_2-p_3)\frac{c_{23}}{c_{33}}$, $\kappa_1{}^3=(p_2-p_1)\frac{c_{12}c_{23}}{c_{22}c_{33}}+(p_1-p_3)\frac{c_{13}}{c_{33}}$. 
\end{proof}

\subsection{The leading order behavior of the iterates}

In this subsection, we derive the necessary bounds to prove \ref{item.approx1} in Theorem \ref{thm:approx.sol}.
\begin{lemma}\label{lem:pre2}
Let $n \in \mathbb{N}$. Assume that for every $I,a$, every multi-index $\alpha$, and for all $(t,x)\in(0,t_n]\times[0,\delta]^3$ there holds:
 \begin{align}\label{en.assum}
|\partial^\alpha_x( e^{[\bf n]}_{ Ia}- e^{\bf [0]}_{ Ia})|\leq &\left\{\begin{array}{ll}
 C_{\alpha,n}t^{-p_I+\varepsilon}  & I\leq a, \\
 C_{\alpha,n}t^{p_I-2p_a+\varepsilon}  & I>a,
 \end{array}\right.
\end{align}
Then, the following bound holds as well:
 \begin{align}\label{omegan.conc}
|\partial^\alpha_x(
\omega^{[\bf n]}_{bC}- \omega^{\bf [0]}_{bC})|\leq &\left\{\begin{array}{ll} 
 C_{\alpha,n}t^{p_C+\varepsilon}  & b\leq C, \\
    C_{\alpha,n}t^{2p_b-p_C+\varepsilon} & b>C.
\end{array}\right.
\end{align}
for all $(t,x)\in(0,t_n]\times[0,\delta]^3$, every multi-index $\alpha$, and all indices $b,C$. 
\end{lemma}
\begin{proof}
By definition, $e^{[\bf n]}_{Ia}\omega_{aC}^{[\bf n]}=\delta_{IC}$. Hence, by Cramer's rule, we have
\begin{align}\label{Cramer}
(\omega^{[\bf n]}_{bC})=\frac{1}{\mathrm{det}(e_{Ia}^{[\bf n]})}(\Omega_{bC}^{[\bf n]}),
\end{align}
where $(-1)^{b+C}\Omega_{bC}^{[\bf n]}$ is the determinant of the matrix produced by deleting the $b$-column and $C$-row of $(e_{Cb}^{[\bf n]})$. From our assumption \eqref{en.assum} and definition \eqref{eIi.kIJ.0}, it holds 
\begin{align}\label{det.en.est}
\begin{split}
|\partial^\alpha_x[\mathrm{det}(e^{[\bf n]}_{Ia})-\mathrm{det}(e^{[\bf 0]}_{Ia})]|\leq&\, C_{\alpha,n}t^{-1+\varepsilon},\\ 
|\partial^\alpha_x[(\Omega^{[\bf n]}_{bC})- (\Omega^{[\bf 0]}_{bC})]|\leq &
\left\{\begin{array}{ll}
C_{\alpha,n}t^{-1+p_b+\varepsilon}& b\leq C\\
C_{\alpha,n}t^{-1+2p_b-p_C+\varepsilon}& b>C
\end{array}\right.,
\end{split}
\end{align}
where 
\begin{align}\label{E0.est}
\mathrm{det}(e^{[\bf 0]}_{Ia})=f_{11}f_{22}f_{33}t^{-1},\qquad (\Omega^{[\bf 0]}_{bC})=\left\{\begin{array}{ll}
(f\star f)_{\underline{b}C}t^{-1+p_{\underline{b}}} & b\leq C\\
0 & b>C
\end{array}\right.
.\end{align}
Next, we compute
\begin{align}\label{omega.n-omega.0}
(\omega^{[\bf n]}_{bC})-(\omega^{\bf [0]}_{bC})=\frac{1}{\mathrm{det}(e_{Ia}^{[\bf n]})}[(\Omega_{bC}^{[\bf n]})-(\Omega_{bC}^{[\bf 0]})]
-\frac{\mathrm{det}(e_{Ia}^{[\bf n]})-\mathrm{det}(e_{Ia}^{[\bf 0]})}{\mathrm{det}(e_{Ia}^{[\bf n]})\mathrm{det}(e_{Ia}^{[\bf 0]})}(\Omega_{bC}^{[\bf 0]}) .
\end{align}
Combining \eqref{Cramer}-\eqref{omega.n-omega.0}, we arrive at the desired conclusion \eqref{omegan.conc}.
\end{proof}
\begin{lemma}\label{lem:pre3}
Let $n \in \mathbb{N}$. Assume that for every $I,C,a,b$, every multi-index $\alpha$, and for all $(t,x)\in(0,t_n]\times[0,\delta]^3$ there holds:
 \begin{align} 
 \notag
|\partial^\alpha_x( e^{[\bf n]}_{ Ia}- e^{\bf [0]}_{ Ia})|\leq &\left\{\begin{array}{ll}
 C_{\alpha,n}t^{-p_I+\varepsilon}  & I\leq a, \\
 C_{\alpha,n}t^{p_I-2p_a+\varepsilon}  & I>a,
 \end{array}\right.\\ \notag
|\partial^\alpha_x(
\omega^{[\bf n]}_{bC}- \omega^{\bf [0]}_{bC})|\leq &\left\{\begin{array}{ll} 
 C_{\alpha,n}t^{p_C+\varepsilon}  & b\leq C, \\
    C_{\alpha,n}t^{2p_b-p_C+\varepsilon} & b>C.
\end{array}\right.
\end{align}
Then, shrinking $t_n$ if necessary, the following basic estimate is valid:
\begin{align}\label{dx.e.dx.omega}
\begin{split}
|\partial_x^{\alpha_1}\omega_{aD}^{[\bf n]}\partial_x^{\alpha_2}e_{Ia}^{[\bf n]} | \leq&\, C_{\alpha,n} t^{\lvert p_D-p_I|}\lvert\log t|^{\lvert \alpha|},\\
 |\partial_x^{\alpha_1}\omega_{bC}^{[\bf n]}\partial_x^{\alpha_2}e_{Ca}^{[\bf n]}| \leq&\, C_{\alpha,n} t^{\lvert p_b-p_a|}\lvert\log t|^{\lvert \alpha|},
\end{split}
\end{align}
for all $(t,x)\in(0,t_n]\times[0,\delta]^3$, all indices $a,b,I,D$, and $| \alpha_1 |+| \alpha_2|=| \alpha|$.
\end{lemma}
\begin{proof}
The argument for the two bounds is similar, so we only discuss the first one. If $D=I$, then it is clear that the least decaying term corresponds to $a=I$. Adding and subtracting the zeroth iterates we have 
\begin{align}\label{dx.e.omega.I=D}
|\partial_x^{\alpha_1}\omega_{aD}^{[\bf n]}\partial_x^{\alpha_2}e_{Ia}^{[\bf n]}|
\leq C_{\alpha,n}t^\varepsilon+|\partial_x^{\alpha_1}\omega_{aD}^{[\bf 0]}\partial_x^{\alpha_2}e_{Ia}^{[\bf 0]}|\leq  C_{\alpha,n}(t^\varepsilon+|\log t|^{|\alpha|}),
\end{align}
where each logarithm comes from when $\partial_x$ hits $t^{p_I(x)},t^{-p_I(x)}$ in \eqref{eIi.kIJ.0}. Since $t^\varepsilon\leq |\log t|^{|\alpha|}$, provided $t_n$ is sufficiently small, the desired bound follows. 

Let us assume now $D<I$. Then we have three cases, depending on the value of $a$. For $a\ge I$, we have
\begin{align}\label{dx.e.omega.D<I<a}
|\partial_x^{\alpha_2}e_{Ia}^{[\bf n]}|\leq C_{\alpha,n}t^{-p_I}|\log t|^{|\alpha_2|},\qquad |\partial_x^{\alpha_1}\omega_{aD}^{[\bf n]}|=|\partial_x^{\alpha_1}(\omega_{aD}^{[\bf n]}-\omega_{aD}^{[\bf 0]})|\leq C_{\alpha,n}t^{2p_a-p_D+\varepsilon},
\end{align}
which leads to the desired bound after noticing that $2p_a-p_D-p_I\ge p_I-p_D=|p_I-p_D|$. For $a\leq D$, we instead have 
\begin{align}\label{dx.e.omega.a<D<I}
|\partial_x^{\alpha_1}\omega_{aD}^{[\bf n]}|\leq C_{\alpha,n}t^{p_D}|\log t|^{|\alpha_1|},
\qquad |\partial_x^{\alpha_2}e_{Ia}^{[\bf n]}|=|\partial_x^{\alpha_2}(e_{Ia}^{[\bf n]}-e_{Ia}^{[\bf 0]})|\leq C_{\alpha,n}t^{p_I-2p_a+\varepsilon},
\end{align}
which again agrees with the claimed bound, since $p_I+p_D-2p_a\ge p_I-p_D=|p_I-p_D|$. Lastly, for $D<a<I$, it holds 
\begin{align}\label{dx.e.omega.D<a<I}
\begin{split}
|\partial_x^{\alpha_1}\omega_{aD}^{[\bf n]}|=
|\partial_x^{\alpha_1}(\omega_{aD}^{[\bf n]}-\omega_{aD}^{[\bf 0]})|\leq&\, C_{\alpha,n}t^{2p_a-p_D+\varepsilon},\\ 
|\partial_x^{\alpha_2}e_{Ia}^{[\bf n]}|=|\partial_x^{\alpha_2}(e_{Ia}^{[\bf n]}-e_{Ia}^{[\bf 0]})|\leq&\, C_{\alpha,n}t^{p_I-2p_a+\varepsilon},
\end{split}
\end{align} 
which actually gives a $t^\varepsilon$ better bound than the one in \eqref{dx.e.dx.omega}.

The case $I>D$ is treated similarly.  
\end{proof}
We now proceed to the circular estimates which imply \ref{item.approx1}, \ref{item.approx2} of Theorem \ref{thm:approx.sol} by an induction argument.
\begin{lemma}\label{lem:approx1}
Let $n \in \mathbb{N}$. Assume that, for all indices $I,C,a,b$, multi-indices $\alpha$, $(t,x)\in(0,t_n]\times[0,\delta]^3$ there holds
\begin{align*}
|\partial^\alpha_x( e^{[\bf n]}_{ Ia}- e^{\bf [0]}_{ Ia})|\leq &\left\{\begin{array}{ll}
 C_{\alpha,n}t^{-p_I+\varepsilon}  & I\leq a, \\
 C_{\alpha,n}t^{p_I-2p_a+\varepsilon}  & I>a,
 \end{array}\right.\\
|\partial^\alpha_x(
\omega^{[\bf n]}_{bC}- \omega^{\bf [0]}_{bC})|\leq &\left\{\begin{array}{ll} 
C_{\alpha,n}t^{p_C+\varepsilon}  & b\leq C, \\
C_{\alpha,n}t^{2p_b-p_C+\varepsilon} & b>C.
\end{array}\right.
\end{align*}
Then, the spatial Ricci curvature satisfies
\begin{align}\label{Ric.conc}
|\partial^\alpha_x R_{IJ}^{\bf [n]}|\leq C_{\alpha,n}t^{-2+\varepsilon+|p_I-p_J|},
\end{align}
for all multi-indices $\alpha$, indices $I,J$, and $(t,x)\in(0,t_n]\times[0,\delta]^3$.
\end{lemma}
\begin{proof}
We will prove the stronger statement, namely that for all multi-indices $\alpha$ and for all $I,J$ there holds
\begin{align}\label{Ricn.ref.est}
| \partial_x^{\alpha} R_{IJ}^{[\bf n]} | \leq C_{\alpha,n}t^{-2+2\varepsilon+ |p_I-p_J|}|\log t|^{2+|\alpha|}.
\end{align}
Then, the desired estimate follows by shrinking $t_n$ if necessary, to absorb $|\log t|^{2+|\alpha|}$ into $t^\varepsilon$. 

Expanding $R_{IJ}^{[\bf n]}$ using the formulas \eqref{RIJ}, \eqref{gammaIJB}, we notice that 
\begin{align}\label{RIJ.n.form}
R_{IJ}^{[\bf n]}=L\big(e_{\ell_1i}^{[\bf n]}\partial_i(\omega_{a\ell_2}^{[\bf n]}e_{\ell_3j}^{[\bf n]}\partial_je_{\ell_4a}^{[\bf n]}), \omega_{a\ell_1}^{[\bf n]}(e_{\ell_2i}^{[\bf n]}\partial_ie_{\ell_3a}^{[\bf n]})\omega_{\ell_4b}^{[\bf n]}(e_{\ell_5i}^{[\bf n]}\partial_ie_{\ell_6b}^{[\bf n]})\big),
\end{align}
where $L(\cdot,\cdot)$ is a linear expression in its arguments. Here, $\ell_m$ are pairwise contracting or equal to $I,J$ (one for each term). Employing Lemma \ref{lem:pre3} we deduce that
\begin{align}\label{RIJ.n.est1}
\notag|\partial_x^\alpha[e_{\ell_1i}^{[\bf n]}\partial_i(\omega_{a\ell_2}^{[\bf n]}e_{\ell_3j}^{[\bf n]}\partial_je_{\ell_4a}^{[\bf n]})]|\leq &\,C_{\alpha,n}t^{-p_{\ell_1}-p_{\ell_3}+|p_{\ell_4}-p_{\ell_2}|}|\log t|^{2+|\alpha|}\\
\leq&\, C_{\alpha,n}(t^{-p_I-p_J}+t^{-2\max p_i+|p_I-p_J|})|\log t|^{2+|\alpha|} \\
\notag\leq&\, C_{\alpha,n}t^{-2+2\varepsilon+|p_I-p_J|}|\log t|^{2+|\alpha|}
\end{align}
and
\begin{align}\label{RIJ.n.est2}
\notag|\partial_x^\alpha[\omega_{a\ell_1}^{[\bf n]}(e_{\ell_2i}^{[\bf n]}\partial_ie_{\ell_3a}^{[\bf n]})\omega_{\ell_4b}^{[\bf n]}(e_{\ell_5i}^{[\bf n]}\partial_ie_{\ell_6b}^{[\bf n]})]|\leq&\, C_{\alpha,n}t^{-p_{\ell_2}+|p_{\ell_3}-p_{\ell_1}|}t^{-p_{\ell_5}+|p_{\ell_6}-p_{\ell_4}|}|\log t|^{2+|\alpha|}\\
\leq&\,C_{\alpha,n}(t^{-p_I-p_J}+t^{-2\max p_i+|p_I-p_J|})|\log t|^{2+|\alpha|} \\
\notag\leq&\, C_{\alpha,n}t^{-2+2\varepsilon+|p_I-p_J|}|\log t|^{2+|\alpha|}
\end{align}
Applying the bounds \eqref{RIJ.n.est1}-\eqref{RIJ.n.est2} to \eqref{RIJ.n.form} gives the claimed bound \eqref{Ricn.ref.est}.
\end{proof}
\begin{lemma}\label{lem:approx2}
Let $n \geq 1$ and suppose there exists $t_n>0$ such that for every multi-index $\alpha$ and indices $I,J$ the inequalities
\begin{align*}
| \partial_x^{\alpha}(k_{IJ}^{[\bf{n-1}]}- k_{IJ}^{[\bf{0}]}) | \leq &\,C_{\alpha,n} t^{-1+\varepsilon+| p_I-p_J|}, \\
| \partial_x^{\alpha}R_{IJ}^{[\bf{n-1}]} | \leq&\, C_{\alpha, n} t^{-2+\varepsilon+| p_I-p_J|},  
\end{align*}
hold true, for all $(t,x)\in(0,t_{n-1}]\times[0,\delta]^3$. Then, there exists $t_n\in(0,t_{n-1})$ such that the following bound holds:
\begin{align*}
 |\partial^\alpha_x(k_{IJ}^{[\bf n]}-k_{IJ}^{[\bf 0]})|\leq 
C_{\alpha,n}t^{-1+\varepsilon+|p_I-p_J|} 
\end{align*}
for all $(t,x)\in(0,t_{n-1}]\times[0,\delta]^3$.
\end{lemma}
\begin{proof}
We notice that \eqref{kIJ.it} can be rewritten as
\begin{align}\label{kijrewritten} 
\partial_t (t  k_{IJ}^{[\bf{n}]}-tk_{IJ}^{[\bf 0]}) - (k_{CC}^{[\bf{n-1}]}+\frac{1}{t}) (t  k_{IJ}^{[\bf{n}]}-tk^{[\bf 0]}_{IJ}) = t  R_{IJ}^{[\bf{n-1}]}+(k_{CC}^{[\bf{n-1}]}+\frac{1}{t}) tk^{[\bf 0]}_{IJ}.
\end{align} 
By the assumed bounds on $k_{IJ}^{[\bf{n-1}]}$, however, we conclude that
\begin{align}\label{trkbound1} 
\sup_{x \in[0,\delta]^3} | \partial_x^{\alpha}(k_{CC}^{[\bf{n-1}]}+\frac{1}{t})| = \sup_{x \in[0,\delta]^3} | \partial_x^{\alpha}(k_{CC}^{[\bf{n-1}]} - k_{CC}^{[\bf{0}]})|\leq  C_{\alpha,n}t^{-1+\varepsilon}, 
\end{align}
while the bounds on $R_{IJ}^{[\bf{n-1}]}$ imply that 
\begin{align}\label{tRbound1}  
|t\partial_x^{\alpha}R_{IJ}^{[\bf{n-1}]} | \leq C_{\alpha, n}t^{-1+\varepsilon+| p_I-p_J|},\qquad\forall (t,x)\in(0,t_{n-1}]\times[0,\delta]^3.
\end{align} 
Also, given definition \eqref{eIi.kIJ.0}, we notice that the asymptotic initial conditions \eqref{asym.cond.it} are equivalent to the differences $t  k_{IJ}^{[\bf{n}]}-tk_{IJ}^{[\bf 0]}$ having trivial initial conditions at $t=0$. Hence, we may solve \eqref{kijrewritten} via integrating factors to obtain
\begin{align}\label{kijrewritten.sol}
t(k_{IJ}^{[\bf n]}-k_{IJ}^{[\bf 0]})=e^{\int^t_0w^{[\bf n-1]}d\tau}\int^t_0 e^{-\int^\tau_0w^{[\bf n-1]}d\overline{\tau}}
\big\{\tau R_{IJ}^{[\bf n-1]}+(k_{CC}^{[\bf n-1]}-k_{CC}^{[\bf 0]})\tau k_{IJ}^{[\bf 0]}\big\}d\tau
\end{align}
where $w^{[\bf n-1]}(t,x)=(k_{CC}^{[\bf n-1]}-k_{CC}^{[\bf 0]})(t,x)$. Note that 
\begin{align*}
\sup_{x\in[0,\delta]^3}\bigg|\partial_x^\alpha\int^t_0w^{[\bf n-1]}d\tau\bigg|\leq C_{\alpha,n}t^{\varepsilon}
\end{align*}
by virtue of \eqref{trkbound1}. Differentiating now \eqref{kijrewritten.sol} with $\partial_x^\alpha$ and using the bounds \eqref{trkbound1}, \eqref{tRbound1} yields the desired result.
\end{proof}
\begin{lemma} \label{lem:approx3}
Let $N \geq 1$ and suppose there exists $t_N>0$ such that for every $1\leq n \leq N$ and every multi-index $\alpha$, $k^{[\bf{n}]}_{IJ}$ satisfies the following estimate:
\begin{align*}
 | \partial_x^{\alpha} (\en{k}_{IJ}-\ze{k}_{IJ}) | \leq C_{\alpha, n} t^{-1+\varepsilon+| p_I-p_J|},
\end{align*}
for all $(t,x)\in (0, t_{N-1}]\times[0,\delta]^3$.
Then, after choosing $t_N$ smaller if necessary, there holds  
\begin{align} \notag
|\partial^\alpha_x( e^{[\bf n]}_{ Ia}- e^{\bf [0]}_{ Ia})|\leq &\left\{\begin{array}{ll}
C_{\alpha,n}t^{-p_I+\varepsilon}  & I\leq a, \\
C_{\alpha,n}t^{p_I-2p_a+\varepsilon}  & I>a,
\end{array}\right.
\end{align}
for all $(t,x)\in(0,t_N]\times[0,\delta]^3$ and all indices $I,a$.
\end{lemma}

\begin{proof}
Rewrite \eqref{eIi.it} in the form
\begin{align}\label{eIi.rewritten}
\begin{split}
&\partial_t \big[ t^{p_{\underline{I}}}(e_{\underline{I}a}^{[\bf n]} -e_{\underline{I}a}^{[\bf 0]})\big] -(k_{\underline{I} \underline{I}}^{[\bf n]} -k_{\underline{I} \underline{I}}^{[\bf 0]} )\big[t^{p_{\underline{I}}}(e_{\underline{I}a}^{[\bf n]} -e_{\underline{I}a}^{[\bf 0]})\big]\\ 
=&\,t^{p_I} e_{\underline{I}a}^{[\bf 0]}(k_{\underline{I} \underline{I}}^{[\bf n]}- k_{\underline{I} \underline{I}}^{[\bf 0]}) + \sum_{C\neq I}k_{IC}^{[\bf{n-1}]} \big[t^{p_I}(e_{Ca}^{[\bf{n-1}]}-e_{Ca}^{[\bf 0]})\big] + \sum_{C\neq I}t^{p_I} e_{Ca}^{[\bf 0]}  ( k_{IC}^{[\bf{n-1}]}-k_{IC}^{[\bf{0}]}).
\end{split}
\end{align}
Let us denote by $\Omega_{Ia}^{[\bf n]}$ the preceding RHS and let $w_I^{\bf [n]}=k_{\underline{I} \underline{I}}^{[\bf n]} -k_{\underline{I} \underline{I}}^{[\bf 0]}$. By our assumption and finite induction in $n$, we have that the following bounds hold:
\begin{align}\label{wI.EIa.est}
|\partial_x^\alpha w_I^{[\bf n]}|\leq C_{\alpha,n} t^{-1+\varepsilon},\qquad |\partial_x^\alpha \Omega_{Ia}^{[\bf n]}|\leq 
\left\{\begin{array}{ll}
C_{\alpha,n}t^{-1+\varepsilon}& I\leq a,\\
C_{\alpha,n}t^{-1+2p_I-2p_a+\varepsilon}& I>a.
\end{array}\right.
\end{align}
Notice that $t^{p_{\underline{I}}}(e_{\underline{I}a}^{[\bf n]} -e_{\underline{I}a}^{[\bf 0]})$ has trivial data at $t=0$, due to the asymptotic initial condition \eqref{asym.cond.it}.
Solving \eqref{eIi.rewritten} via integrating factors gives
\begin{align}\label{eIi.rewritten.sol}
t^{p_{\underline{I}}}(e_{\underline{I}a}^{[\bf n]} -e_{\underline{I}a}^{[\bf 0]})
=e^{\int^t_0w^{[\bf n]}_{\underline{I}}d\tau}\int^t_0e^{-\int^\tau_0w^{[\bf n]}_{\underline{I}}d\overline{\tau}}\Omega^{[\bf n]}_{\underline{I}a}d\tau
\end{align}
Differentiating \eqref{eIi.rewritten.sol} with $\partial_x^\alpha$ and using \eqref{wI.EIa.est}, we conclude the desired result.
\end{proof}
A straightforward induction argument now implies that:
\begin{proposition}\label{prop:ind1}
Points \ref{item.approx1}, \ref{item.approx2} of Theorem \ref{thm:approx.sol} hold true.
\end{proposition}
\begin{proof}
Lemma \ref{lem:approx1} implies that the assumption of Lemma \ref{lem:approx2} is satisfied for $n=1$. Hence, points \ref{item.approx1}, \ref{item.approx2} of Theorem \ref{thm:approx.sol} are valid for $n=0$. Assume that they are satisfied for all iterates with superscript $n-1$. Then Lemma \ref{lem:approx1} implies that \eqref{kIJ.n-1.est} is satisfied. By Lemma \ref{lem:approx2} we also obtain \eqref{eIa.n-0.est}. Lemma \ref{lem:pre2} implies that \eqref{omegabC.n-0.est} holds true. Finally, Lemma \ref{lem:approx1} gives \eqref{Ric.n.est}. This completes the induction argument and the proof of points \ref{item.approx1}, \ref{item.approx2} of Theorem \ref{thm:approx.sol}.
\end{proof}

\subsection{Comparing successive iterates}\label{subsec:succ.it}

In the next lemmas we derive circular estimates for $k^{[\bf n]}-k^{[\bf{n-1}]}$, $e^{[\bf n]}-e^{[\bf n-1]}$.
\begin{lemma}\label{lem:kijnn-1}
Let $N\geq 2$ and suppose there exists $t_{N-1}>0$, such that for every $1\leq n \leq N-1$, every multi-index $\alpha$, and indices $I,C,b,a$ the following holds for all $(t,x)\in (0, t_{N-1}]\times [0,\delta]^3$:
\begin{align*}
|\partial_x^{\alpha}(e^{[\bf n]}_{Ia}-e_{Ia}^{[\bf{n-1}]}) | \leq& 
\left\{\begin{array}{ll}
C_{\alpha,n}t^{-p_I+n\varepsilon}& I\leq a\\
C_{\alpha,n}t^{p_I-2p_a+n\varepsilon}& I>a
\end{array}\right.  \\
|\partial_x^{\alpha}(\omega^{[\bf n]}_{bC}-\omega_{bC}^{[\bf{n-1}]}) | \leq& \left\{\begin{array}{ll}
C_{\alpha,n}t^{p_C+n\varepsilon}& b\leq C\\
C_{\alpha,n}t^{2p_b-p_C+n\varepsilon}& b>C
\end{array}\right. 
\end{align*}
Then, taking $t_N \in (0,t_{N-1})$ smaller if necessary, the following bound holds:
\begin{align*}
| \partial_x^{\alpha}(k_{IJ}^{[\bf n]}-k_{IJ}^{[\bf {n-1}]}) | \leq C^{\prime}_{\alpha,n} t^{-1+|p_I-p_J|+n \varepsilon},
\end{align*}
for every $2\leq n \leq N$, every multi-index $\alpha$, indices $I,J$, and for all $(t,x)\in (0,t_N)\times [0,\delta]^3$.
\end{lemma}
\begin{proof}
\noindent \textit{Step 1.} First, we estimate the difference of successive Ricci curvature components $R_{IJ}^{[\bf n]}-R_{IJ}^{[\bf n-1]}$. Going back to the expression \eqref{RIJ.n.form}, we notice that the previous difference has the form
\begin{align}\label{RIJ.diff}
\notag R_{IJ}^{[\bf n]}-R_{IJ}^{[\bf n-1]}
=&\, L\big((e_{\ell_1i}^{[\bf n]}-e_{\ell_1i}^{[\bf n-1]})\partial_i(\omega_{a\ell_2}^{[\bf n]}e_{\ell_3j}^{[\bf n]}\partial_je_{\ell_4a}^{[\bf n]}),
\ldots,\\ 
& e_{\ell_1i}^{[\bf n-1]}\partial_i[\omega_{a\ell_2}^{[\bf n-1]}e_{\ell_3j}^{[\bf n-1]}\partial_j(e_{\ell_4a}^{[\bf n]}-e_{\ell_4a}^{[\bf n-1]})],\\
\notag &(\omega_{a\ell_1}^{[\bf n]}-\omega_{a\ell_1}^{[\bf n-1]})(e_{\ell_2i}^{[\bf n]}\partial_ie_{\ell_3a}^{[\bf n]})\omega_{\ell_4b}^{[\bf n]}(e_{\ell_5i}^{[\bf n]}\partial_ie_{\ell_6b}^{[\bf n]}),\ldots, \\
\notag& \omega_{a\ell_1}^{[\bf n-1]}(e_{\ell_2i}^{[\bf n-1]}\partial_ie_{\ell_3a}^{[\bf n-1]})\omega_{\ell_4b}^{[\bf n-1]}[e_{\ell_5i}^{[\bf n-1]}(\partial_ie_{\ell_6b}^{[\bf n]}-\partial_ie_{\ell_6b}^{[\bf n-1]})]\big)
\end{align}
Every term in the last linear expression can be estimated exactly as in \eqref{RIJ.n.est1}, \eqref{RIJ.n.est2}, only now each factor which is a difference of successive iterates, $e^{[\bf n]}-e^{[\bf n-1]}$ or $\omega^{[\bf n]}-\omega^{[\bf n-1]}$, contributes an extra $t^{n\varepsilon}$, resulting in the bound
\begin{align}\label{RIJ.diff.est}
|\partial_x^\alpha (R_{IJ}^{[\bf n]}-R_{IJ}^{[\bf n-1]})|\leq C_{\alpha,n}t^{-2+(n+2)\varepsilon+|p_I-p_J|}|\log t|^{2+|\alpha|},
\end{align}
for every $1\leq n\leq N-1$. 
\vspace*{.05in}

\noindent\textit{Step 2.} For $2\leq n\leq N$, \eqref{kIJ.it} implies the following equation for $k^{[\bf n]}-k^{[\bf n-1]}$:
\begin{align} \label{enkenmokdiff}
\begin{split}
&\partial_t\big(t(\en{k}_{IJ}-\enmo{k}_{IJ})\big) -(\enmo{k}_{CC}+\frac{1}{t})t(\en{k}_{IJ}-\enmo{k}_{IJ}) \\
=& \, t \enmo{k}_{IJ}(\enmo{k}_{CC}-\enmt{k}_{CC}) +t (\enmo{R}_{IJ}-\enmt{R}_{IJ}).
\end{split}
\end{align}
Solving \eqref{enkenmokdiff} via integrating factors, we obtain the formula:
\begin{align}\label{enkenmokexplicit}
\begin{split}
&t (\en{k}_{IJ}-\enmo{k}_{IJ})(t,x)\\
 = &\,e^{\int^t_0w^{[\bf n-1]}d\tau}     
\int_0^t e^{-\int^\tau_0w^{[\bf n-1]}d\overline{\tau}} \big\{\tau\enmo{k}_{IJ}(\enmo{k}_{CC}-\enmt{k}_{CC}) +\tau(\enmo{R}_{IJ}-\enmt{R}_{IJ})\big\} d\tau,
\end{split}
\end{align}
where $w^{[\bf n-1]}=k^{[\bf n-1}_{CC}+t^{-1}=k^{[\bf n-1]}_{CC}-k^{[\bf 0]}_{CC}$, satisfying $|\partial_x^\alpha w^{[\bf n-1]}|\leq C_{\alpha,n} t^{-1+\varepsilon}$. The desired bound on $k^{[\bf n]}-k^{\bf [n-1]}$ follows by finite induction, after differentiating \eqref{enkenmokexplicit} with $\partial_x^\alpha$, and using the already derived estimate \eqref{RIJ.diff.est}. 
\end{proof}

\begin{lemma}\label{lem:en-en-1}
Let $N\geq 1$ and suppose there exists $t_N>0$ such that for every $1\leq n \leq N$, every multi-index $\alpha$ and indices $I,J$, the following holds:
\begin{align*}
|\partial_x^{\alpha} (\en{k}_{IJ}-{k}_{IJ}^{[\bf{n-1}]})|\leq C_{\alpha,n} t^{-1+n\varepsilon+|p_I-p_J|},
\end{align*}
for all $(t,x)\in (0,t_N]\times [0,\delta]^3$. Then, after choosing $t_N>0$ smaller if necessary, for every multi-index $\alpha$ and $1\leq n \leq N$, the following bounds hold:
\begin{align*}
|\partial_x^{\alpha}(e^{[\bf n]}_{Ia}-e_{Ia}^{[\bf{n-1}]}) | \leq& 
\left\{\begin{array}{ll}
C_{\alpha,n}t^{-p_I+n\varepsilon}& I\leq a\\
C_{\alpha,n}t^{p_I-2p_a+n\varepsilon}& I>a
\end{array}\right.  \\
|\partial_x^{\alpha}(\omega^{[\bf n]}_{bC}-\omega_{bC}^{[\bf{n-1}]}) | \leq& \left\{\begin{array}{ll}
C_{\alpha,n}t^{p_C+n\varepsilon}& b\leq C\\
C_{\alpha,n}t^{2p_b-p_C+n\varepsilon}& b>C
\end{array}\right. 
\end{align*}
for all $(t,x)\in(0,t_N]\times[0,\delta]^3$ and $1\leq n \leq N$. 
\end{lemma}

\begin{proof}
We begin by writing the equation satisfied by $\en{e}_{Ia}-\enmo{e}_{Ia}$, using \eqref{eIi.it}, in the following form: 
\begin{align}\label{eIa.diff.eq}
&\partial_t\big[t^{p_{\underline{I}}}(\en{e}_{\underline{I}a}-\enmo{e}_{\underline{I}a})\big]- (\en{k}_{\underline{I} \underline{I}}-k_{\underline{I} \underline{I}}^{[\bf{0}]})t^{p_{\underline{I}}}(\eIan-\enmo{e}_{Ia}) \\ 
\notag=&\,t^{p_{\underline{I}}} \enmo{e}_{\underline{I}a}(\en{k}_{\underline{I} \underline{I}}-\enmo{k}_{\underline{I} \underline{I}})+\sum_{C\neq I}\bigg[\enmo{k}_{\underline{I}C}(t^{p_{\underline{I}}}(\enmo{e}_{Ca}-\enmt{e}_{Ca}))+t^{p_{\underline{I}}}  \enmt{e}_{Ca}(\enmo{k}_{\underline{I}C}-\enmt{k}_{\underline{I}C}) \bigg].
\end{align}
Then we proceed by solving the latter equation using integrating factors, as we did in the proof of Lemma \ref{lem:approx3}, and argue by finite induction as in the proof of the previous lemma to infer the bound
\begin{align}\label{en-en-1.est}
|\partial_x^{\alpha}(e^{[\bf n]}_{Ia}-e_{Ia}^{[\bf{n-1}]}) | \leq& 
\left\{\begin{array}{ll}
C_{\alpha,n}t^{-p_I+n\varepsilon}& I\leq a\\
C_{\alpha,n}t^{p_I-2p_a+n\varepsilon}& I>a
\end{array}\right.
\end{align}
for all $(t,x)\in(0,t_N]\times[0,\delta]^3$ and $1\leq n \leq N$. 

Once we have controlled the difference of successive iterates for the frame, we employ the formula \eqref{Cramer} to obtain an expression for the difference $\omega^{[\bf n]}-\omega^{[\bf n-1]}$ analogous to \eqref{omega.n-omega.0}:
\begin{align}\label{omegan-omegan-1}
(\omega^{[\bf n]}_{bC})-(\omega^{\bf [n-1]}_{bC})=\frac{1}{\mathrm{det}(e_{Ia}^{[\bf n]})}[(\Omega_{bC}^{[\bf n]})-(\Omega_{bC}^{[\bf n-1]})]
-\frac{\mathrm{det}(e_{Ia}^{[\bf n]})-\mathrm{det}(e_{Ia}^{[\bf n-1]})}{\mathrm{det}(e_{Ia}^{[\bf n]})\mathrm{det}(e_{Ia}^{[\bf n-1]})}(\Omega_{bC}^{[\bf n-1]})
\end{align}
Then, we notice that the differences of determinants and co-factor matrices of the frame iterates satisfy bounds analogous to \eqref{det.en.est}, only $t^{n\varepsilon}$ by virtue of \eqref{en-en-1.est}:
\begin{align*}
|\partial^\alpha_x[\mathrm{det}(e^{[\bf n]}_{Ia})-\mathrm{det}(e^{[\bf n-1]}_{Ia})]|\leq&\, C_{\alpha,n}t^{-1+n\varepsilon},\\ 
|\partial^\alpha_x[(\Omega^{[\bf n]}_{bC})- (\Omega^{[\bf n-1]}_{bC})]|\leq &
\left\{\begin{array}{ll}
C_{\alpha,n}t^{-1+p_b+n\varepsilon}& b\leq C\\
C_{\alpha,n}t^{-1+2p_b-p_C+n\varepsilon}& b>C
\end{array}\right.
\end{align*}
Applying the latter to the differentiated version of \eqref{omegan-omegan-1} gives the desired bound for the difference of successive co-frame iterates, which completes the proof of the lemma.
\end{proof}
\begin{proposition}\label{prop:ind2}
For any $N\ge1$, there exists $t_N>0$, such that for every $1\leq n\leq N$, every multi-index $\alpha$, and indices $I,C,a,b$, the following estimates hold: 
\begin{align*}
|\partial_x^{\alpha}(e^{[\bf n]}_{Ia}-e_{Ia}^{[\bf{n-1}]}) | \leq& 
\left\{\begin{array}{ll}
C_{\alpha,n}t^{-p_I+n\varepsilon}& I\leq a\\
C_{\alpha,n}t^{p_I-2p_a+n\varepsilon}& I>a
\end{array}\right.  \\
|\partial_x^{\alpha}(\omega^{[\bf n]}_{bC}-\omega_{bC}^{[\bf{n-1}]}) | \leq& \left\{\begin{array}{ll}
C_{\alpha,n}t^{p_C+n\varepsilon}& b\leq C\\
C_{\alpha,n}t^{2p_b-p_C+n\varepsilon}& b>C
\end{array}\right. \\
|\partial_x^{\alpha} (\en{k}_{IJ}-{k}_{IJ}^{[\bf{n-1}]})|\leq&\, C_{\alpha,n} t^{-1+n\varepsilon+|p_I-p_J|}\\
|\partial_x^\alpha (R_{IJ}^{[\bf n]}-R_{IJ}^{[\bf n-1]})|\leq &\,C_{\alpha,n}t^{-2+(n+1)\varepsilon+|p_I-p_J|}
\end{align*}
for all $(t,x)\in(0,t_N]\times[0,\delta]^3$.
\end{proposition}
\begin{proof}
For $N=1$, the conclusion is satisfied by virtue of points \ref{item.approx1}, \ref{item.approx2} of Theorem \ref{thm:approx.sol}. Assuming the conclusion holds for $N-1$, Lemmas \ref{lem:kijnn-1}, \ref{lem:en-en-1} imply that the above estimates regarding $e_{Ia}^{\bf [n]},\omega^{\bf [n]}_{bC},k_{IJ}^{[\bf n]}$ are valid for every $1\leq n\leq N$. By \eqref{RIJ.diff.est}, we also have the desired estimate for the differences of the spatial Ricci components up to $n=N$. This completes the proof by induction.
\end{proof}

\subsection{Approximate solution to the evolution equations}

In the next lemma we prove that $k^{[\bf n]}$ approximately satisfies the evolution equation \eqref{e0.k} in vacuum.
\begin{lemma} \label{lemmakn.approx.ev}
For every $n\ge1$ and every multi-index $\alpha$, there exists a $t_n>0$ such that the following estimates hold for all $(t,x)\in (0,t_n]\times [0,\delta]^3$ and $I,J$:
\begin{align}\label{kn.approx.eq}
| \partial_x^{\alpha}\big( \partial_t \en{k}_{IJ}- \en{k}_{CC} \en{k}_{IJ}-\en{R}_{IJ} \big)| \leq C_{\alpha,n} t^{-2+n\varepsilon}.
\end{align}
\end{lemma}
\begin{proof}
Plugging in \eqref{kIJ.it} we have
\begin{align*} 
\partial_t \en{k}_{IJ}- \en{k}_{CC} \en{k}_{IJ}-\en{R}_{IJ} = (\enmo{R}_{IJ}-\en{R}_{IJ})+ (\enmo{k}_{CC}-\en{k}_{CC})\en{k}_{IJ}. 
\end{align*} 
The desired bound follows by using Proposition \ref{prop:ind2} to control the previous RHS and its spatial derivatives.
\end{proof}
We will also need to compare $k^{[\bf n]}$ with the actual second fundamental form $\widetilde{k}^{[\bf n]}$ of the $t$ slices relative to ${\bf g}^{[\bf n]}$.
\begin{lemma}\label{lem:k.approx.2ndfund}
For every $n\ge1$ and every multi-index $\alpha$, there exists a $t_n>0$ such that the following estimate holds:
\begin{align}\label{k.approx.2ndfund}
| \partial_t^r\partial_x^{\alpha}(k^{[\bf n]}_{IJ}-\widetilde{k}^{[\bf n]}_{IJ}) | \leq C_{\alpha, n} t^{-1-r+n\varepsilon+|p_I-p_J|},\qquad r=0,1,
\end{align}
for all $(t,x) \in (0, t_n] \times [0,\delta]^3$ and indices $I,J$.
\end{lemma}
\begin{proof}
We note that $\widetilde{k}^{[\bf n]}$ satisfies 
\begin{align*}
\partial_t \eIan = \en{\widetilde{k}}_{IC}  \en{e}_{Ca} \quad\Rightarrow\quad \widetilde{k}^{[\bf n]}_{IJ}=\omega^{[\bf n]}_{aJ}\partial_te^{[\bf n]}_{Ia}.
\end{align*}
Plugging in \eqref{eIi.it}, we compute
\begin{align*}
\en{\widetilde{k}}_{IJ} =&\,\en{k}_{\uI  \uI} \delta_{\uI J}+\sum_{C\neq I}k^{[\bf n-1]}_{IC}e_{Ca}^{[\bf n-1]}\omega^{[\bf n]}_{aJ}\\
=&\, \en{k}_{\uI  \uI} \delta_{\uI J}+\sum_{C\neq I}k^{[\bf n-1]}_{IC}(e_{Ca}^{[\bf n-1]}-e_{Ca}^{[\bf n]})\omega^{[\bf n]}_{aJ}+\sum_{C\neq I}(k^{[\bf n-1]}_{IC}-k^{[\bf n]}_{IC})\delta_{CJ}+\sum_{C\neq I}k^{[\bf n]}_{IC}\delta_{CJ}\\
=&\,k^{[\bf n]}_{IJ}+\sum_{C\neq I}k^{[\bf n-1]}_{IC}(e_{Ca}^{[\bf n-1]}-e_{Ca}^{[\bf n]})\omega^{[\bf n]}_{aJ}+\sum_{C\neq I}(k^{[\bf n-1]}_{IC}-k^{[\bf n]}_{IC})\delta_{CJ}
\end{align*}
or equivalently
\begin{align}\label{kn-tilde.kn.eq}
 \en{\widetilde{k}}_{IJ}-k^{[\bf n]}_{IJ}=&\sum_{C\neq I}k^{[\bf n-1]}_{IC}(e_{Ca}^{[\bf n-1]}-e_{Ca}^{[\bf n]})\omega^{[\bf n]}_{aJ}+\sum_{C\neq I}(k^{[\bf n-1]}_{IC}-k^{[\bf n]}_{IC})\delta_{CJ}
\end{align}
The desired bound for $r=0$ follows by employing Proposition \ref{prop:ind2} and point \ref{item.approx1} of Theorem \ref{thm:approx.sol} to estimate the preceding RHS. For the case $r=1$, we take the time derivative of \eqref{kn-tilde.kn.eq} and plug in the equations \eqref{kIJ.it}, \eqref{dt.omegabC.it}, \eqref{enkenmokdiff}, \eqref{eIa.diff.eq} to replace all time derivatives in the RHS with terms that have already been controlled in Proposition \ref{prop:ind2}, and point \ref{item.approx1} of Theorem \ref{thm:approx.sol}. We omit the details.
\end{proof}
With the previous lemmas at our disposal, we are now ready to complete the proof of point \ref{item.approx3} of Theorem \ref{thm:approx.sol} for $\mu,\nu=I,J$. 
\begin{proposition}\label{prop:approx4}
For every $n\ge1$ and every multi-index $\alpha$, there exists a $t_n>0$ such that the following estimate hold:
\begin{align}\label{Ric4.n.est}
|\partial^\alpha_x {\bf R}_{IJ}^{{\bf [n]}}|\leq 
 C_{\alpha,n}t^{-2+n\varepsilon},
 \end{align}
for all $(t,x) \in (0, t_n] \times [0,\delta]^3$ and indices $I,J$.
\end{proposition}
\begin{proof}
The evolution equation \eqref{e0.k} holds true for any Lorentzian metric, hence, also for ${\bf g}^{[\bf n]}$. We thus have 
\begin{align*}
{\bf R}_{IJ}^{{\bf [n]}}=&\, R_{IJ}^{[\bf n]}-\partial_t\widetilde{k}_{IJ}^{[\bf n]}+\widetilde{k}_{CC}^{[\bf n]}\widetilde{k}^{[\bf n]}_{IJ}\\
=&\, R_{IJ}^{[\bf n]}-\partial_tk_{IJ}^{[\bf n]}+k_{CC}^{[\bf n]}k^{[\bf n]}_{IJ}\\
&-\partial_t(\widetilde{k}_{IJ}^{[\bf n]}-k_{IJ}^{[\bf n]})+(\widetilde{k}_{CC}^{[\bf n]}-k_{CC}^{[\bf n]})\widetilde{k}^{[\bf n]}_{IJ}
+k_{CC}^{[\bf n]}(\widetilde{k}^{[\bf n]}_{IJ}-k^{[\bf n]}_{IJ})
\end{align*}
The desired bound follows by using Lemmas \ref{lemmakn.approx.ev}, \ref{lem:k.approx.2ndfund}.
\end{proof}

\subsection{Approximate solution to the constraints}

The sequence of iterates ${\bf g}^{[\bf n]}$ we have constructed so far is an approximate solution to the evolutionary part of the Einstein vacuum equations, ie. \eqref{Ric4.n.est}. 
We will now proceed to show that for asymptotic data satisfying Definition \ref{defn1}, ${\bf g}^{[\bf n]}$ is also an approximate solution to the constraint equations. Let us denote by $\nabla^{[\bf n]},D^{[\bf n]}$ the connections intrinsic to $g^{[\bf n]},{\bf g}^{[\bf n]}$ respectively and let 
\begin{align}\label{gamma.n}
\gamma^{[\bf n]}_{IJB}={\bf g}^{[\bf n]}(D^{[\bf n]}_{e^{[\bf n]}_I}e^{[\bf n]}_J,e_B^{[\bf n]})=g^{[\bf n]}(\nabla^{[\bf n]}_{e^{[\bf n]}_I}e^{[\bf n]}_J,e_B^{[\bf n]}).
\end{align}
\begin{lemma}\label{lem:gamma}
For every multi-index and indices $I,J,B$, the $n$-th spatial connection coefficients satisfy
\begin{align*}
\notag |\partial^\alpha_x\gamma^{[\bf n]}_{IJB}|\leq&\,C_{\alpha,n}(t^{-p_I+|p_B-p_J|}+t^{-p_J+|p_B-p_I|}+t^{-p_B+|p_I-p_J|})|\log t|^{1+|\alpha|},\qquad I\neq J,B
\end{align*}
and
\begin{align*}
\big|\partial_x^\alpha \big[\gamma^{[\bf n]}_{\underline{I}J\underline{I}}+t^{-p_{\underline{J}}}E_{\underline{J}}\log(f_{\underline{I}\underline{I}}t^{-p_{\underline{I}}})-\sum_{J\leq a\leq I}h_{a\underline{I}}E_{\underline{I}}(t^{-p_{\underline{J}}}f_{\underline{J}a})\big]\big|\leq&\,C_{\alpha,n} t^{-p_J+\varepsilon},
\end{align*}
for all $(t,x)\in(0,t_n]\times[0,\delta]^3$, where $E_J=\sum_{b\ge J}f_{Jb}\partial_b$ and the last sum is not present when $J>I$. 
\end{lemma}
\begin{proof}
The first bound readily follows using the formula \eqref{gammaIJB} and Proposition \ref{prop:ind1}. For the second bound, employing \eqref{gammaIJB} and Proposition \ref{prop:ind1} once more we have
\begin{align*}
\gamma^{[\bf n]}_{\underline{I}J\underline{I}}=&-\omega_{\underline{I}\underline{I}}^{[\bf n]}e_J^{\bf [n]}e^{[\bf n]}_{\underline{I}\underline{I}}
+\omega_{a\underline{I}}^{[\bf n]}e^{\bf [n]}_{\underline{I}}e^{\bf [n]}_{Ja}\\
=&-t^{-p_{\underline{J}}}E_{\underline{J}}\log(f_{\underline{I}\underline{I}}t^{-p_{\underline{I}}})
+\sum_{J\leq a\leq I}t^{-p_{\underline I}}\omega^{[\bf 0]}_{a\underline{I}}E_{\underline{I}}(t^{-p_{\underline{J}}}f_{\underline{J}a})
+(\gamma^{[\bf n]}_{\underline{I}J\underline{I}})_{error}\\
=&-t^{-p_{\underline{J}}}E_{\underline{J}}\log(f_{\underline{I}\underline{I}}t^{-p_{\underline{I}}})
+\sum_{J\leq a\leq I}h_{a\underline{I}}E_{\underline{I}}(t^{-p_{\underline{J}}}f_{\underline{J}a})
+(\gamma^{[\bf n]}_{\underline{I}J\underline{I}})_{error},
\end{align*}
where $|\partial_x^\alpha(\gamma^{[\bf n]}_{\underline{I}J\underline{I}})_{error}|\leq C_{\alpha,n}t^{-p_J+\varepsilon}$. This completes the proof of the lemma.
\end{proof}
Next, we show that the constraints are satisfied to leading order.
\begin{lemma} \label{lem:constraint}
Let $p_I,f_{Ia}$ satisfy the algebraic Kasner conditions and the asymptotic differential conditions \eqref{frame.mom}. Then,
for every $n\in \mathbb{N}$ and every multi-index $\alpha$, there exists a $t_n>0$ such that the following estimates hold:
\begin{align}\label{approx.const1}
\begin{split}
\big| \partial_x^{\alpha}\big[R^{[\bf n]}-| \en{k}|^2 +(\mathrm{tr}  \en{k})^2\big]\big| \leq &\,C_{\alpha,n}t^{-2+\varepsilon}, \\
\big| \partial_x^{\alpha}\big[\en{\nabla}_J\en{k}_{IJ} -\en{e}_I\mathrm{tr}\en{k}\big]\big| \leq&\, C_{\alpha,n}t^{-1-p_I+\varepsilon}.
\end{split}
\end{align}
for all $(t,x)\in(0,t_n]\times[0,\delta]^3$ and index $I$.
\end{lemma}

\begin{proof}
From point \ref{item.approx2} of Theorem \ref{thm:approx.sol}, it follows that 
\begin{align*} 
| \partial_x^{\alpha}R^{[\bf n]}| \leq C_{\alpha,n} t^{-2+\varepsilon}.
\end{align*}
We also have
\begin{align*}
  |\en{k}|^2 -(\tr \en{k})^2 =&\,(\en{k}-k^{[\bf 0]})_{IJ} (\en{k}-k^{[\bf 0]})_{IJ}-( \text{tr}\en{k}-\text{tr}k^{[\bf 0]})^2\\ 
  &+2(\en{k}-k^{[\bf 0]})_{IJ} k^{[\bf 0]}_{JI}- 2({{k^{[\bf 0]}})_\ell}^{\ell}  {(\en{k}-k^{[\bf 0]})_\ell}^\ell\\
&+\frac{1}{t^2}\sum_i p_i^2-\frac{1}{t^2}\bigg(\sum_i p_i \bigg)^2
  \end{align*}
The last line cancels, by virtue of the Kasner algebraic conditions. From point \ref{item.approx1} of Theorem \ref{thm:approx.sol}, the remaining terms and their spatial derivatives are bounded by $C_{\alpha,n}t^{-2+\varepsilon}$. This gives the first bound in \eqref{approx.const1}.

For the second bound in \eqref{approx.const1}, we expand the momentum constraint:
\begin{align}\label{momconst.n.exp}
\notag&\en{\nabla}_J\en{k}_{IJ}-e_I^{[\bf n]}\tr \en{k}\\
\notag=&\, \en{e}_{Ja}\partial_a\en{k}_{IJ}-\eIan\partial_a(\tr\en{k}) -\en{\gamma}_{JIC}\en{k}_{CJ}-\en{\gamma}_{JJC}\en{k}_{IC}\\
\notag=&\, \en{e}_{Ja}\partial_a(\en{k}_{IJ}-k^{[\bf 0]}_{IJ})-\eIan\partial_a(\tr\en{k}-\text{tr}k^{[\bf 0]})\\
&-\en{\gamma}_{JIC}(\en{k}_{CJ}-k_{CJ}^{[\bf 0]})+\en{\gamma}_{JCJ}(\en{k}_{IC}-k^{[\bf 0]}_{IC})
 -(e_{\underline{I}a}^{[\bf n]}-e^{[\bf 0]}_{\underline{I}a})\frac{\partial_ap_{\underline{I}}}{t}\\
\notag&+\sum_J\bigg(\big[\gamma_{J\underline{I}J}^{[\bf n]}+t^{-p_{\underline{I}}}E_{\underline{I}}\log (f_{JJ}t^{-p_J})
-\sum_{I\leq a\leq J}t^{-p_{\underline J}}\omega^{[\bf 0]}_{a\underline{J}}E_{\underline{J}}(t^{-p_{\underline{I}}}f_{\underline{I}a})\big]\frac{p_J-p_{\underline{I}}}{t}\bigg)\\
\notag& -t^{-1-p_{\underline{I}}}E_{\underline{I}} p_{\underline{I}}
-\sum_Jt^{-1-p_{\underline{I}}}(p_J-p_{\underline{I}})\bigg[E_{\underline{I}}\log (f_{JJ}t^{-p_J})-\sum_{I\leq a\leq J}t^{p_{\underline{I}}-p_{ J}}\omega^{[\bf 0]}_{aJ}E_{J}(t^{-p_{\underline{I}}}f_{\underline{I}a})\bigg]
\end{align}
From point \ref{item.approx1} of Theorem \ref{thm:approx.sol} and Lemma \ref{lem:gamma}, we have 
\begin{align}\label{momconst.n.error}
&\notag\big|\partial_x^\alpha\big[\en{e}_{Ja}\partial_a(\en{k}_{IJ}-k^{[\bf 0]}_{IJ})-\eIan\partial_a(\tr\en{k}-\text{tr}k^{[\bf 0]})\big]\big|\\
\leq&\, C_{\alpha,n}(t^{-p_J}t^{-1+\varepsilon+|p_I-p_J|}+t^{-p_I}t^{-1+\varepsilon})\\
\notag\leq&\, C_{\alpha,n}t^{-1-p_I+\varepsilon},
\end{align} 
\begin{align}\label{momconst.n.error2}
\big|\partial_x^\alpha\big[(e_{\underline{I}a}^{[\bf n]}-e^{[\bf 0]}_{\underline{I}a})\frac{\partial_ap_{\underline{I}}}{t}
-\sum_J(\gamma_{J\underline{I}J}^{[\bf n]}+t^{-p_{\underline{I}}}E_{\underline{I}}\log (f_{JJ}t^{-p_J})\frac{p_J-p_{\underline{I}}}{t}\big]\big|\leq C_{\alpha,n}t^{-1-p_I+\varepsilon},
\end{align} 
and
\begin{align}\label{momconst.n.error3}
\notag&\big|\partial_x^\alpha\big[\en{\gamma}_{JIC}(\en{k}_{CJ}-k_{CJ}^{[\bf 0]})+\en{\gamma}_{JJC}(\en{k}_{IC}-k^{[\bf 0]}_{IC})\big]\big|\\
\leq&\,C_{\alpha,n}t^{-1+\varepsilon+|p_C-p_J|}(t^{-p_I+|p_C-p_J|}+t^{-p_J+|p_C-p_I|}+t^{-p_C+|p_I-p_J|})|\log t|^{1+|\alpha|})\\
\notag&+C_{\alpha,n}t^{-1+\varepsilon+|p_I-p_C|}t^{-p_C}\\
\notag\leq &\,C_{\alpha,n}t^{-1-p_I+\varepsilon}
\end{align} 
Finally, we observe that 
\begin{align}\label{momconst.n.canc}
\notag& t^{-1-p_{\underline{I}}}E_{\underline{I}} p_{\underline{I}}
+\sum_Jt^{-1-p_{\underline{I}}}(p_J-p_{\underline{I}})\bigg[E_{\underline{I}}\log (f_{JJ}t^{-p_J})-\sum_{I\leq a\leq J}t^{p_{\underline{I}}-p_{ J}}\omega^{[\bf 0]}_{aJ}E_{J}(t^{-p_{\underline{I}}}f_{\underline{I}a})\bigg]\\
=&\,t^{-1-p_{\underline{I}}}\bigg[E_{\underline{I}} p_{\underline{I}}+\sum_J(p_J-p_{\underline{I}})E_{\underline{I}}\log (f_{JJ})-\log t\sum_J(p_J-p_{\underline{I}})E_{\underline{I}}p_J\\
\notag&-\sum_J\sum_{I\leq a\leq J}(p_J-p_{\underline{I}})h_{aJ}E_Jf_{\underline{I}a}+\log t\sum_J(p_J-p_{\underline{I}})\delta_{\underline{I}J}E_Jp_{\underline{I}}\bigg]\\
\notag=&\,t^{-1-p_{\underline{I}}}\bigg[E_{\underline{I}} p_{\underline{I}}+\sum_J(p_J-p_{\underline{I}})E_{\underline{I}}\log (f_{JJ})-\sum_J\sum_{I\leq a\leq J}(p_J-p_{\underline{I}})h_{aJ}E_Jf_{\underline{I}a}\bigg]
\end{align}
where in the last equality, we used the algebraic Kasner conditions. The last line vanishes by assumption \eqref{frame.mom}, see also Lemma \ref{lem:frame.mom}. Combining \eqref{momconst.n.exp}-\eqref{momconst.n.canc}, we deduce the second inequality in \eqref{approx.const1}. The proof of the lemma is complete.
\end{proof}

\begin{proposition}\label{prop:approx5}
Assume that the conclusion of Lemma \ref{lem:constraint} holds true. Then for every $n\ge1$ the spacetime Ricci components of ${\bf g}^{\bf [n]}$  satisfy:
\begin{align}\label{Ric4.n.est.const}
|\partial_x^\alpha {\bf R}_{0\nu}^{[\bf n]}|\leq C_{\alpha,n}t^{-2+n\varepsilon},
\end{align}
for every multi-index $\alpha$ and $\nu=0,I$. 
\end{proposition}
\begin{proof}
First, we recall the Gauss and Codazzi equations for the metric ${\bf g}^{[\bf n]}$:
\begin{align}
\label{Gauss.n}
R(g^{[\bf n]})-|\widetilde{k}^{[\bf n]}|^2+(\text{tr}\tilde{k}^{[\bf n]})^2=&\,{\bf R}^{[\bf n]}+2{\bf R}_{00}^{[\bf n]}\\
\label{Codazzi.n}\nabla^{[\bf n]}_J\widetilde{k}^{[\bf n]}_{IJ}-e^{[\bf n]}_I\text{tr}\widetilde{k}^{[\bf n]}=&\,{\bf R}^{[\bf n]}_{0I}
\end{align}
By Lemma \ref{lem:constraint} and Proposition \ref{prop:approx4} we deduce that 
\begin{align}\label{R0I.R00.est}
|\partial_x^\alpha {\bf R}_{0I}^{[\bf n]}|\leq C_{\alpha,n}t^{-1+p_I+\varepsilon}\qquad|\partial_x^\alpha {\bf R}_{00}^{[\bf n]}|\leq C_{\alpha,n}t^{-2+\varepsilon}.
\end{align}
The twice contracted second Bianchi identity reads:
\begin{align}\label{2nd.Bianchi}
D^{[\bf n]}_0{\bf R}_{0I}^{[\bf n]}=D^{[\bf n]}_C{\bf R}_{CI}^{[\bf n]}-\frac{1}{2}D^{[\bf n]}_I{\bf R}^{[\bf n]},\qquad D^{[\bf n]}_0{\bf R}_{00}^{[\bf n]}=D^{[\bf n]}_C{\bf R}_{C0}^{[\bf n]}-\frac{1}{2}D^{[\bf n]}_0{\bf R}^{[\bf n]}.
\end{align}
Expanding the covariant derivatives and subtracting the background variables, we rewrite \eqref{2nd.Bianchi} in the form
\begin{align}
\label{R0I.eq}
\partial_t{\bf R}_{0I}^{[\bf n]}-(\text{tr}\widetilde{k}^{[\bf n]}+\widetilde{k}^{[\bf n]}_{\underline{I}\underline{I}}){\bf R}_{0\underline{I}}^{[\bf n]}=&\,\frac{1}{2}e_I^{[\bf n]}{\bf R}_{00}^{\bf [n]}+e_C^{[\bf n]}{\bf R}_{CI}^{[\bf n]}-\frac{1}{2}e_I^{[\bf n]}{\bf R}^{[\bf n]}_{CC}\\
\notag&+\gamma_{CJC}^{[\bf n]}{\bf R}_{CI}^{[\bf n]}-\gamma_{CID}^{[\bf n]}{\bf R}_{CD}^{[\bf n]}+\sum_{C\neq I}\widetilde{k}^{[\bf n]}_{CI}{\bf R}_{0C}^{[\bf n]}\\
\label{R00.eq}\partial_t{\bf R}_{00}^{[\bf n]}-2\text{tr}\widetilde{k}^{[\bf n]}{\bf R}_{00}^{[\bf n]}=&\,2e^{[\bf n]}_C{\bf R}_{0C}^{[\bf n]}-\partial_t{\bf R}^{[\bf n]}_{CC}+\gamma_{CJC}^{[\bf n]}{\bf R}_{0J}^{[\bf n]}-\widetilde{k}^{[\bf n]}_{CD}{\bf R}^{[\bf n]}_{CD}
\end{align}
Subtracting the iterates $k^{[\bf n]},k^{[\bf 0]}$, we are left with the system
\begin{align}
\label{R0I.eq2}
\partial_t{\bf R}_{0I}^{[\bf n]}+\frac{1+p_{\underline{I}}}{t}{\bf R}_{0\underline{I}}^{[\bf n]}=&\,\frac{1}{2}e_I^{[\bf n]}{\bf R}_{00}^{\bf [n]}+\sum_{C\neq I}\widetilde{k}^{[\bf n]}_{CI}{\bf R}_{0C}^{[\bf n]}+F_I^{[\bf n]}\\
\notag&+(\text{tr}\widetilde{k}^{[\bf n]}-\text{tr}k^{[\bf n]}+\text{tr}k^{[\bf n]}-\text{tr}k^{[\bf 0]}+\widetilde{k}^{[\bf n]}_{\underline{I}\underline{I}}-k^{[\bf n]}_{\underline{I}\underline{I}}+k^{[\bf n]}_{\underline{I}\underline{I}}-{k^{[\bf 0]}_{\underline{I}\underline{I}}}){\bf R}_{0\underline{I}}^{[\bf n]}\\
\label{R00.eq2}\partial_t{\bf R}_{00}^{[\bf n]}+\frac{2}{t}{\bf R}_{00}^{[\bf n]}=&\,2e^{[\bf n]}_C{\bf R}_{0C}^{[\bf n]}+\gamma_{CJC}^{[\bf n]}{\bf R}_{0J}^{[\bf n]}+F_0^{[\bf n]}\\
\notag&+2(\text{tr}\widetilde{k}^{[\bf n]}-\text{tr}k^{[\bf n]}+\text{tr}k^{[\bf n]}-\text{tr}k^{[\bf 0]}){\bf R}_{00}^{[\bf n]}
\end{align}
where 
\begin{align}
\label{FI}F^{[\bf n]}_I=&\,e_C^{[\bf n]}{\bf R}_{CI}^{[\bf n]}-\frac{1}{2}e_I^{[\bf n]}{\bf R}^{[\bf n]}_{CC}
+\gamma_{CJC}^{[\bf n]}{\bf R}_{CI}^{[\bf n]}-\gamma_{CID}^{[\bf n]}{\bf R}_{CD}^{[\bf n]}\\
\label{F0}F^{[\bf n]}_0=&-\partial_t{\bf R}^{[\bf n]}_{CC}-\widetilde{k}^{[\bf n]}_{CD}{\bf R}^{[\bf n]}_{CD}
\end{align}
Using integrating factors and \eqref{R0I.R00.est}, we obtain the formulas 
\begin{align}
\label{R0I.eq3}
{\bf R}_{0I}^{[\bf n]}=&\,t^{-1-p_{\underline{I}}}\int^t_0\tau^{1+p_{\underline{I}}}\bigg[\frac{1}{2}e_{\underline{I}}^{[\bf n]}{\bf R}_{00}^{\bf [n]}+\sum_{C\neq I}\widetilde{k}^{[\bf n]}_{C\underline{I}}{\bf R}_{0C}^{[\bf n]}+F_{\underline{I}}^{[\bf n]}\\
\notag&+(\text{tr}\widetilde{k}^{[\bf n]}-\text{tr}k^{[\bf n]}+\text{tr}k^{[\bf n]}-\text{tr}k^{[\bf 0]}+\widetilde{k}^{[\bf n]}_{\underline{I}\underline{I}}-k^{[\bf n]}_{\underline{I}\underline{I}}+k^{[\bf n]}_{\underline{I}\underline{I}}-\widetilde{k}^{[\bf 0]}_{\underline{I}\underline{I}}){\bf R}_{0\underline{I}}^{[\bf n]}\bigg]d\tau\\
\label{R00.eq3}{\bf R}_{00}^{[\bf n]}=&\,t^{-2}\int^t_0\tau^2\bigg[2e^{[\bf n]}_C{\bf R}_{0C}^{[\bf n]}+\gamma_{CJC}^{[\bf n]}{\bf R}_{0J}^{[\bf n]}+F_0^{[\bf n]}\\
\notag&+2(\text{tr}\widetilde{k}^{[\bf n]}-\text{tr}k^{[\bf n]}+\text{tr}k^{[\bf n]}-\text{tr}k^{[\bf 0]}){\bf R}_{00}^{[\bf n]}\bigg]d\tau
\end{align}
By Lemma \ref{lem:k.approx.2ndfund}, Proposition \ref{prop:approx4} and Lemma \ref{lem:gamma}, we deduce the estimates
\begin{align}
\label{FI.est}
\bigg|\partial_x^\alpha\bigg(t^{-1-p_I}\int^t_0\tau^{1+p_I}F_I^{[\bf n]} d\tau\bigg)\bigg|\leq&\, C_{\alpha,n}t^{-2+(n+1)\varepsilon}|\log t|^{|\alpha|}\\
\notag\leq&\, C_{\alpha,n}t^{-2+n\varepsilon}\\
\label{F0.est}\bigg|\partial_x^\alpha\bigg(t^{-2}\int^t_0\tau^{2}F_0^{\bf[n]}d\tau\bigg)\bigg|=&\,\bigg|\partial_x^\alpha\bigg(t^{-2}\int^t_0(2\tau{\bf R}^{[\bf n]}_{CC}-\tau^2\widetilde{k}^{[\bf n]}_{CD}{\bf R}^{[\bf n]}_{CD})d\tau-{\bf R}^{[\bf n]}_{CC}\bigg)\bigg|\\
\notag\leq&\,C_{\alpha,n}t^{-2+n\varepsilon} 
\end{align}
Applying \eqref{R0I.R00.est}, \eqref{FI.est}, \eqref{F0.est} to \eqref{R00.eq3} gives the bound 
\begin{align}\label{R00.est2}
|\partial_x^\alpha {\bf R}_{00}^{[\bf n]}|\leq C_{\alpha,n}t^{-2+2\varepsilon}
+C_{\alpha,n}t^{-2+n\varepsilon}.
\end{align}
Using the latter, along with \eqref{R0I.R00.est}, \eqref{FI.est}, \eqref{F0.est}, and Lemma \ref{k.approx.2ndfund}, we obtain 
\begin{align}\label{R0I.est2}
|\partial_x^\alpha {\bf R}_{0I}^{[\bf n]}|\leq C_{\alpha,n}t^{-2+2\varepsilon}+C_{\alpha,n}t^{-2+n\varepsilon}.
\end{align}
Note that for $n>1$, the bounds \eqref{R00.est2}, \eqref{R0I.est2} are an improvement by $t^\varepsilon$ over \eqref{R0I.R00.est}. Repeating the above argument $n-2$ times, using \eqref{R00.est2}, \eqref{R0I.est2}, instead of \eqref{R0I.R00.est} and so on and so forth, yields the desired \eqref{Ric4.n.est.const}.
\end{proof}

\section{Construction of an actual solution to the modified evolution equations}\label{sec:actual.sol}

In this section, we carry out a localized construction of a singular solution $e_{Ia},k_{IJ},\gamma_{IJB}$ to the set of equations \eqref{e0.eIi2}, \eqref{e0.kIJ.mod2}, \eqref{e0.gamma_IJB.mod2}, whose behavior matches that of the iterates $e^{[\bf n]}_{Ia},k^{[\bf n]}_{IJ},\gamma^{[\bf n]}_{IJB}$ to a sufficiently large polynomial order, as $t\rightarrow0$.

\subsection{Domain of definition}\label{subsec:domain}

The domain on which the actual solution will be living is defined using the metric ${\bf g}^{[\bf n]}$, for some sufficiently large $n$ which is fixed in the end. 
Let $\Sigma_t$ denote the level sets of $t$ in $[0,T]\times[0,\delta]^3$. Denote the boundary of the initial slice by $S_0=\partial\Sigma_0=\cup_{a=1}^3S_{0,a}^\pm$, where $S_{0,a}^+=\{x^a=0\}$, $S_{0,a}^-=\{x^a=\delta\}$.
We then define $\mathcal{H}=\cup_{a=1}^3\mathcal{H}_a^\pm$, where $\mathcal{H}^\pm_a$ are the ingoing hypersurfaces obtained by flowing each side $S_{0,a}^\pm$ of the cube $S_0$ through the corresponding vector field
\begin{align}\label{flow.vec}
X_a^\pm=\partial_t\pm\sigma\frac{\nabla^{[\bf n]}x^a}{|\nabla^{[\bf n]}x^a|_{g^{\bf [n]}}},\qquad \nabla^{[\bf n]}x^a=e_{Ia}^{[\bf n]}e_I^{[\bf n]}=(g^{[\bf n]})^{ia}\partial_i,\qquad\sigma>0.
\end{align}
For simplicity, we suppress the index $n$ in the notation for the domain and its geometry. Let $U_t\subset\Sigma_t$ denote the slices whose boundary  $S_t=\cup_{a=1}^3S_{t,a}^\pm=\partial U_t$ is in turn the slicing of $\mathcal{H}$ induced by the above flow, $t\in[0,T]$. 
\begin{remark}\label{rem:sigma}
The constant $\sigma$ is chosen such that $X_a^\pm$ is spacelike and sufficiently ingoing. It will be fixed below in order to absorb any energy flux terms in the estimates coming from $\mathcal{H}$, see the proof of Proposition \ref{prop:en.est} and \eqref{choice.sigma}. 
\end{remark}
\begin{lemma}\label{lem:domain}
The flow of $X^\pm_a$ is well defined in $[0,T]$ and  
the hypersurface $\mathcal{H}$ is spacelike. The coordinate functions $x^b$, $b\neq a$, induce coordinate vector fields $\slashed{\partial}_b$ on $TS^\pm_{t,a}$ that satisfy
\begin{align}\label{partial.St}
\slashed{\partial}_b=f_{bc}^{a,\pm}\partial_c,\qquad |\partial_x^\alpha (f_{bc}^{a,\pm}-\delta_{bc})|\leq C_{\alpha,n}t^{1+p_a-2p_{\min\{a,c\}}}|\log t|^{|\alpha|+1}
\end{align}
In these coordinates, the induced metric $g_{\slashed{b}\slashed{c}}^{[\bf n]}=g^{[\bf n]}(\slashed{\partial}_b,\slashed{\partial}_c)$ on $S_{t,a}^\pm$ satisfies 
\begin{align}\label{metric.St}
|\partial_x^\alpha (g_{\slashed{b}\slashed{c}}^{[\bf n]}-g_{bc}^{[\bf n]})|\leq C_{\alpha,n}t^{2p_{\max\{b,c\}}+\varepsilon},\qquad |\partial_x^\alpha[ (g^{[\bf n]})^{\slashed{b}\slashed{c}}-(g^{[\bf n]})^{bc}]|\leq C_{\alpha,n}t^{-2p_{\min\{b,c\}}+\varepsilon}
\end{align}
Moreover, the inward unit normal to $S_{t,a}^\pm$ in $U_t$ is a perturbation of
$\pm\frac{\nabla^{[\bf n]}x^a}{|\nabla^{[\bf n]}x^a|_{g^{\bf [n]}}}=\pm \frac{e_{Ia}^{[\bf n]}e_I^{[\bf n]}}{\sqrt{e_{C\underline a}^{[\bf n]}e_{C\underline a}^{[\bf n]}}}$:
\begin{align}\label{n.St}
n_{S_{t,a}^\pm}=n^I_{S_{t,a}^\pm}e_I^{[\bf n]},\qquad \big|\partial_x^\alpha\big[n^I_{S_{t,a}^\pm}\mp  \frac{e_{Ia}^{[\bf n]}}{\sqrt{e_{C\underline a}^{[\bf n]}e_{C\underline a}^{[\bf n]}}}\big]\big|\leq C_{\alpha,n}t^{|p_a-p_I|+\varepsilon}.
\end{align}
The future unit normal to $S_{t,a}^\pm$ within $\mathcal{H}$ is a perturbation of\\ $X_a^\pm=(\sigma^2-1)^{-\frac{1}{2}}\big[\partial_t\pm\sigma (e_{C\underline a}^{[\bf n]}e_{C\underline a}^{[\bf n]})^{-\frac{1}{2}}e_{Ia}^{[\bf n]}e_I^{[\bf n]}\big]$:
\begin{align}\label{n.St.H}
\begin{split}
n_{S_{t,a}^\pm}^{\mathcal{H}}&\,=(n_{S_{t,a}^\pm}^{\mathcal{H}})^0\partial_t+(n_{S_{t,a}^\pm}^{\mathcal{H}})^I e_I^{[\bf n]},\qquad
|\partial_x^\alpha[(n_{S_{t,a}^\pm}^{\mathcal{H}})^0-(\sigma^2-1)^{-\frac{1}{2}}]|\leq C_{\alpha,n}t^\varepsilon,\\
&\big|\partial_x^\alpha\big[(n_{S_{t,a}^\pm}^{\mathcal{H}})^I\mp(\sigma^2-1)^{-\frac{1}{2}}\sigma (e_{C\underline a}^{[\bf n]}e_{C\underline a}^{[\bf n]})^{-\frac{1}{2}}e_{Ia}^{[\bf n]}\big]\big|\leq C_{\alpha,n}t^{|p_a-p_I|+\varepsilon}.
\end{split}
\end{align}
The future unit normal to $\mathcal{H}^\pm_a$ is a perturbation of $(\sigma^2-1)^{-\frac{1}{2}}\big[\sigma\partial_t\pm (e_{C\underline a}^{[\bf n]}e_{C\underline a}^{[\bf n]})^{-\frac{1}{2}}e_{Ia}^{[\bf n]}e_I^{[\bf n]}\big]$:
\begin{align}\label{n.H}
\begin{split}
n_{\mathcal{H}^\pm_a}&\,=n_{\mathcal{H}^\pm_a}^0\partial_t+n_{\mathcal{H}^\pm_a}^Ie_I^{[\bf n]},\qquad |\partial_x^\alpha[n_{\mathcal{H}^\pm_a}^0- \sigma(\sigma^2-1)^{-\frac{1}{2}} ]|\leq C_{\alpha,n}t^\varepsilon,\\
&\big|\partial_x^\alpha\big[n_{\mathcal{H}^\pm_a}^I\mp  (\sigma^2-1)^{-\frac{1}{2}} (e_{C\underline a}^{[\bf n]}e_{C\underline a}^{[\bf n]})^{-\frac{1}{2}}e_{Ia}^{[\bf n]}\big]\big|\leq C_{\alpha,n}t^{|p_a-p_I|+\varepsilon}.
\end{split}
\end{align}
The induced volume forms on $S_{t,a}^\pm,\mathcal{H}_a^\pm$ satisfy:
\begin{align}\label{vol.St}
\big|\partial_x^\alpha\big[\frac{\mathrm{vol}_{S_{t,a}^\pm}}{\slashed{d}x^b\slashed{d}x^c}-\sqrt{c_{bb}c_{cc}}t^{p_b+p_c}\big]\big|\leq C_{\alpha,n} t^\varepsilon,\qquad\big|\partial_x^\alpha\big[\frac{\mathrm{vol}_{\mathcal{H}^\pm_a}}{dt \mathrm{vol}_{S_{t,a}^\pm}}-(\sigma^2-1)^{-\frac{1}{2}}\big]\big|\leq C_{\alpha,n}t^\varepsilon,
\end{align}
where $\slashed{d}x^b,\slashed{d}x^c$ are the 1-forms dual to $\slashed{\partial}_b,\slashed{\partial}_c$, $b,c\neq a$.
\end{lemma}
\begin{proof}
The tangential vector fields to $S^\pm_{0,a}$ are $\{\partial_b\}_{b\neq a}$. Lie propagating them along $X_a^\pm$ gives $\slashed{\partial}_b\in TS_{t,a}^\pm$, $b\neq a$:
\begin{align}\label{fbc.eq}
[X_a^\pm,f_{bc}^{a,\pm}\partial_c]=0\quad\Rightarrow\quad \bigg(\partial_t \pm\sigma\frac{e^{[\bf n]}_{Ia}e^{[\bf n]}_{Ii}}{\sqrt{e^{[\bf n]}_{I\underline{a}}e^{[\bf n]}_{I\underline{a}}}}\partial_i\bigg)f_{bc}^{a,\pm}\mp \sigma f_{bi}^{a,\pm}\partial_i \bigg(\frac{e^{[\bf n]}_{Ia}e^{[\bf n]}_{Ic}}{\sqrt{e^{[\bf n]}_{I\underline{a}}e^{[\bf n]}_{I\underline{a}}}}\bigg)=0.
\end{align}
From point \ref{item.approx1} of Theorem \ref{thm:approx.sol}, we observe that the coefficients in the previous equation and their $\partial_x^\alpha$ derivatives satisfy the bound:
\begin{align}\label{fbc.eq.coeff}
\bigg|\partial_x^\alpha\bigg(\frac{e^{[\bf n]}_{Ia}e^{[\bf n]}_{Ic}}{\sqrt{e^{[\bf n]}_{I\underline{a}}e^{[\bf n]}_{I\underline{a}}}}\bigg)\bigg|\leq C_{\alpha,n}t^{p_a-2\min\{p_a,p_c\}}|\log t|^{|\alpha|}
\end{align}
In particular, the coefficients are bounded by $Ct^{-1+\varepsilon}$, that is to say, they are uniformly integrable in $[0,t]$. Since \eqref{fbc.eq} is a transport equation in $X_a^\pm$, for $f_{bc}^{a,\pm}$ with initial data $\delta_{bc}$ at $t=0$, the flow is well-defined in $[0,T]$ and $f_{bc}^{a,\pm}$ satisfy \eqref{partial.St}. 

The estimate \eqref{metric.St} for the induced metric components on $S_{t,a}^\pm$ follows from \eqref{partial.St} and \eqref{gij.0}. 

For the inward unit normal $n_{S_{t,a}^\pm}$ to $S_{t,a}^\pm$ in $U_t$, we first subtract from $\pm\frac{\nabla^{[\bf n]}x^a}{|\nabla^{[\bf n]}x^a|_{g^{\bf [n]}}}$ its projections to $\slashed{\partial}_b$, $b\neq a$, to obtain (without summing in $a$)
\begin{align}\label{n.St.calc}
\notag V_{S_{t,a}^\pm}=&\pm\frac{\nabla^{[\bf n]}x^a}{|\nabla^{[\bf n]}x^a|_{g^{\bf [n]}}}\mp\sum_{b\neq a}\frac{g^{[\bf n]}(\nabla^{[\bf n]}x^a,\slashed{\partial}_b)}{|\nabla^{[\bf n]}x^a|_{g^{\bf [n]}}g^{[\bf n]}_{\slashed{b}\slashed{b}}}\slashed{\partial}_b\\
=&\pm [(g^{\bf n]})^{aa}]^{-\frac{1}{2}}(g^{\bf [n]})^{ia}\partial_i\mp  \sum_{b\neq a}[(g^{\bf n]})^{aa}]^{-\frac{1}{2}}(g^{[\bf n]}_{\slashed{b}\slashed{b}})^{-1}f_{ba}^{a,\pm}f_{bi}^{a,\pm}\partial_i
\end{align}
From the estimates \eqref{partial.St}-\eqref{metric.St} we have that 
\begin{align}\label{n.St.error}
\notag&|\partial_x^\alpha \{[(g^{\bf n]})^{aa}]^{-\frac{1}{2}}(g^{[\bf n]}_{\slashed{b}\slashed{b}})^{-1}f_{ba}^{a,\pm}f_{bi}^{a,\pm}\}|\\
\leq&\, C_{\alpha,n}t^{p_a-2p_b+1+p_a-2p_{\min\{a,b\}}+1}t^{p_a-2p_{\min\{a,i\}}}|\log t|^{|\alpha|+1}\\
\notag\leq&\, C_{\alpha,n}t^{p_a-2p_{\min\{a,i\}}+\varepsilon},
\end{align}
since $2-2p_b+2p_a-2p_{\min\{a,b\}}\ge 2\varepsilon$. Recall that $[(g^{\bf n]})^{aa}]^{-\frac{1}{2}}(g^{\bf [n]})^{ia}\sim t^{p_a-2p_{\min\{a,i\}}}$. Hence, the first term in the RHS of \eqref{n.St.calc} is of leading order. Rewrite
\begin{align*}
[(g^{\bf n]})^{aa}]^{-\frac{1}{2}}(g^{\bf [n]})^{ia}\partial_i=\frac{e_{Ia}^{[\bf n]}e_{Ii}^{[\bf n]}}{\sqrt{e_{C\underline a}^{[\bf n]}e_{C\underline a}^{[\bf n]}}}\omega^{[\bf n]}_{iD}e_D^{[\bf n]}=\frac{e_{Ia}^{[\bf n]}}{\sqrt{e_{C\underline a}^{[\bf n]}e_{C\underline a}^{[\bf n]}}}e_I^{[\bf n]},\qquad \frac{e_{Ia}^{[\bf n]}}{\sqrt{e_{C\underline a}^{[\bf n]}e_{C\underline a}^{[\bf n]}}}\sim t^{|p_a-p_I|}
\end{align*}
To conclude the desired estimate \eqref{n.St}, we observe that $n_{S_{t,a}^\pm}= V_{S_{t,a}^\pm}/|V_{S_{t,a}^\pm}|_{g^{[\bf n]}}$ and that $V_{S_{t,a}^\pm}$ is unit to leading order as $t\rightarrow0$.

The estimate \eqref{n.St.H} for $n_{S_{t,a}^\pm}^{\mathcal{H}}$ is obtained similarly, by subtracting from $X_a^\pm$ its projections to $\slashed{\partial}_b$, $b\neq a$, and normalizing the resulting vector field. In turn, \eqref{n.H} is immediate from \eqref{n.St.H}. Finally, the bounds for volume forms are implied by \eqref{metric.St} and the form of the normal $n^{\mathcal{H}}_{S_{t,a}^\pm}$.
\end{proof}
%
%

%
\subsection{The modified system of equations for the remainder terms}\label{subsec:rem.eq} 

We now consider the system of equations \eqref{e0.e}, \eqref{e0.kIJ.mod}, \eqref{e0.gamma_IJB.mod}, dropping all $R_{\mu\nu}^{(4)}$ terms, since we are working in vacuum:
\begin{align}
\label{e0.eIi2}\partial_te_{Ia}=&\,k_{IC}e_{Ca},\\
\label{e0.kIJ.mod2}\partial_tk_{IJ}-\text{tr}k\,k_{IJ}=&\,\frac{1}{2}\bigg[e_C\gamma_{IJC}-e_I\gamma_{CJC}
+e_C\gamma_{JIC}-e_J\gamma_{CIC}\\
\notag&-\gamma_{CID}\gamma_{DJC}-\gamma_{CCD}\gamma_{IJD}
-\gamma_{CJD}\gamma_{DIC}-\gamma_{CCD}\gamma_{JID}\bigg]\\
\notag&+\frac{1}{2}\delta_{IJ}\bigg[2e_D\gamma_{CDC}+\gamma_{CED}\gamma_{DEC}
+\gamma_{CCD}\gamma_{EED}+k_{CD}k_{CD}-(k_{CC})^2\bigg],\\
\label{e0.gamma_IJB.mod2}\partial_t\gamma_{IJB}-k_{IC}\gamma_{CJB}=&\,e_Bk_{JI}-e_Jk_{BI}\\
\notag&+\gamma_{JBC}k_{CI}+\gamma_{JIC}k_{BC}-\gamma_{BJC}k_{CI}-\gamma_{BIC}k_{JC}\\
\notag&-\delta_{IB}\bigg[e_Ck_{CJ}-\gamma_{CCD}k_{DJ}-\gamma_{CJD}k_{CD}-e_J\text{tr}k\bigg]\\
\notag&+\delta_{IJ}\bigg[e_Ck_{CB}-\gamma_{CCD}k_{DB}-\gamma_{CBD}k_{CD}-e_B\text{tr}k\bigg].
\end{align}
We would like to produce a solution $e_{Ia},k_{IJ},\gamma_{IJB}$, each component of which is equal to the corresponding $e_{Ia}^{[\bf n]},k^{[\bf n]}_{IJ},\gamma^{[\bf n]}_{IJB}$ plus a sufficiently decaying term, as $t\rightarrow0$. 

For this purpose, define the remainder terms
\begin{align}\label{e.k.gamma.d}
e^{(d)}_{Ii}=e_{Ii}-e^{\bf [n]}_{Ii},\qquad
\gamma^{(d)}_{IJB}=\gamma_{IJB}-\gamma_{IJB}^{[\bf n]},\qquad
k^{(d)}_{IJ}=k_{IJ}-k^{[\bf n]}_{IJ}.
\end{align}
Next, we plug \eqref{e.k.gamma.d} into \eqref{e0.eIi2}-\eqref{e0.kIJ.mod2} 
to obtain the system of equations satisfied by the differences $e^{(d)}_{Ii},\gamma^{(d)}_{IJB},k^{(d)}_{IJ}$. We use in the process the version of the equations \eqref{e0.kIJ.mod}, \eqref{e0.gamma_IJB.mod} for $\widetilde{k}^{[\bf n]}_{IJ},\gamma_{IJB}^{[\bf n]}$ and \eqref{eIi.it} to replace the terms which only contain iterates. A tedious, but straightforward, computation gives the equations:
\begin{align}
\label{e0.eI.d}\partial_te_{Ia}^{(d)}+\frac{p_{\underline{I}}}{t}e^{(d)}_{\underline{I}a}=&\,
(k_{IC}^{[\bf n]}-k^{[\bf 0]}_{IC})e^{(d)}_{Ca}
+k_{IC}^{(d)}e^{[\bf n]}_{Ca}+k_{IC}^{(d)}e^{(d)}_{Ca}+(\mathcal{I}^{[\bf n]}_e)_{Ia},\\
\label{e0.kIJ.d}\partial_tk_{IJ}^{(d)}+\frac{1}{t}k^{(d)}_{IJ}
=&\,\frac{1}{2}\big\{e_C\gamma_{IJC}^{(d)}-e_I\gamma_{CJC}^{(d)}
+e_C\gamma_{JIC}^{(d)}-e_J\gamma_{CIC}^{(d)}+2\delta_{IJ}e_D\gamma_{CDC}^{(d)}\big\}\\
\notag&-\delta_{\underline{I}J}\frac{p_{\underline{I}}}{t}k^{(d)}_{CC}+\delta_{IJ}\frac{1}{t}k^{(d)}_{CC}-\delta_{IJ}\sum_C\frac{p_C}{t}k_{CC}^{(d)}
+\mathfrak{K}_{IJ}^{(d)}+(\mathcal{I}_k^{[\bf n]})_{IJ},\\
\label{e0.gamma_IJB.d} \partial_t\gamma_{IJB}^{(d)}+\frac{p_{\underline{I}}}{t}\gamma_{\underline{I}JB}^{(d)}=&\,e_Bk_{JI}^{(d)}-e_Jk_{BI}^{(d)}
-\delta_{IB}\big[e_Ck_{CJ}^{(d)}-e_Jk_{CC}^{(d)}\big]+\delta_{IJ}\big[e_Ck_{CB}^{(d)}-e_Bk^{(d)}_{CC}\big]\\
\notag&-\frac{p_{\underline{I}}}{t}\gamma_{JB\underline{I}}^{(d)}-\frac{p_{\underline{B}}}{t}\gamma_{JI\underline{B}}^{(d)}+\frac{p_{\underline{I}}}{t}\gamma_{BJ\underline{I}}^{(d)}+\frac{p_{\underline{J}}}{t}\gamma_{BI\underline{J}}^{(d)}\\
\notag&-\delta_{IB}\bigg[\frac{p_{\underline{J}}}{t}\gamma_{CCJ}^{(d)}+\sum_C\frac{p_C}{t}\gamma_{CJC}^{(d)}\bigg]
+\delta_{IJ}\bigg[\frac{p_{\underline{B}}}{t}\gamma_{CC\underline{B}}^{(d)}+\sum_C\frac{p_C}{t}\gamma_{CBC}^{(d)}\bigg]\\
\notag&-\delta_{\underline{I}J}\frac{\partial_ap_{\underline{I}}}{t}e^{(d)}_{Ba}
+\delta_{\underline{I}B}\frac{\partial_ap_{\underline{I}}}{t}e_{Ja}^{(d)}
+\delta_{IB}\frac{\partial_ap_{\underline{J}}}{t}e_{\underline{J}a}^{(d)}-\delta_{IJ}\frac{\partial_ap_{\underline{B}}}{t}e_{\underline{B}a}^{(d)}\\
\notag&+\mathfrak{G}_{IJB}^{(d)}+(\mathcal{I}_\gamma^{[\bf n]})_{IJB},
\end{align}
where
\begin{align}
\label{frak.K}\mathfrak{K}_{IJ}^{(d)}=&\,\frac{1}{2}\bigg[e_{Ca}^{(d)}\partial_a\gamma_{IJC}^{[\bf n]}-e_{Ia}^{(d)}\partial_a\gamma_{CJC}^{[\bf n]}
+e_{Ca}^{(d)}\partial_a\gamma_{JIC}^{[\bf n]}-e_{Ja}^{(d)}\partial_a\gamma_{CIC}^{[\bf n]}+2\delta_{IJ}e_{Da}^{(d)}\partial_a\gamma_{CDC}^{[\bf n]}\bigg]\\
\notag&-\frac{1}{2}\bigg[\gamma_{CID}^{[\bf n]}\gamma_{DJC}^{(d)}+\gamma_{CCD}^{[\bf n]}\gamma_{IJD}^{(d)}
+\gamma_{CJD}^{[\bf n]}\gamma_{DIC}^{(d)}+\gamma_{CCD}^{[\bf n]}\gamma_{JID}^{(d)}\\
\notag&+\gamma_{CID}^{(d)}\gamma_{DJC}^{[\bf n]}+\gamma_{CCD}^{(d)}\gamma_{IJD}^{[\bf n]}
+\gamma_{CJD}^{(d)}\gamma_{DIC}^{[\bf n]}+\gamma_{CCD}^{(d)}\gamma_{JID}^{[\bf n]}\\
\notag&+\gamma_{CID}^{(d)}\gamma_{DJC}^{(d)}+\gamma_{CCD}^{(d)}\gamma_{IJD}^{(d)}
+\gamma_{CJD}^{(d)}\gamma_{DIC}^{(d)}+\gamma_{CCD}^{(d)}\gamma_{JID}^{(d)}\bigg]\\
\notag&+\delta_{IJ}\bigg[(k^{[\bf n]}_{CD}-k^{[\bf 0]}_{CD})k^{(d)}_{CD}-(k_{CC}^{\bf n]}-k^{[\bf 0]}_{CC})k_{DD}^{(d)}
+\frac{1}{2}k_{CD}^{(d)}k_{CD}^{(d)}-\frac{1}{2}(k_{CC}^{(d)})^2\bigg]\\
\notag&+\frac{1}{2}\delta_{IJ}\bigg[\gamma_{CED}^{[\bf n]}\gamma_{DEC}^{(d)}
+\gamma_{CCD}^{[\bf n]}\gamma_{EED}^{(d)}
+\gamma_{CED}^{(d)}\gamma_{DEC}^{[\bf n]}
+\gamma_{CCD}^{(d)}\gamma_{EED}^{[\bf n]}
+\gamma_{CED}^{(d)}\gamma_{DEC}^{(d)}\\
\notag&+\gamma_{CCD}^{(d)}\gamma_{EED}^{(d)}\bigg]
+(\text{tr}k^{[\bf n]}-\text{tr}k^{[\bf 0]})k_{IJ}^{(d)}
+\text{tr}k^{(d)}(k^{[\bf n]}_{IJ}-k^{[\bf 0]}_{IJ})
+\text{tr}k^{(d)}k^{(d)}_{IJ},\\
\label{frak.G}\mathfrak{G}_{IJB}^{(d)}=&\,(k_{IC}^{[\bf n]}-k_{IC}^{[\bf 0]})\gamma_{CJB}^{(d)}
+k_{IC}^{(d)}\gamma_{CJB}^{[\bf n]}+k_{IC}^{(d)}\gamma_{CJB}^{(d)}
+\gamma_{JBC}^{[\bf n]}k_{CI}^{(d)}+\gamma_{JIC}^{[\bf n]}k_{BC}^{(d)}\\
\notag&-\gamma_{BJC}^{[\bf n]}k_{CI}^{(d)}-\gamma_{BIC}^{[\bf n]}k_{JC}^{(d)}
+\gamma_{JBC}^{(d)}(k_{CI}^{[\bf n]}-k_{CI}^{[\bf 0]})+\gamma_{JIC}^{(d)}(k_{BC}^{[\bf n]}-k_{BC}^{[\bf 0]})\\
\notag&-\gamma_{BJC}^{(d)}(k_{CI}^{[\bf n]}-k_{CI}^{[\bf 0]})
-\gamma_{BIC}^{(d)}(k_{JC}^{[\bf n]}-k_{JC}^{[\bf 0]})
+\gamma_{JBC}^{(d)}k_{CI}^{(d)}+\gamma_{JIC}^{(d)}k_{BC}^{(d)}\\
\notag&-\gamma_{BJC}^{(d)}k_{CI}^{(d)}-\gamma_{BIC}^{(d)}k_{JC}^{(d)}
+\delta_{IB}\bigg[\gamma_{CCD}^{[\bf n]}k_{DJ}^{(d)}+\gamma_{CJD}^{[\bf n]}k_{CD}^{(d)}
+\gamma_{CCD}^{(d)}(k_{DJ}^{[\bf n]}-k_{DJ}^{[\bf 0]})\\
\notag&+\gamma_{CJD}^{(d)}(k_{CD}^{[\bf n]}-k_{CD}^{[\bf 0]})+\gamma_{CCD}^{(d)}k_{DJ}^{(d)}+\gamma_{CJD}^{(d)}k_{CD}^{(d)}\bigg]
-\delta_{IJ}\bigg[\gamma_{CCD}^{[\bf n]}k_{DB}^{(d)}+\gamma_{CBD}^{[\bf n]}k_{CD}^{(d)}\\
\notag&
+\gamma_{CCD}^{(d)}(k_{DB}^{[\bf n]}-k_{DB}^{[\bf 0]})
+\gamma_{CBD}^{(d)}(k_{CD}^{[\bf n]}-k_{CD}^{[\bf 0]})
+\gamma_{CCD}^{(d)}k_{DB}^{(d)}+\gamma_{CBD}^{(d)}k_{CD}^{(d)}\bigg]\\
\notag&+e^{(d)}_{Ba}\partial_a(k_{JI}^{\bf [n]}-k_{JI}^{[\bf 0]})-e_{Ja}^{(d)}\partial_a(k_{BI}^{\bf [n]}-k_{BI}^{[\bf 0]})
-\delta_{IB}\big\{e_{Ca}^{(d)}\partial_a(k_{CJ}^{[\bf n]}-k_{CJ}^{[\bf 0]})\\
\notag&-e_{Ja}^{(d)}\partial_a(k_{CC}^{[\bf n]}-k_{CC}^{[\bf 0]})\big\}
+\delta_{IJ}\big\{e_{Ca}^{(d)}\partial_a(k_{CB}^{[\bf n]}-k_{CB}^{[\bf 0]})-e_{Ba}^{(d)}\partial_a(k_{CC}^{[\bf n]}-k^{[\bf 0]}_{CC})\big\}
\end{align}
and
\begin{align}
\label{cal.I.e}(\mathcal{I}^{[\bf n]}_e)_{Ia}=&\sum_{C\neq I}(k^{[\bf n]}_{IC}e_{Ca}^{[\bf n]}-k^{[\bf n-1]}_{IC}e_{Ca}^{[\bf n-1]})\\
\label{cal.I.k}(\mathcal{I}^{[\bf n]}_k)_{IJ}=&\,\frac{1}{2}({\bf R}_{IJ}^{\bf [n]}+{\bf R}_{JI}^{\bf [n]})
+\delta_{IJ}({\bf R}_{00}^{\bf [n]}+\frac{1}{2}{\bf R}^{\bf [n]})\\
\notag&-\partial_t(k_{IJ}^{[\bf n]}-\widetilde{k}^{[\bf n]}_{IJ})
+\text{tr}k^{[\bf n]}k^{[\bf n]}_{IJ}-\text{tr}\widetilde{k}^{[\bf n]}\widetilde{k}^{[\bf n]}_{IJ}\\
\notag&+\frac{1}{2}\delta_{IJ}\big[k_{CD}^{[\bf n]}k^{[\bf n]}_{CD}-(k_{CC}^{[\bf n]})^2\big]
-\frac{1}{2}\delta_{IJ}\big[\widetilde{k}_{CD}^{[\bf n]}\widetilde{k}^{[\bf n]}_{CD}-(\widetilde{k}_{CC}^{[\bf n]})^2\big]
\\
\label{cal.I.gamma}(\mathcal{I}^{[\bf n]}_\gamma)_{IJB}=&\,\delta_{IB}{\bf R}_{0J}^{\bf [n]}-\delta_{IJ}{\bf R}_{0B}^{[\bf n]}\\
\notag&+(k^{[\bf n]}_{IC}-\widetilde{k}^{[\bf n]}_{IC})\gamma_{CJB}^{[\bf n]}
+e_B^{[\bf n]}(k_{JI}^{[\bf n]}-\widetilde{k}^{[\bf n]}_{JI})-e_J^{[\bf n]}(k^{[\bf n]}_{BI}-\widetilde{k}^{[\bf n]}_{BI})\\
\notag&+\gamma_{JBC}^{[\bf n]}(k_{CI}^{[\bf n]}-\widetilde{k}_{CI}^{[\bf n]})+\gamma_{JIC}^{[\bf n]}(k^{[\bf n]}_{BC}-\widetilde{k}^{[\bf n]}_{BC})-\gamma_{BJC}^{[\bf n]}(k^{[\bf n]}_{CI}-\widetilde{k}^{[\bf n]}_{CI})\\
\notag&-\gamma_{BIC}^{[\bf n]}(k_{JC}^{[\bf n]}-\widetilde{k}_{JC}^{[\bf n]})
-\delta_{IB}\bigg[e_C^{[\bf n]}(k_{CJ}^{[\bf n]}-\widetilde{k}_{CJ}^{[\bf n]})-\gamma_{CCD}^{[\bf n]}(k_{DJ}^{[\bf n]}-\widetilde{k}_{DJ}^{[\bf n]})\\
\notag&-\gamma_{CJD}^{[\bf n]}(k_{CD}^{[\bf n]}-\widetilde{k}_{CD}^{[\bf n]})-e_J^{[\bf n]}(\text{tr}k^{[\bf n]}-\text{tr}\widetilde{k}^{[\bf n]})\bigg]
+\delta_{IJ}\bigg[e_C^{[\bf n]}(k_{CB}^{[\bf n]}-\widetilde{k}_{CB}^{[\bf n]})\\
\notag&-\gamma_{CCD}^{[\bf n]}(k_{DB}^{[\bf n]}-\widetilde{k}_{DB}^{[\bf n]})-\gamma_{CBD}^{[\bf n]}(k_{CD}^{[\bf n]}-\widetilde{k}_{CD}^{[\bf n]})-e_B^{[\bf n]}(\text{tr}k^{[\bf n]}-\text{tr}\widetilde{k}^{[\bf n]})\bigg]
\end{align}
The terms in \eqref{frak.K}, \eqref{frak.G} should be viewed as error terms in the energy estimates below. On the other hand, the terms explicitly written in the equations \eqref{e0.eI.d}-\eqref{e0.kIJ.d} with $t^{-1}$ coefficients need more careful handling. They will be absorbed by considering appropriately large $t$-weights in our norms. To this end, we will need $(\mathcal{I}^{[\bf n]}_e)_{Ia}, (\mathcal{I}^{[\bf n]}_k)_{IJ}, (\mathcal{I}^{[\bf n]}_\gamma)_{IJB}$ to decay at a sufficiently fast polynomial rate in $t$. It is immediate from Theorem \ref{thm:approx.sol} and Lemma \ref{lem:k.approx.2ndfund} that
\begin{lemma}\label{lem:cal.I}
For every multi-index $\alpha$ and all indices $I,J,B,a$, the expressions \eqref{cal.I.e}-\eqref{cal.I.gamma} satisfy the following bounds:
\begin{align}\label{cal.I.est}
|\partial_x^\alpha(\mathcal{I}^{[\bf n]}_e)_{Ia}|,\,|\partial_x^\alpha(\mathcal{I}^{[\bf n]}_k)_{IJ}|,\,|\partial_x^\alpha(\mathcal{I}^{[\bf n]}_\gamma)_{IJB}|\leq C_{\alpha,n}t^{M},
\end{align}
for all $(t,x)\in(0,t_n]\times[0,\delta]^3$,
where $M=M(n)\rightarrow+\infty$, as $n\rightarrow+\infty$. 
\end{lemma}
The following lemma makes evident the hyperbolic structure of the system \eqref{e0.eI.d}-\eqref{e0.gamma_IJB.d}.
\begin{lemma}\label{lem:var.d.symm}
The iterates satisfy $k_{IJ}^{[\bf n]}=k_{JI}^{[\bf n]}$, $\gamma_{IJB}^{[\bf n]}=-\gamma_{IBJ}^{[\bf n]}$, for all $(t,x)\in(0,T]\times [0,\delta]^3$. Moreover, 
the variables $k_{IJ}^{(d)},\gamma^{(d)}_{IJB}$ enjoy the same symmetry/antisymmetry properties 
\begin{align}\label{k.gamma.d.symm}
k_{IJ}^{(d)}=k_{JI}^{(d)},\qquad\gamma^{(d)}_{IJB}=-\gamma_{IBJ}^{(d)}
\end{align}
for all $(t,x)\in \{U_t\}_{t\in[\eta,T]}$, provided that they are valid on $U_\eta$, $\eta>0$. The same symmetries hold for $k_{IJ},\gamma_{IJB}$.
\end{lemma}
\begin{remark}
Equations \eqref{e0.kIJ.d}, \eqref{e0.gamma_IJB.d} form a first order symmetric hyperbolic system for $k_{IJ}^{(d)},\gamma_{IJB}^{(d)}$, provided \eqref{k.gamma.d.symm} holds true. Indeed, multiplying \eqref{e0.kIJ.d} with $k_{IJ}^{(d)}$, \eqref{e0.gamma_IJB.d} with $\frac{1}{2}\gamma_{IJB}^{(d)}$, and adding the resulting identities we notice that the first derivatives in the resulting RHS combine to give only whole derivatives of products $k^{(d)}\star\gamma^{(d)}$.
%
%
%
%
\end{remark}
\begin{proof}
First, notice that the corresponding properties hold for $k_{IJ}^{[\bf n]},\gamma_{IJB}^{[\bf n]}$ in $[0,t_n]\times[0,\delta]^3\supset \{U_t\}_{t\in[0,T]}$. Indeed, for $\gamma_{IJB}^{[\bf n]}$ this is clear by definition \eqref{gamma.n}. For $k_{IJ}^{[\bf n]}$, equation \eqref{kIJ.it} implies that
\begin{align*}
\partial_t (k_{IJ}^{[\bf n]}-k_{JI}^{[\bf n]})-k_{CC}^{[\bf n-1]} (k_{IJ}^{[\bf n]}-k_{JI}^{[\bf n]})=0.
\end{align*}
Rewrite
\begin{align*}
\partial_t[t(k_{IJ}^{[\bf n]}-k_{JI}^{[\bf n]})]-k_{CC}^{[\bf n-1]} (k_{IJ}^{[\bf n]}-k_{JI}^{[\bf n]})-(k_{CC}^{[\bf n-1]}-k_{CC}^{[\bf 0]}) t(k_{IJ}^{[\bf n]}-k_{JI}^{[\bf n]})=0
\end{align*}
By point \ref{item.approx1} in Theorem \ref{thm:approx.sol}, $|k_{CC}^{[\bf n-1]}-k_{CC}^{[\bf 0]}|\leq C_nt^{-1+\varepsilon}$. Also, from \eqref{asym.cond.it} it follows that
\begin{align*}
\lim_{t\rightarrow0}t(k_{IJ}^{[\bf n]}-k_{JI}^{[\bf n]})=0
\end{align*}
Hence, 
solving the above ODE via integrating factors for $t(k_{IJ}^{[\bf n]}-k_{JI}^{[\bf n]})$, we conclude that $t(k_{IJ}^{[\bf n]}-k_{JI}^{[\bf n]})=0$ everywhere.

Now, if $k^{(d)}_{IJ},\gamma^{(d)}_{IJB}$ satisfy \eqref{k.gamma.d.symm} on $U_T$, so do $k_{IJ},\gamma_{IJB}$. Returning to equations \eqref{e0.kIJ.mod2}, \eqref{e0.gamma_IJB.mod2}, we have that
\begin{align*}
\partial_t(k_{IJ}-k_{JI})-\text{tr}k(k_{IJ}-k_{JI})=0,\qquad\partial_t(\gamma_{IJB}+\gamma_{IBJ})-k_{IC}(\gamma_{CJB}+\gamma_{CBJ})=0
\end{align*}
The conclusion follows by observing that the initial data of $k_{IJ}-k_{JI},\gamma_{IJB}+\gamma_{IBJ}$ are trivial on $U_\delta$ and that the integral curves of $\partial_t$ emanating from all points in $U_\delta$ rule the entire domain $\{U_t\}_{t\in[\delta,T]}$.
\end{proof}

\subsection{Local existence}\label{subsec:exist} 

Our goal in this section is to prove the following. 
\begin{theorem}\label{thm:loc.exist}
For every $s\ge 4$ and $N_0\in\mathbb{N}$, there exists $n_{N_0,s}\in\mathbb{N}$ sufficiently large, such that for every $n\ge n_{N_0,s}$, there exists $T=T_{N_0,s,n}>0$ sufficiently small and a solution $e_{Ia},k_{IJ},\gamma_{IJB}$ to \eqref{e0.eIi2}-\eqref{e0.gamma_IJB.mod2}, in the domain $\{U_t\}_{t\in(0,T_{N_0,s,n}]}$ (see Section \ref{subsec:domain}), such that the following estimate holds:
\begin{align}\label{loc.exist.est}
\|e^{(d)}\|^2_{H^s(U_t)}+\|k^{(d)}\|^2_{H^s(U_t)}+\|\gamma^{(d)}\|^2_{H^s(U_t)}\leq t^{2N_0},
\end{align}
for all $t\in(0,T_{N_0,s,n}]$,
where the remainder terms $e_{Ia}^{(d)},k_{IJ}^{(d)},\gamma_{IJB}^{(d)}$ are as in \eqref{e.k.gamma.d}.
\end{theorem}
\begin{proof}
It is split in Propositions \ref{prop:loc.exist1}, \ref{prop:loc.exist2}.
\end{proof}
We first begin with a solution defined in the future of a non-singular time, furnished by standard local existence.
\begin{lemma}\label{lem:loc.exist}
For every $\eta>0$ sufficiently small and $n\in\mathbb{N}$, there exists $T=T(\eta,n)$ and a unique smooth solution $e_{Ia},k_{IJ},\gamma_{IJB}$ to \eqref{e0.eIi2}, \eqref{e0.kIJ.mod2}, \eqref{e0.gamma_IJB.mod2} in $\{U_t\}_{t\in[\eta,T]}$, such that 
\begin{align*}
e_{Ia}(\eta,x)=e_{Ia}^{[\bf n]}(\eta,x),\qquad k_{IJ}(\eta,x)=k_{IJ}^{[\bf n]}(\eta,x),\qquad \gamma_{IJB}(\eta,x)=\gamma_{IJB}^{[\bf n]}(\eta,x), 
\end{align*}
for all $x\in U_\eta$.
Moreover $k_{IJ}=k_{JI}$, $\gamma_{IJB}=-\gamma_{IBJ}$.  
\end{lemma}
\begin{proof}
By standard local existence for first order symmetric hyperbolic systems, we have that a solution to \eqref{e0.eIi2}-\eqref{e0.gamma_IJB.mod2} exists in the future domain of dependence of $U_\eta$, denoted by $\mathcal{D}(U_\eta)$. However, given that we have modified the original equations \eqref{e0.k}, \eqref{gammaIJB.eq}, the latter domain might not a priori be the same as the domain of dependence for the Einstein vacuum equations. $\mathcal{D}(U_\eta)$ is the largest domain for which the energy associated to the linearized version of \eqref{e0.kIJ.mod2}-\eqref{e0.gamma_IJB.mod2}, on a future time slice, can be bounded from the initial one. Moreover, $\mathcal{D}(U_\eta)$ is determined from the principal terms in the equations. Hence, it is the same for the system \eqref{e0.eI.d}-\eqref{e0.gamma_IJB.d}. As we will show  in Section \ref{subsec:en.est}, by deriving such energy estimates for the variables $e_{Ia}^{(d)},k_{IJ}^{(d)},\gamma_{IJB}^{(d)}$ in $\{U_t\}_{t\in[\eta,T]}$, we can choose the constant $\sigma$ in \eqref{flow.vec} sufficiently large 
and $T$ sufficiently small (independently of $\eta$)
to guarantee that $\mathcal{D}(U_\eta)$ necessarily contains $\{U_t\}_{t\in[\eta,T]}$. From Lemma \ref{lem:var.d.symm}, we also have the desired symmetry/antisymmetry relations.
\end{proof}
%

Let
\begin{align}\label{norms}
\notag\|e^{(d)}\|^2_{H^s(U_t)}=&\sum_{I,a}\|e^{(d)}_{Ia}\|^2_{H^s(U_t)},
\qquad \|e^{(d)}\|_{W^s(U_t)}=\sum_{I,a}\|e^{(d)}_{Ia}\|_{W^s(U_t)},\\
\|k^{(d)}\|^2_{H^s(U_t)}=&\sum_{I,J}\|k^{(d)}_{IJ}\|^2_{H^s(U_t)},
\qquad \|k^{(d)}\|_{W^s(U_t)}=\sum_{I,J}\|k^{(d)}_{IJ}\|_{W^s(U_t)},\\
\notag\|\gamma^{(d)}\|^2_{H^s(U_t)}=&\sum_{I,J,B}\|\gamma^{(d)}_{IJB}\|^2_{H^s(U_t)},\qquad
\|\gamma^{(d)}\|_{W^s(U_t)}=\sum_{I,J,B}\|\gamma^{(d)}_{IJB}\|_{W^s(U_t)},
\end{align}
where for a function $f:\{U_t\}_{t\in[0,T]}\to\mathbb{R}$ with the appropriate regularity, the $H^s(U_t),W^s(U_t)$ norms are defined as follows:
\begin{align}\label{norms.f}
\begin{split}
\|f\|_{H^s(U_t)}^2=&\sum_{|\alpha|\leq s}\int_{U_t}[\partial_x^\alpha f(t,x)]^2\mathrm{vol}_{g^{[\bf n]}},\qquad\mathrm{vol}_{g^{[\bf n]}}=\sqrt{|g^{[\bf n]}|}dx^1dx^2dx^3,\\ 
\|f\|_{W^s(U_t)}=&\sum_{|\alpha|\leq s}\text{esssup}_{x\in U_t}|\partial_x^\alpha f(t,x)|.
\end{split}
\end{align}
We will bootstrap energy estimates with suitably large weights, which will guarantee that the time of existence in Lemma \ref{lem:loc.exist} is independent of $\eta>0$ and that the differences \eqref{e.k.gamma.d} decay to sufficiently large polynomial order.
\begin{proposition}\label{prop:loc.exist1}
Let $e_{Ia},k_{IJ},\gamma_{IJB}$ be as in Lemma \ref{lem:loc.exist}. 
For every $s\ge 4$ and $N_0\in\mathbb{N}$, there exists $n_{N_0,s}\in\mathbb{N}$ sufficiently large, such that for every $n\ge n_{N_0,s}$, there exists $T=T_{N_0,s,n}>0$ sufficiently small such that the following estimate holds:
\begin{align}\label{loc.exist.est2}
\|e^{(d)}\|^2_{H^s(U_t)}+\|k^{(d)}\|^2_{H^s(U_t)}+\|\gamma^{(d)}\|^2_{H^s(U_t)}\leq t^{2N_0},
\end{align}
for all $t\in[\eta,T_{N_0,s,n}]$. In particular, the time of existence $T(\eta,n)$ given by Lemma \ref{lem:loc.exist} can be fixed to be $T_{N_0,s,n}$, independent of $\eta>0$.
\end{proposition}
\begin{proof}
It is given in the end of Section \ref{subsec:en.est}, after deriving the main weighted energy estimates, see Proposition \ref{prop:en.est}.
\end{proof}
Since the estimate \eqref{loc.exist.est} is independent of $\eta>0$, it is clear now that we can a extract a subsequence which will satisfy Theorem \ref{thm:loc.exist}.
\begin{proposition}\label{prop:loc.exist2}
Let $s,N_0,n$ be as in Proposition \ref{prop:loc.exist1}. Then there exists a sequence of initial times $\eta_m\rightarrow0$ such that \\[5pt]
1. The corresponding sequence of solutions $(e_{Ia,m},k_{IJ,m},\gamma_{IJB,m})$ furnished by Lemma \ref{lem:loc.exist} converges in $C^1$, as $\eta\rightarrow0$, to a limit $(e_{Ia},k_{IJ},\gamma_{IJB})$. \\[5pt]
2. The limit solves the system \eqref{e0.eIi2}-\eqref{e0.gamma_IJB.mod2} in $\{U_t\}_{t\in(0,T_{N_0,s,n}]}$.\\[5pt]
3. Moreover, the corresponding differences $e^{(d)}_{Ia},k^{(d)}_{IJ},\gamma^{(d)}_{IJB}$ satisfy the estimate \eqref{loc.exist.est}, for all $t\in(0,T_{N_0,s,n}]$.
\end{proposition}
\begin{proof}
Using the uniform estimate \eqref{loc.exist.est2}, Arzelà-Ascoli, and a standard diagonal argument, we infer that the sequence $(e_{Ia,m},k_{IJ,m},\gamma_{IJB,m})$ has a subsequence converging in $C^1$ to a limit $(e_{Ia},k_{IJ},\gamma_{IJB})$ for every fixed $t\in(0,T_{N_0,s,n}]$. Also, the class $C^1$ is enough to ensure that the limit satisfies the system \eqref{e0.eIi2}-\eqref{e0.gamma_IJB.mod2}. Moreover, for every $t\in(0,T_{N_0,s,n}]$, the former subsequence has a subsequence converging weakly in $H^s(U_t)$ and the limit satisfies \eqref{loc.exist.est2}. By uniqueness of limits, we conclude that the limit $(e_{Ia},k_{IJ},\gamma_{IJB})$ satisfies \eqref{loc.exist.est2} for all $t\in(0,T_{N_0,s,n}]$.
\end{proof}
\subsection{Weighted energy estimates for the remainder terms}\label{subsec:en.est}

In this subsection we derive the main energy estimates for the variables $e^{(d)}_{Ia},k^{(d)}_{IJ},\gamma_{IJB}^{(d)}$, see Proposition \ref{prop:en.est}, that complete the proof Proposition \ref{prop:loc.exist1} in Section \ref{subsubsec:prop.proof}.

\subsubsection{Bootstrap assumptions and basic implications}

Consider the solution furnished by Lemma \ref{lem:loc.exist} and fix $s\ge4$. We make the bootstrap assumptions
\begin{align}\label{Boots}
\|e^{(d)}\|^2_{H^s(U_t)}+\|k^{(d)}\|^2_{H^s(U_t)}+\|\gamma^{(d)}\|^2_{H^s(U_t)}\leq t^{11},
\end{align}
for all $t\in[\eta,T_{Boot})$, where $T_{Boot}<T(\eta,n)$. Notice that such a bootstrap time exists by continuity, since the variables $e^{(d)}_{Ia},k^{(d)}_{IJ},\gamma^{(d)}_{IJB}$ vanish on $U_\eta$. 
\begin{remark}\label{rem:cont}
Deriving the estimate \eqref{loc.exist.est2} for $N_0>5$ clearly improves the bootstrap assumptions. A standard continuity argument then implies that the time of existence can be pushed to some $T=T_{N_0,s,n}$, independent of $\eta$. 
\end{remark}
By \eqref{gij.0} and point \ref{item.approx1} of Theorem \ref{thm:approx.sol}, we have that 
\begin{align}\label{vol.gn}
\big|\sqrt{|g^{[\bf n]}|}-\sqrt{c_{11}c_{22}c_{33}}t\big|\leq C_nt^{1+\varepsilon}
\end{align}
Hence, the bootstrap assumptions and classical Sobolev embedding imply the bound
\begin{align}\label{Sob}
\|e^{(d)}\|_{W^{s-2}(U_t)}+\|k^{(d)}\|_{W^{s-2}(U_t)}+\|\gamma^{(d)}\|_{W^{s-2}(U_t)}\leq Ct^5.
\end{align}
\begin{lemma}\label{lem:IBP}
Given a function $f:\{U_t\}_{t\in[\eta,T_{Boot})}\to\mathbb{R}$, the following inequalities hold true:
\begin{align}
\label{IBP:ineq.dt}
\int_{U_t}f\mathrm{vol}_{g^{[\bf n]}}+\sum_{a,\pm}\int_{\mathcal{H}^\pm_a}n_{\mathcal{H}^\pm_a}^0f\mathrm{vol}_{\mathcal{H}_a^\pm}
\leq&\int_{U_\eta}f\mathrm{vol}_{g^{[\bf n]}}
+\int_\eta^t\int_{U_\tau}\partial_\tau f\mathrm{vol}_{g^{[\bf n]}}d\tau\\
\notag&+\int_\eta^t\int_{U_\tau}(\tau^{-1}+C\tau^{-1+\varepsilon})|f|\mathrm{vol}_{g^{[\bf n]}}d\tau,\\
\label{IBP:ineq.eD}\int_\eta^t\int_{U_\tau}e_If\mathrm{vol}_{g^{[\bf n]}}d\tau\leq& \int_\eta^t\int_{U_\tau}\tau^{-1+\varepsilon}|f|\mathrm{vol}_{g^{[\bf n]}}d\tau\\
\notag&+\sum_{a,\pm}\int_{\mathcal{H}^\pm_a}n^I_{\mathcal{H}^\pm_a}f+Ct^4|f|\mathrm{vol}_{\mathcal{H}_a^\pm}.
\end{align}
\end{lemma}
\begin{proof}
Applying the Stokes theorem to the divergences ${\bf div}^{[\bf n]}(f\partial_t)$, ${\bf div}^{[\bf n]}(fe_I)$, in the region $\{U_\tau\}_{\tau\in[\eta,t]}$, for $t\in[\eta,T_{Boot})$, gives the identities: 
\begin{align}\label{IBP:dt}
\int^t_\eta\int_{U_\tau}{\bf div}^{[\bf n]}(f\partial_t)\mathrm{vol}_{g^{[\bf n]}}d\tau
=&\int_{U_\eta}{\bf g}^{[\bf n]}(\partial_t,\partial_t)f\mathrm{vol}_{g^{[\bf n]}}\\
\notag&-\int_{U_t}{\bf g}^{[\bf n]}(\partial_t,\partial_t)f\mathrm{vol}_{g^{[\bf n]}}
-\sum_{a,\pm}\int_{\mathcal{H}^\pm_a}{\bf g}^{[\bf n]}(\partial_t,n_{\mathcal{H}^\pm_a})f\mathrm{vol}_{\mathcal{H}_a^\pm},\\
\label{IBP:eD}\int^t_\eta\int_{U_\tau}{\bf div}^{[\bf n]}(fe_I)\mathrm{vol}_{g^{[\bf n]}}d\tau
=&\int_{U_\eta}{\bf g}^{[\bf n]}(e_I,\partial_t)f\mathrm{vol}_{g^{[\bf n]}}\\
\notag&-\int_{U_t}{\bf g}^{[\bf n]}(e_I,\partial_t)f\mathrm{vol}_{g^{[\bf n]}}
-\sum_{a,\pm}\int_{\mathcal{H}^\pm_a}{\bf g}^{[\bf n]}(e_I,n_{\mathcal{H}^\pm_a})f\mathrm{vol}_{\mathcal{H}_a^\pm}.
\end{align}
On the other hand, by \eqref{Sob} we have 
\begin{align}\label{IBP:error}
\begin{split}
{\bf g}^{[\bf n]}_{tt}=-1,\qquad {\bf g}^{[\bf n]}(\partial_t,n_{\mathcal{H}^\pm_a})=-n_{\mathcal{H}^\pm_a}^0,
{\bf g}^{[\bf n]}(e_I,\partial_t)=0,\\
\big|{\bf g}^{[\bf n]}(e_I,n_{\mathcal{H}^\pm_a})-n_{\mathcal{H}^\pm_a}^I\big|=|e^{(d)}_{Ib}\omega_{bD}^{[\bf n]}n_{\mathcal{H}^\pm_a}^D|\leq Ct^4
\end{split}
\end{align}
and
\begin{align}\label{IBP:error2}
\begin{split}
{\bf div}^{[\bf n]}(f\partial_t)=&\,\partial_t f-\widetilde{k}^{[\bf n]}_{CC}f,\qquad |\widetilde{k}_{CC}^{[\bf n]}+t^{-1}|\leq Ct^{-1+\varepsilon},\\
{\bf div}^{[\bf n]}(fe_I)=&\,e_If+g^{[\bf n]}(\nabla^{[\bf n]}_{e_D^{[\bf n]}}e_I,e^{[\bf n]}_D),
\end{split}
\end{align}
where 
\begin{align}\label{IBP:error3}
\notag g^{[\bf n]}(\nabla^{[\bf n]}_{e_D^{[\bf n]}}e_I,e^{[\bf n]}_D)=&\,\gamma_{DID}^{[\bf n]}
+g^{[\bf n]}(\nabla^{[\bf n]}_{e_D^{[\bf n]}}(e_{Ib}^{(d)}\partial_b),e^{[\bf n]}_D)\\
=&\,\gamma_{DID}^{[\bf n]}+g^{[\bf n]}(\nabla^{[\bf n]}_{e_D^{[\bf n]}}(e_{Ib}^{(d)}\omega^{[\bf n]}_{bC}e_C^{[\bf n]}),e^{[\bf n]}_D)\\
\notag=&\,\gamma_{DID}^{[\bf n]}+e_{Ib}^{(d)}\omega^{[\bf n]}_{bC}\gamma_{DCD}^{[\bf n]}
+e_C^{[\bf n]}(e_{Ib}^{(d)}\omega^{[\bf n]}_{bC}).
\end{align}
By Lemma \ref{lem:gamma}, point \ref{item.approx1} of Theorem \ref{thm:approx.sol}, and \eqref{Sob}, we have that
\begin{align}\label{IBP:error4}
|g^{[\bf n]}(\nabla^{[\bf n]}_{e_D^{[\bf n]}}e_I,e^{[\bf n]}_D)|\leq Ct^{-1+\varepsilon}
\end{align}
Combining \eqref{IBP:dt}-\eqref{IBP:error4}, we conclude the desired inequalities.
\end{proof}

\subsubsection{Main estimates}\label{subsubsec:main.est}

The overall weighted energy estimate that we derive in this section is stated in the following proposition.
\begin{proposition}\label{prop:en.est}
Assume that the bootstrap assumptions \eqref{Boots} are valid for some $s\ge4$. Then, there exists $\sigma>0$ sufficiently large in \eqref{flow.vec}, depending only on the principal symbol of \eqref{e0.kIJ.d}-\eqref{e0.gamma_IJB.d}, such that the following energy inequality holds:
\begin{align}\label{main.en.ineq}
\notag& \frac{1}{2}t^{-2N_0}\big\{\|e^{(d)}\|^2_{H^s(U_t)}+\|k^{(d)}\|^2_{H^s(U_t)}+\frac{1}{2}\|\gamma^{(d)}\|^2_{H^s(U_t)}\big\}\\
&+\int^t_\eta\frac{N_0-C_*}{\tau}\tau^{-2N_0}\big\{\|e^{(d)}\|^2_{H^s(U_\tau)}+\|k^{(d)}\|^2_{H^s(U_\tau)}+\frac{1}{2}\|\gamma^{(d)}\|^2_{H^s(U_\tau)}\big\}d\tau\\
\notag\leq&\int^t_\eta C\tau^{-1+\varepsilon}\tau^{-2N_0}\big\{\|e^{(d)}\|^2_{H^s(U_\tau)}+\|k^{(d)}\|^2_{H^s(U_\tau)}+\|\gamma^{(d)}\|^2_{H^s(U_\tau)}\big\}d\tau+C t^{2M-2N_0+2-\varepsilon},
\end{align}
for all $t\in[\eta,T_{Boot})$, where $M$ is as in Lemma \ref{lem:cal.I}. The constant $C_*>0$ depends only on the Kasner exponents $p_I$, their $s$ spatial derivatives, and it is independent of $n$, while the constant $C$ is allowed to also depend on $n$.
\end{proposition}
To derive the $H^s(U_t)$ estimate \eqref{main.en.ineq}, we first differentiate the system \eqref{e0.eI.d}-\eqref{e0.gamma_IJB.d} with $\partial_x^\alpha$, for $|\alpha|\leq s$, to obtain: 
\begin{align}
\label{e0.eI.d.diff}\partial_t\partial_x^\alpha e_{Ia}^{(d)}+\partial_x^\alpha(\frac{p_{\underline{I}}}{t}e^{(d)}_{\underline{I}a})=&\,
\partial_x^\alpha\big\{(k_{IC}^{[\bf n]}-k^{[\bf 0]}_{IC})e^{(d)}_{Ca}
+k_{IC}^{(d)}e^{[\bf n]}_{Ca}+k_{IC}^{(d)}e^{(d)}_{Ca}\big\}+\partial_x^\alpha(\mathcal{I}^{[\bf n]}_e)_{Ia}
\end{align}
and
\begin{align}
\notag&\partial_t\partial_x^\alpha k_{IJ}^{(d)}+\frac{1}{t}\partial_x^\alpha k^{(d)}_{IJ}\\
\label{e0.kIJ.d.diff}=&\,\frac{1}{2}\big\{e_C\partial_x^\alpha\gamma_{IJC}^{(d)}-e_I\partial_x^\alpha\gamma_{CJC}^{(d)}
+e_C\partial_x^\alpha\gamma_{JIC}^{(d)}-e_J\partial_x^\alpha\gamma_{CIC}^{(d)}+2\delta_{IJ}e_D\partial_x^\alpha\gamma_{CDC}^{(d)}\big\}\\
\notag&-\partial_x^\alpha\big\{\delta_{\underline{I}J}\frac{p_{\underline{I}}}{t}k^{(d)}_{CC}-\delta_{IJ}\frac{1}{t}k^{(d)}_{CC}+\delta_{IJ}\sum_C\frac{p_C}{t}k_{CC}^{(d)}\big\}
+\mathfrak{K}_{IJ}^{(d),\alpha}+\partial_x^\alpha (\mathcal{I}_k^{[\bf n]})_{IJ},\\
\notag&\partial_t \partial_x^{\alpha} \gamma_{IJB}^{(d)}+\partial_x^{\alpha}(\frac{p_{\underline{I}}}{t} \gamma_{\underline{I}JB}^{(d)}) \\
\label{e0.gamma_IJB.d.diff}=&\,e_B\partial_x^{\alpha} k_{JI}^{(d)}-e_J\partial_x^{\alpha}k_{BI}^{(d)}
-\delta_{IB}\big[e_C\partial_x^{\alpha}k_{CJ}^{(d)}-e_J\partial_x^{\alpha}k_{CC}^{(d)}\big]+\delta_{IJ}\big[e_C\partial_x^{\alpha}k_{CB}^{(d)}-e_B\partial_x^{\alpha}k^{(d)}_{CC}\big]\\
\notag&- \partial_x^{\alpha}\bigg(\frac{p_{\underline{I}}}{t}\gamma_{JB\underline{I}}^{(d)}-\frac{p_{\underline{B}}}{t}\gamma_{JI\underline{B}}^{(d)}+\frac{p_{\underline{I}}}{t}\gamma_{BJ\underline{I}}^{(d)}+\frac{p_{\underline{J}}}{t}\gamma_{BI\underline{J}}^{(d)}\bigg)\\
\notag&-\delta_{IB}\partial_x^{\alpha}\bigg[\frac{p_{\underline{J}}}{t}\gamma_{CCJ}^{(d)}+\sum_C\frac{p_C}{t}\gamma_{CJC}^{(d)}\bigg]
+\delta_{IJ}\partial_x^{\alpha}\bigg[\frac{p_{\underline{B}}}{t}\gamma_{CC\underline{B}}^{(d)}+\sum_C\frac{p_C}{t}\gamma_{CBC}^{(d)}\bigg]\\
\notag&-\partial_x^{\alpha}\bigg[\delta_{\underline{I}J}\frac{\partial_ap_{\underline{I}}}{t}e^{(d)}_{Ba}
+\delta_{\underline{I}B}\frac{\partial_ap_{\underline{I}}}{t}e_{Ja}^{(d)}
+\delta_{IB}\frac{\partial_ap_{\underline{J}}}{t}e_{\underline{J}a}^{(d)}-\delta_{IJ}\frac{\partial_ap_{\underline{B}}}{t}e_{\underline{B}a}^{(d)} \bigg] \\ \notag&
+\mathfrak{G}_{IJB}^{(d),\alpha}+\partial_x^{\alpha}(\mathcal{I}_\gamma^{[\bf n]})_{IJB},
\end{align}
where 
\begin{align}
\label{frak.K.diff}\mathfrak{K}_{IJ}^{(d),\alpha}=&\sum_{\substack{|\alpha_1|+|\alpha_2|=|\alpha|\\|\alpha_2|<|\alpha|}}\frac{1}{2}\bigg[\partial_x^{\alpha_1}e_{Cb}\partial_x^{\alpha_2}\partial_b\gamma_{IJC}^{(d)}-\partial_x^{\alpha_1}e_{Ib}\partial_x^{\alpha_2}\partial_b\gamma_{CJC}^{(d)}
+\partial_x^{\alpha_1}e_{Cb}\partial_x^{\alpha_2}\partial_b\gamma_{JIC}^{(d)}\\
\notag&-\partial_x^{\alpha_1}e_{Jb}\partial_x^{\alpha_2}\partial_b\gamma_{CIC}^{(d)}+2\delta_{IJ}\partial_x^{\alpha_1}e_{Db}\partial_x^{\alpha_2}\partial_b\gamma_{CDC}^{(d)}\bigg]+\partial_x^\alpha\mathfrak{K}_{IJ}^{(d)},\\
\label{frak.G.diff}\mathfrak{G}_{IJB}^{(d),\alpha}=&\sum_{\substack{|\alpha_1|+|\alpha_2|=|\alpha|\\|\alpha_2|<|\alpha|}}\bigg[\partial_x^{\alpha_1}e_{Ba}\partial^{\alpha_2}\partial_ak_{JI}^{(d)}-\partial_x^{\alpha_1}e_{Ja}\partial^{\alpha_2}\partial_ak_{BI}^{(d)}\\
\notag&-\delta_{IB}\big[\partial_x^{\alpha_1}e_{Ca}\partial^{\alpha_2}\partial_ak_{CJ}^{(d)}-\partial_x^{\alpha_1}e_{Ja}\partial^{\alpha_2}\partial_ak_{CC}^{(d)}\big]
\\
\notag&+\delta_{IJ}\big[\partial_x^{\alpha_1}e_{Ca}\partial^{\alpha_2}\partial_ak_{CB}^{(d)}-\partial_x^{\alpha_1}e_{Ba}\partial^{\alpha_2}\partial_ak^{(d)}_{CC}\big]\bigg]
+\partial_x^\alpha\mathfrak{G}_{IJB}^{(d)}.
\end{align}
Next, we write the overall differential inequality for the \eqref{e0.eI.d.diff}-\eqref{e0.gamma_IJB.d.diff}.
\begin{lemma}\label{lem:en.id}
There exists a constant $C_*$, depending only on the Kasner exponents and their $s$ spatial derivatives, such that the following differential inequality holds true:
\begin{align}\label{en.id}
\notag&\partial_t\Big(\sum_{|\alpha|\leq s}\frac{1}{2}t^{-2N_0}\big\{\partial_x^\alpha e_{Ia}^{(d)}\partial_x^\alpha e_{Ia}^{(d)}+\partial_x^\alpha k_{IJ}^{(d)}\partial_x^\alpha k_{IJ}^{(d)}+\frac{1}{2}\partial_x^\alpha \gamma_{IJB}^{(d)}\partial_x^\alpha \gamma_{IJB}^{(d)}\big\}\Big)\\
\notag&+\sum_{|\alpha|\leq s}\frac{N_0-C_*}{t}t^{-2N_0}\big\{\partial_x^\alpha e_{Ia}^{(d)}\partial_x^\alpha e_{Ia}^{(d)}+\partial_x^\alpha k_{IJ}^{(d)}\partial_x^\alpha k_{IJ}^{(d)}+\frac{1}{2}\partial_x^\alpha \gamma_{IJB}^{(d)}\partial_x^\alpha \gamma_{IJB}^{(d)}\big\}\\
\leq&\sum_{|\alpha|\leq s}\bigg[t^{-2N_0}\big\{e_C\big[\partial_x^\alpha k_{IJ}^{(d)}\partial_x^\alpha \gamma_{IJC}^{(d)}\big]-e_I\big[\partial_x^\alpha k_{IJ}^{(d)}\partial_x^\alpha \gamma_{CJC}^{(d)}\big]+e_D\big[\partial_x^\alpha k_{II}^{(d)}\partial_x^\alpha \gamma_{CDC}^{(d)}\big]\big\}\\
\notag&+t^{-2N_0}\partial_x^\alpha e_{Ia}^{(d)}\partial_x^\alpha\big\{(k_{IC}^{[\bf n]}-k^{[\bf 0]}_{IC})e^{(d)}_{Ca}
+k_{IC}^{(d)}e^{[\bf n]}_{Ca}+k_{IC}^{(d)}e^{(d)}_{Ca}\big\}+t^{-2N_0}\partial_x^\alpha e_{Ia}^{(d)}\partial_x^\alpha(\mathcal{I}^{[\bf n]}_e)_{Ia}\\
\notag&+t^{-2N_0}\partial_x^\alpha k_{IJ}^{(d)}\big[\mathfrak{K}_{IJ}^{(d),\alpha}+\partial_x^\alpha (\mathcal{I}_k^{[\bf n]})_{IJ}\big]
+\frac{1}{2}t^{-2N_0}\partial_x^\alpha \gamma_{IJB}^{(d)}\big[\mathfrak{G}_{IJB}^{(d),\alpha}+\partial_x^\alpha (\mathcal{I}_\gamma^{[\bf n]})_{IJB}\big]\bigg],
\end{align}
for all $t\in[\eta,T_{Boot})$.
\end{lemma}
\begin{proof}
Consider the algebraic combination of equations
\begin{align*}
\sum_{|\alpha|\leq s}\big\{t^{-2N_0}\partial_x^\alpha e^{(d)}_{Ia}\times\eqref{e0.eI.d.diff}+t^{-2N_0}\partial_x^\alpha k^{(d)}_{IJ}\times\eqref{e0.kIJ.d.diff}+\frac{1}{2}t^{-2N_0}\partial_x^\alpha\gamma_{IJB}^{(d)}\times\eqref{e0.gamma_IJB.d.diff}\big\}, 
\end{align*}
differentiate by parts in the top order terms and use Lemma \ref{lem:var.d.symm} to write them as whole derivatives. 
The lower order terms with $t^{-1}$ coefficients are grouped together using Young's inequality to give the brackets in the LHS having the $C_*$ constant coefficient.
\end{proof}




%
We will make use of the following error estimates.
\begin{lemma}\label{lem:error.est}
Assume that the bootstrap assumptions are valid for some $s\ge4$.
The following expressions satisfy
\begin{align}\label{error.est}
\notag&\int_{U_t}t^{-2N_0}\bigg[\partial_x^\alpha e_{Ia}^{(d)}\partial_x^\alpha\big\{(k_{IC}^{[\bf n]}-k^{[\bf 0]}_{IC})e^{(d)}_{Ca}
+k_{IC}^{(d)}e^{[\bf n]}_{Ca}+k_{IC}^{(d)}e^{(d)}_{Ca}\big\}+\partial_x^\alpha e_{Ia}^{(d)}\partial_x^\alpha(\mathcal{I}^{[\bf n]}_e)_{Ia}\\
&+\partial_x^\alpha k_{IJ}^{(d)}\big[\mathfrak{K}_{IJ}^{(d),\alpha}+\partial_x^\alpha (\mathcal{I}_k^{[\bf n]})_{IJ}\big]
+\frac{1}{2}\partial_x^\alpha \gamma_{IJB}^{(d)}\big[\mathfrak{K}_{IJB}^{(d),\alpha}+\partial_x^\alpha (\mathcal{I}_\gamma^{[\bf n]})_{IJB}\big]\bigg]\mathrm{vol}_{g^{\bf[n]}}\\
\notag\leq&\,Ct^{-1+\varepsilon}t^{-2N_0}\big\{\|e^{(d)}\|^2_{H^s(U_t)}+\|k^{(d)}\|^2_{H^s(U_t)}+\|\gamma^{(d)}\|^2_{H^s(U_t)}\big\}+Ct^{1-\varepsilon}t^{2M-2N_0},
\end{align}
for all $t\in[\eta,T_{Boot})$ and $|\alpha|\leq s$, where $M$ is as in Lemma \ref{lem:cal.I}.
\end{lemma}
\begin{proof}
The last term in the RHS of \eqref{error.est} comes from Young's inequality, using Lemma \ref{lem:cal.I}. Similarly, all quadratic terms in $e^{(d)}_{Ia},k^{(d)}_{Ia},\gamma^{(d)}_{IJB}$ are treated by Young's inequality, observing that their coefficients are bounded by $C_{\alpha,n}t^{-1+\varepsilon}$, by virtue of point \ref{item.approx1} in Theorem \ref{thm:approx.sol}. For the cubic terms in the differences, we recall \eqref{frak.K}-\eqref{frak.G} and notice that at least one factor has at most $s-2$ spatial derivatives, since $s\ge4$. We estimate that factor in $W^{s-2,\infty}(U_t)$ using \eqref{Sob} and apply Young's inequality once more. These terms are actually much better behaved, thanks to our bootstrap assumptions \eqref{Boots}, which are used whenever the terms $e_{Ia}^{(d)}, k_{IJ}^{(d)}, \gamma_{IJB}^{(d)}$ are encountered. This accounts for all terms in the LHS of \eqref{error.est}.  
\end{proof}
Now we can proceed to the 
\begin{proof}[Proof of Proposition \ref{prop:en.est}]
Integrate \eqref{en.id} in the domain $\{U_\tau\}_{\tau\in[\eta,t]}$ and employ Lemma \ref{lem:error.est} to obtain the inequality 
\begin{align}\label{en.ineq}
\notag&\int_\eta^t\int_{U_\tau}\partial_{\tau}\sum_{|\alpha|\leq s}\frac{1}{2}\tau^{-2N_0}\big\{\partial_x^\alpha e_{Ia}^{(d)}\partial_x^\alpha e_{Ia}^{(d)}+\partial_x^\alpha k_{IJ}^{(d)}\partial_x^\alpha k_{IJ}^{(d)}+\frac{1}{2}\partial_x^\alpha \gamma_{IJB}^{(d)}\partial_x^\alpha \gamma_{IJB}^{(d)}\big\}\mathrm{vol}_{g^{[\bf n]}}d\tau\\
\notag&+\int^t_\eta\frac{N_0-C_*}{\tau}\tau^{-2N_0}\big\{\|e^{(d)}\|^2_{H^s(U_\tau)}+\|k^{(d)}\|^2_{H^s(U_\tau)}+\frac{1}{2}\|\gamma^{(d)}\|^2_{H^s(U_\tau)}\big\}
\mathrm{vol}_{g^{[\bf n]}}d\tau\\
\leq&\sum_{|\alpha|\leq s}\int_\eta^t\int_{U_\tau}
\tau^{-2N_0}\big\{e_C\big[\partial_x^\alpha k_{IJ}^{(d)}\partial_x^\alpha \gamma_{IJC}^{(d)}\big]-e_I\big[\partial_x^\alpha k_{IJ}^{(d)}\partial_x^\alpha \gamma_{CJC}^{(d)}\big]\\
\notag&+e_D\big[\partial_x^\alpha k_{II}^{(d)}\partial_x^\alpha \gamma_{CDC}^{(d)}\big]\big\}\mathrm{vol}_{g^{[\bf n]}}d\tau
+\int^t_\eta C\tau^{-1+\varepsilon}\tau^{-2N_0}\big\{\|e^{(d)}\|^2_{H^s(U_\tau)}\\
\notag&+\|k^{(d)}\|^2_{H^s(U_\tau)}+\|\gamma^{(d)}\|^2_{H^s(U_\tau)}\big\}\mathrm{vol}_{g^{[\bf n]}}d\tau+Ct^{2-\varepsilon}t^{2M-2N_0}
\end{align}
Next, we use Lemma \ref{lem:IBP} to integrate by parts in $\partial_t,e_C,e_I,e_D$, recalling that $e^{(d)},k^{(d)},\gamma^{(d)}$ vanish on $U_\eta$:
\begin{align}\label{en.ineq2}
\notag&\frac{1}{2}t^{-2N_0}\big\{\|e^{(d)}\|^2_{H^s(U_t)}+\|k^{(d)}\|^2_{H^s(U_t)}+\frac{1}{2}\|\gamma^{(d)}\|^2_{H^s(U_t)}\big\}\\
\notag&+\sum_{|\alpha|\leq s}\sum_{b,\pm}\int_{\mathcal{H}^\pm_b}\frac{1}{2}\tau^{-2N_0}n_{\mathcal{H}_b^\pm}^0\big\{\partial_x^\alpha e_{Ia}^{(d)}\partial_x^\alpha e_{Ia}^{(d)}+\partial_x^\alpha k_{IJ}^{(d)}\partial_x^\alpha k_{IJ}^{(d)}+\frac{1}{2}\partial_x^\alpha \gamma_{IJB}^{(d)}\partial_x^\alpha \gamma_{IJB}^{(d)}\big\}\mathrm{vol}_{\mathcal{H}^\pm_b}\\
&+\int^t_\eta\frac{N_0-C_*-1}{\tau}\tau^{-2N_0}\big\{\|e^{(d)}\|^2_{H^s(U_\tau)}+\|k^{(d)}\|^2_{H^s(U_\tau)}+\frac{1}{2}\|\gamma^{(d)}\|^2_{H^s(U_\tau)}\big\}d\tau\\
\notag\leq&\sum_{|\alpha|\leq s}\sum_{b,\pm}\int_{\mathcal{H}_b^\pm}
\tau^{-2N_0}n^D_{\mathcal{H}^\pm_b}\big[\partial_x^\alpha k_{IJ}^{(d)}\partial_x^\alpha \gamma_{IJD}^{(d)}-\partial_x^\alpha k_{DJ}^{(d)}\partial_x^\alpha \gamma_{IJI}^{(d)}+\partial_x^\alpha k_{II}^{(d)}\partial_x^\alpha \gamma_{CDC}^{(d)}\big]\mathrm{vol}_{\mathcal{H}_b^\pm}\\
\notag&+\sum_{|\alpha|\leq s}\sum_{b,\pm}\int_{\mathcal{H}_b^\pm}
C\tau^4\tau^{-2N_0}\big\{\partial_x^\alpha e_{Ia}^{(d)}\partial_x^\alpha e_{Ia}^{(d)}+\partial_x^\alpha k_{IJ}^{(d)}\partial_x^\alpha k_{IJ}^{(d)}+\frac{1}{2}\partial_x^\alpha \gamma_{IJB}^{(d)}\partial_x^\alpha \gamma_{IJB}^{(d)}\big\}
\mathrm{vol}_{\mathcal{H}_b^\pm}\\
\notag&+\int^t_\eta C\tau^{-1+\varepsilon}\tau^{-2N_0}\big\{\|e^{(d)}\|^2_{H^s(U_\tau)}
+\|k^{(d)}\|^2_{H^s(U_\tau)}+\|\gamma^{(d)}\|^2_{H^s(U_\tau)}\big\}d\tau+Ct^{2-\varepsilon}t^{2M-2N_0}
\end{align}
To conclude the estimate \eqref{main.en.ineq}, we need to show that the $\mathcal{H}$-boundary terms in the previous LHS can absorb the ones in the RHS. 
This is clearly possible by shrinking the interval of existence in a manner that depends on $n$, but it is independent of $\eta$, and by taking $\sigma$ sufficiently large such that 
\begin{align}\label{choice.sigma}
C\tau^4+10\sqrt{n^D_{\mathcal{H}^\pm_{\underline{b}}}n^D_{\mathcal{H}^\pm_{\underline{b}}}}<\frac{1}{4}n^0_{\mathcal{H}^\pm_b}
\end{align}
The latter is possible since $n^D_{\mathcal{H}^\pm_b}\sim \sigma^{-1}$, $n^0_{\mathcal{H}^\pm_b}\sim1$, for $\sigma$ sufficiently large. Replacing $C_*+1$ by another constant labeled again $C_*$, still independent of $n$, completes the proof of the proposition.
\end{proof}
%


\subsubsection{Proof of Proposition \ref{prop:loc.exist1}}\label{subsubsec:prop.proof}

With the energy inequality \eqref{main.en.ineq} at our disposal,
we choose $N_0\in\mathbb{N}$ such that $N_0>C_*$. The latter number depends only on the $p_I$'s and their $s$ coordinate derivatives, for some $s\ge4$, which are all fixed to begin with. By Lemma \ref{lem:cal.I}, we then choose $n=n(N_0)$ sufficiently large, such that $2M-2N_0+2-\varepsilon>0$. For these choices of parameters, \eqref{main.en.ineq} implies that
\begin{align}\label{main.en.ineq3}
& t^{-2N_0}\big\{\|e^{(d)}\|^2_{H^s(U_t)}+\|k^{(d)}\|^2_{H^s(U_t)}+\|\gamma^{(d)}\|^2_{H^s(U_t)}\big\}\\
\leq&\int^t_\eta C\tau^{-1+\varepsilon}\tau^{-2N_0}\big\{\|e^{(d)}\|^2_{H^s(U_\tau)}+\|k^{(d)}\|^2_{H^s(U_\tau)}+\|\gamma^{(d)}\|^2_{H^s(U_\tau)}\big\}d\tau+C t^{2M-2N_0+2-\varepsilon}.\notag
\end{align}
Applying Gr\"onwall's inequality in $[\eta,t]$ yields the energy estimate
\begin{align}\label{main.en.ineq3}
t^{-2N_0}\big\{\|e^{(d)}\|^2_{H^s(U_t)}+\|k^{(d)}\|^2_{H^s(U_t)}+\|\gamma^{(d)}\|^2_{H^s(U_t)}\big\}\leq C t^{2M-2N_0+2-\varepsilon} e^{\int^t_\eta  C\tau^{-1+\varepsilon}}.
\end{align}
In particular, shrinking the original interval of existence $[\eta,T(\eta,n)]$ if necessary, in a way that only depends on $n$, we have that
\begin{align}\label{main.en.ineq4}
\|e^{(d)}\|^2_{H^s(U_t)}+\|k^{(d)}\|^2_{H^s(U_t)}+\|\gamma^{(d)}\|^2_{H^s(U_t)}\leq t^{2N_0},
\end{align}
for all $t\in[\eta,T_{Boot})$. For $N_0>5$, the latter estimate is an improvement of our bootstrap assumptions \eqref{Boots}. A standard continuation argument implies that the time interval on which \eqref{main.en.ineq4} holds true can be enlarged up to some $[\eta,T]$, where $T=T_{N_0,s,n}$ depends only on the parameters chosen above and not on $\eta$. 

\section{Recovery of the Einstein vacuum equations}\label{sec:EVE}

In Section \ref{sec:actual.sol} we constructed a singular solution $e_{Ia},k_{IJ},\gamma_{IJB}$ to the modified system of equations \eqref{e0.eIi2}-\eqref{e0.gamma_IJB.mod2}. Now we need to show that it actually corresponds to a metric that satisfies the Einstein vacuum equations. 

Consider the metric ${\bf g}$ of the form \eqref{metric}, for which $e_I=e_{Ia}\partial_a$ is a $g$-orthonormal frame. This completely determines the metric. The variables $k_{IJ},\gamma_{IJB}$ that we have solved for, using \eqref{e0.kIJ.mod2}-\eqref{e0.gamma_IJB.mod2} are not a priori the connection coefficients of $e_0=\partial_t$, $e_I$, since the equations \eqref{e0.k}, \eqref{gammaIJB.eq} have been modified using the constraints. Nevertheless, they define a connection $\widetilde{D}$ as follows:
\begin{align}\label{D.tilde}
\widetilde{D}_{e_0}e_\mu=0,\qquad\widetilde{D}_{e_I}e_0=-k_{IJ}e_J,\qquad \widetilde{D}_{e_I}e_J=-k_{IJ}e_0+\gamma_{IJB}e_B.
\end{align}
By Lemma \ref{lem:var.d.symm}, we have that $\widetilde{D}$ is compatible with ${\bf g}$. However, it is not necessarily torsion-free. Define 
\begin{align}\label{torsion}
C_{\alpha\mu\nu}={\bf g}([e_\alpha,e_\mu]-\widetilde{D}_{e_\alpha}e_\mu+\widetilde{D}_{e_\mu}e_\alpha,e_\nu)=-C_{\mu\alpha\nu}
\end{align}
and the curvatures
\begin{align}\label{Riem.tilde}
\notag\widetilde{\bf R}_{\alpha\beta\mu\nu}=&\,{\bf g}\big((\widetilde{D}_{e_\alpha}\widetilde{D}_{e_\beta}-\widetilde{D}_{e_\alpha}\widetilde{D}_{e_\beta}-\widetilde{D}_{\widetilde{D}_{e_\alpha}e_\beta-\widetilde{D}_{e_\beta}e_\alpha})e_\mu,e_\nu\big),\\ 
\widetilde{\bf R}_{\beta\mu}=&-\widetilde{\bf R}_{0\beta\mu0}+\widetilde{\bf R}_{I\beta\mu I},\\ 
\notag\widetilde{\bf R}=&-\widetilde{\bf R}_{00}+\widetilde{\bf R}_{II}.
\end{align}
It turns out that proving $\widetilde{D}$ is the actual Levi-Civita connection $D$ of ${\bf g}$ and $k_{IJ},\gamma_{IJB}$ the expected connection coefficients, must be done at the same time as showing that ${\bf g}$ is a solution to the Einstein vacuum equations. The following lemma is contained in \cite[Section 4]{FSm}.
\begin{lemma}\label{lem:const.prop}
The variables $C_{\alpha\mu\nu},\widetilde{\bf R}_{\beta\mu}$ satisfy:
\begin{align}\label{C.R.id}
C_{\alpha\beta0}=C_{0\alpha\beta}=0,\qquad
\widetilde{\bf R}_{IJ}+\widetilde{\bf R}_{JI}=-\delta_{IJ}\widetilde{\bf R}_{00},
\qquad\widetilde{\bf R}_{0I}=-\widetilde{\bf R}_{I0}
\end{align}
and
\begin{align}\label{C.R.eq}
\notag\partial_tC_{IJB}=&\,(k\star C)_{IJB}-\delta_{IB}\widetilde{\bf R}_{J0}+\delta_{JB}\widetilde{\bf R}_{I0},\\
\partial_t\widetilde{\bf R}_{I0}=&\,e_I\widetilde{\bf R}_{00}+\frac{1}{2}e_J(\widetilde{\bf R}_{IJ}-\widetilde{\bf R}_{JI})+(k\star \widetilde{\bf R}+\gamma\star \widetilde{\bf R})_I+(C\star Q)_I,\\
\notag\partial_t\widetilde{\bf R}_{00}=&\,e_I\widetilde{\bf R}_{I0}+k\star \widetilde{\bf R}+\gamma\star \widetilde{\bf R}+C\star Q,\\
\notag \partial_t(\widetilde{\bf R}_{IJ}-\widetilde{\bf R}_{JI})=&\,e_J\widetilde{\bf R}_{I0}-e_I\widetilde{\bf R}_{J0}+(k\star \widetilde{\bf R}+\gamma\star \widetilde{\bf R})_{IJ}+(C\star Q)_{IJ},
\end{align}
where the indices in the $\star$ product terms do not matter, the factors $\widetilde{\bf R}$ in these terms are $\widetilde{\bf R}_{\beta\mu}$ components, and $Q$ is a linear expression in $C,\gamma\star\gamma,\gamma\star k,k\star k, ek,e\gamma$ (different in each equation).
\end{lemma}
\begin{remark}
The equations in \eqref{C.R.eq} constitute a first order symmetric hyperbolic system for $C_{IJB}, \widetilde{\bf R}_{I0},\widetilde{\bf R}_{00},\widetilde{\bf R}_{IJ}-\widetilde{\bf R}_{JI}$.
\end{remark}
Next, we argue that the corresponding variables for the approximate solution given by Theorem \ref{thm:approx.sol} are increasingly decaying for $n$ sufficiently large, as $t\rightarrow0$. 
\begin{lemma}\label{lem:tilde.approx}
Consider $\widetilde{D}^{[\bf n]}$, $\widetilde{\bf R}_{\beta\mu}^{[\bf n]}, C_{\alpha\beta\mu}^{[\bf n]}$ to have exactly analogous definitions to \eqref{torsion}, \eqref{Riem.tilde}, but with ${\bf g},e_I,k_{IJ},\gamma_{IJB}$ replaced by the iterates ${\bf g}^{[\bf n]},e_I^{[\bf n]},k^{[\bf n]}_{IJ},\gamma^{[\bf n]}_{IJB}$. Then we have the following bounds:
\begin{align}\label{tilde.approx}
|C_{\alpha\beta\mu}^{[\bf n]}|\leq C t^{-1+n\varepsilon},\qquad |\widetilde{\bf R}_{\beta\mu}^{[\bf n]}|\leq Ct^{-2+n\varepsilon},
\end{align}
for all indices $\alpha,\beta,\mu$, and for all $t\in(0,T_{N_0,s,n}]$.
\end{lemma}
\begin{proof}
By \eqref{gamma.n}, we notice that $C_{IJB}^{[\bf n]}$. For the rest of the indices we have 
\begin{align*}
C_{IJ0}^{[\bf n]}=&-\widetilde{k}^{[\bf n]}_{IJ}+\widetilde{k}_{JI}^{[\bf n]}=0,\\
C_{0IJ}^{[\bf n]}=&\,{\bf g}^{[\bf n]}([e_0,e_I^{[\bf n]}],e_J^{[\bf n]})+{\bf g}^{[\bf n]}(\widetilde{D}_{e_I^{[\bf n]}}^{[\bf n]}e_0,e_J^{[\bf n]})
=(\partial_te_{Ia}^{[\bf n]})\omega_{aJ}^{[\bf n]}-k^{[\bf n]}_{IJ}\\
\tag{by \eqref{eIi.it}}=&\,\bigg(k_{\underline I \underline I}^{[\bf n]}e_{\underline Ia}^{[\bf n]}+\sum_{C\neq I}k_{IC}^{[\bf n-1]}e_{Ca}^{[\bf n-1]}\bigg)\omega^{[\bf n]}_{aJ}-k^{[\bf n]}_{IJ}\\
=&\,k_{\underline I \underline I}^{[\bf n]}\delta_{\underline I J}-k_{IJ}^{[\bf n]}+\sum_{C\neq I}k_{IC}^{[\bf n-1]}\delta_{CJ}
+\sum_{C\neq I}k_{IC}^{[\bf n-1]}e_{Ca}^{[\bf n-1]}(\omega^{[\bf n]}_{aJ}-\omega^{[\bf n-1]}_{aJ})
\end{align*}  
Taking cases $I=J$ and $I\neq J$, the first three terms in the last RHS amount to 0 and $k_{IJ}^{[\bf n-1]}-k_{IJ}^{[\bf n]}$ respectively. Hence, the bound \eqref{tilde.approx} for $C_{\alpha\beta \mu}^{[\bf n]}$ follows from point \ref{item.approx1} of Theorem \ref{thm:approx.sol} and Proposition \ref{prop:ind2}. 

Given the definition of curvature $\widetilde{\bf R}_{\alpha\beta\mu\nu}^{[\bf n]}$, analogous to \eqref{Riem.tilde}, we have the schematic relations
\begin{align*}
\widetilde{\bf R}_{\alpha\beta\mu\nu}^{[\bf n]}= {\bf R}_{\alpha\beta\mu\nu}^{[\bf n]}+e^{\bf [n]}(k^{[\bf n]} -\widetilde{k}^{[\bf n]})+\gamma^{\bf [n]}\star(k^{\bf [n]}-\widetilde{k}^{\bf [n]})+Q(k^{\bf [n]})-Q(\widetilde{k}^{\bf [n]})
\end{align*}
where $e^{\bf [n]}=\partial_t,e_I^{\bf [n]}$ and $Q$ a homogeneous quadratic expression. Contracting $\alpha, \nu$, we obtain the bound \eqref{tilde.approx} for ${\bf R}_{\beta\mu}^{\bf [n]}$ by using Theorem \ref{thm:approx.sol} and Lemma \ref{lem:k.approx.2ndfund}.
\end{proof}
We can now proceed to show that $\widetilde{D}=D$ and the vanishing of the Ricci tensor of ${\bf g}$.
\begin{proposition}\label{prop:EVE}
Let ${\bf g},\widetilde{D}$ be the metric and connection constructed from the solution furnished by Theorem \ref{thm:loc.exist}, for some $N_0,n_{N_0,s}$ sufficiently large, as discussed in the beginning of this section. Then $\widetilde{D}=D$ is the Levi-Civita connection of ${\bf g}$ and moreover, ${\bf g}$ is a solution to the Einstein vacuum equations, ie. ${\bf R}_{\beta\mu}=0$ for all indices.   
\end{proposition}
\begin{proof}
Define the energy 
\begin{align}\label{C.R.en}
\begin{split}
E(t):=&\,t^{-2N_1-2}\sum_{I,J,B}\|C_{IJB}\|_{L^2(U_t)}^2+t^{-2N_1}\sum_I\|\widetilde{\bf R}_{I0}\|^2_{L^2(U_t)}\\
&+t^{-2N_1}\|\widetilde{\bf R}_{00}\|^2_{L^2(U_t)}+\frac{1}{2}t^{-2N_1}\sum_{IJ}\|\widetilde{\bf R}_{IJ}-\widetilde{\bf R}_{JI}\|^2_{L^2(U_t)}
\end{split}
\end{align}
By virtue of the bounds \eqref{tilde.approx} and the energy estimate \eqref{loc.exist.est}, we notice for any $N_1\in\mathbb{N}$, there exists $N_0,n$ sufficiently large, such that $E(t)\rightarrow0$, as $t\rightarrow0$. 

For this choice of parameters $N_0,n$, a similar energy argument to the one in the proof of Proposition \ref{prop:en.est} gives the energy inequality
\begin{align}\label{C.R.en.ineq}
\frac{1}{2} E(t)\leq \int^t_0(\frac{N_1-C_*}{\tau}+C\tau^{-1+\varepsilon})E(\tau)d\tau,
\end{align} 
for all $t\in(0,T_{N_0,s,n}]$. The terms that contribute to the constant $C_*$ are $k\star C$, $\delta_{IB}\widetilde{\bf R}_{J0}$, $\delta_{JB}\widetilde{\bf R}_{I0}$ in the RHS of the equation for $C_{IJB}$ in \eqref{C.R.eq}, $k\star\widetilde{\bf R}$ and $C\star (k\star k)$, $C\star \partial_t k$ in the RHSs of the equations for $\widetilde{\bf R}_{I0},\widetilde{\bf R}_{00},\widetilde{\bf R}_{IJ}-\widetilde{\bf R}_{JI}$. More precisely, only the leading order part of the factors $k=k^{[\bf 0]}+O(t^{-1+\varepsilon})$, $\partial_tk=\partial_tk^{[\bf 0]}+O(t^{-2+\varepsilon})$ gives rise to terms that contribute to $C_*$. Hence, $C_*$ depends only on the $p_I$'s and is independent of $n$. Taking $N_1\ge C_*$ to begin with (possibly increasing $N_0,n$), we conclude that 
\begin{align}\label{C.R.en.ineq}
E(t)\leq \int^t_0C\tau^{-1+\varepsilon}E(\tau)d\tau,\qquad t\in(0,T_{N_0,s,n}].
\end{align} 
Gr\"onwall's inequality implies that $E(t)\equiv0$.

Thus, $C_{IJB},\widetilde{\bf R}_{I0},\widetilde{\bf R}_{00},\widetilde{\bf R}_{IJ}-\widetilde{\bf R}_{JI}$ vanish everywhere. By \eqref{C.R.id}, we have that $C_{\alpha\beta\mu},\widetilde{\bf R}_{\beta\mu}=0$ everywhere for all indices. Hence, $\widetilde{D}=D$ and $\widetilde{\bf R}_{\beta\mu}={\bf R}_{\beta\mu}=0$.
\end{proof}
\section{Uniqueness and smoothness of solutions: Proof of Theorems \ref{thm:smooth}, \ref{thm:uniq}}\label{sec:smooth.uniq}

First, we prove the uniqueness statement $(i)$ in Theorem \ref{thm:uniq}. This will in fact be used to prove Theorem \ref{thm:smooth}. Then we prove that point $(ii)$ in Theorem \ref{thm:uniq} implies point $(i)$. 

\subsection{Uniqueness}\label{subsec:uniq}

{\it Point $(i)$ in Theorem \ref{thm:uniq}.} Let ${\bf g},\widetilde{\bf g}$ be two solutions to \eqref{EVE} of the form \eqref{metric}, satisfying \eqref{uniq.cond.a}-\eqref{uniq.cond.b}. We first note that they have the same asymptotic data $p_i,c_{ij}$, hence, they admit the same approximate metric ${\bf g}^{[\bf n]}$. Define 
\begin{align}\label{hat.var}
\notag\widehat{e}_{Ia}=&\,e_{Ia}-\widetilde{e}_{Ia}=(e_{Ia}-e^{[\bf n]}_{Ia})-(\widetilde{e}_{Ia}-e^{[\bf n]}_{Ia})=e^{(d)}_{Ia}-\widetilde{e}^{(d)}_{Ia}\\
\widehat{k}_{IJ}=&\,k_{IJ}-\widetilde{k}_{IJ}=(k_{IJ}-k^{[\bf n]}_{IJ})-(\widetilde{k}_{IJ}-k^{[\bf n]}_{IJ})=k^{(d)}_{IJ}-\widetilde{k}^{(d)}_{IJ}\\
\notag\widehat{\gamma}_{IJB}=&\,\gamma_{IJB}-\widetilde{\gamma}_{IJB}=(\gamma_{IJB}-\gamma^{[\bf n]}_{IJB})-(\widetilde{\gamma}_{IJB}-\gamma^{[\bf n]}_{IJB})
=\gamma_{IJB}^{(d)}-\widetilde{\gamma}_{IJB}^{(d)}
\end{align}
The variables $\widehat{e}_{Ia},\widehat{k}_{IJ},\widehat{\gamma}_{IJB}$ satisfy a system similar to \eqref{e0.eI.d}-\eqref{e0.gamma_IJB.d}, only without the terms $(\mathcal{I}^{[\bf n]}_e)_{Ia},(\mathcal{I}^{[\bf n]}_k)_{IJ},(\mathcal{I}_\gamma^{[\bf n]})_{IJB}$ in the corresponding RHSs. By taking $M_0>N_0$ in \eqref{uniq.cond.b}, we have that
\begin{align}\label{lim.hat.L2}
\lim_{t\rightarrow0}t^{-2N_0}\big\{\|\widehat{e}\|^2_{L^2(U_t)}+\|\widehat{k}\|_{L^2(U_t)}^2+\|\widehat{\gamma}\|_{L^2(U_t)}^2\big\}=0,
\end{align}
where $N_0\ge C_*$ depends only on the $L^\infty(U_0)$ norms of the $p_i$'s.
Next, we observe that thanks to \eqref{uniq.cond.a}, we may derive an $L^2(U_t)$ estimate  for the variables $\widehat{e}_{Ia},\widehat{k}_{IJ},\widehat{\gamma}_{IJB}$, similar to \eqref{main.en.ineq} for $s=0$, in all of $(0,T]$, since we control the pointwise leading order behavior of the original two sets variables and their first derivatives (for $\gamma,\widetilde{\gamma}$, we use the formula \eqref{gammaIJB} and the behavior of $\omega,\widetilde{\omega}$, inferred by the assumed bounds on the frame coefficients, cf. Lemma \ref{lem:pre2}). Thus, for $N_0\ge C_*$, we have 
\begin{align*}
&t^{-2N_0}\big\{\|\widehat{e}^{(d)}\|^2_{L^2(U_t)}+\|\widehat{k}^{(d)}\|^2_{L^2(U_t)}+\frac{1}{2}\|\widehat{\gamma}^{(d)}\|^2_{L^2(U_t)}\big\}\\
\leq&\int^t_0 C\tau^{-1+\varepsilon}\tau^{-2N_0}\big\{\|\widehat e^{(d)}\|^2_{L^2(U_\tau)}+\|\widehat k^{(d)}\|^2_{L^2(U_\tau)}+\|\widehat\gamma^{(d)}\|^2_{L^2(U_\tau)}\big\}d\tau,
\end{align*}
where we note that there is no term in the previous RHS analogous to $Ct^{2M-2N_0+2-\varepsilon}$ in \eqref{main.en.ineq}, since there are no inhomogeneous terms depending purely on the iterates. Gronwall's inequality yields that the differences $\widehat{e}_{Ia},\widehat{k}_{IJ},\widehat{\gamma}_{IJB}$ vanish everywhere. Thus, the two solutions ${\bf g},\widetilde{\bf g}$ coincide.

{\it Point $(ii)$ in Theorem \ref{thm:uniq}.} Condition \eqref{uniq.cond2} clearly contains \eqref{uniq.cond.a}. We will show that \eqref{uniq.cond.b} in point $(i)$ also holds, hence, implying that the two solutions are equal. Let $e^{(d)}_{Ia},k_{IJ}^{(d)},\gamma_{IJB}^{(d)}$ be as in \eqref{e.k.gamma.d}. By point \ref{item.approx1} in Theorem \ref{thm:approx.sol} and the triangle inequality we have that 
\begin{align}\label{e.k.d.CM.est}
|\partial_x^\alpha e_{Ia}^{(d)}|\leq C_{\alpha,n} t^{p_I-2\min\{p_I,p_a\}+\varepsilon},\qquad
|\partial_x^\alpha k_{IJ}^{(d)}|\leq C_{\alpha,n} t^{-1+|p_I-p_J|+\varepsilon},
\end{align}
for every $(t,x)\in\{U_t\}_{(0,T_{N_0,s,n_0}]}$ and $|\alpha|\leq M_1$. Similarly to Lemma \ref{lem:pre2}, from the bound \eqref{e.k.d.CM.est} on the frame components $e_{Ia}^{(d)}$, we deduce the corresponding bound on the co-frame components $\omega^{(d)}_{bC}=\omega_{bC}-\omega_{bC}^{[\bf n]}$: 
\begin{align}\label{omega.d.CM.est}
|\partial_x^\alpha \omega_{bC}^{(d)}|\leq C_{\alpha,n} t^{-p_C+2\max\{p_b,p_C\}+\varepsilon}
\end{align}
for every $(t,x)\in\{U_t\}_{(0,T_{N_0,s,n}]}$ and $|\alpha|\leq M_1$. 

Multiplying \eqref{e0.eI.d} with $t^{p_{\underline I}}$, integrating in $(0,T_{N_0,s,n}]$, differentiating in $\partial_x^\alpha$, and using the bounds \eqref{e.k.d.CM.est}, \eqref{cal.I.est} for $(\mathcal{I}_e^{[\bf n]})_{Ia}$, we obtain the estimate:
\begin{align}\label{e.d.CM.est.imp}
|\partial_x^\alpha e_{Ia}^{(d)}|\leq C_{\alpha,n} (t^{p_I-2\min\{p_I,p_a\}+2\varepsilon}+t^{M}),
\end{align}
for all $|\alpha|\leq M_1$,
which is $t^\varepsilon$ better than \eqref{e.k.d.CM.est}. Recall that in Lemma \ref{lem:cal.I}, $M=M(n)\rightarrow+\infty$, as $n\rightarrow+\infty$, so we can take $n$ sufficiently large to begin with, such that $M\ge M_0$. The latter also implies the corresponding improved bound for the co-frame components:
\begin{align}\label{omega.d.CM.est.imp}
|\partial_x^\alpha \omega_{bC}^{(d)}|\leq C_{\alpha,n} (t^{-p_C+2\max\{p_b,p_C\}+2\varepsilon}+t^{M}),
\end{align}

Next, we recall equations \eqref{e0.k}, \eqref{kIJ.it} to compute
\begin{align}\label{k.d.eq}
\notag\partial_tk_{IJ}^{(d)}+\frac{1}{t}k^{(d)}_{IJ}+\delta_{\underline I J}\frac{p_{\underline I}}{t}k_{CC}^{(d)}=&\,R_{IJ}-R_{IJ}^{[\bf n]}\\
&+k_{CC}^{(d)}(k^{[\bf n]}_{IJ}-k^{[\bf 0]}_{IJ})
+(k_{CC}^{[\bf n]}-k_{CC}^{[\bf 0]})k^{(d)}_{IJ}+k_{CC}^{(d)}k^{(d)}_{IJ}\\
\notag&+R_{IJ}^{[\bf n]}-R_{IJ}^{[\bf n-1]}+(k^{[\bf n]}_{CC}-k_{CC}^{[\bf n-1]})k^{[\bf n]}_{IJ}
\end{align}
By Proposition \ref{prop:ind2}, we have that 
\begin{align}\label{k.d.eq.inh}
\big|\partial_x^\alpha\big[R_{IJ}^{[\bf n]}-R_{IJ}^{[\bf n-1]}+(k^{[\bf n]}_{CC}-k_{CC}^{[\bf n-1]})k^{[\bf n]}_{IJ}\big]\big|\leq C_{\alpha,n} t^{-1+|p_I-p_J|+n\varepsilon}
\end{align}
Moreover, as in Lemma \ref{lem:approx1}, expressing $R_{IJ}-R_{IJ}^{[\bf n]}$ in terms of $e_{Ia}^{(d)},\omega_{bC}^{(d)}$ and using the bounds \eqref{e.d.CM.est.imp}, \eqref{omega.d.CM.est.imp}, we deduce the estimate
\begin{align}\label{Ric.d.CM.est}
|\partial_x^\alpha (R_{IJ}-R_{IJ}^{[\bf n]})|\leq C_{\alpha,n}(t^{-2+|p_I-p_J|+2\varepsilon}+t^M)
\end{align}
Also, from \eqref{e.k.d.CM.est} and point \ref{item.approx1} in Theorem \ref{thm:approx.sol}, it holds
\begin{align}\label{k.d.eq.hom}
\big|\partial_x^\alpha\big[k_{CC}^{(d)}(k^{[\bf n]}_{IJ}-k^{[\bf 0]}_{IJ})
+(k_{CC}^{[\bf n]}-k_{CC}^{[\bf 0]})k^{(d)}_{IJ}+k_{CC}^{(d)}k^{(d)}_{IJ}\big]\big|\leq C_{\alpha,n}t^{-1+|p_I-p_J|+2\varepsilon}.
\end{align}
Combining \eqref{k.d.eq.inh}-\eqref{k.d.eq.hom}, we improve the bound \eqref{e.k.d.CM.est} for $k_{IJ}^{(d)}$ as follows: First, trace \eqref{k.d.eq} in $(I;J)$, multiply with $t^2$, differentiate in $\partial_x^\alpha$, and integrate in $(0,T_{N_0,s,n_0}]$ to obtain the bound
\begin{align}\label{trk.d.CM.est.imp}
|\partial_x^\alpha k_{CC}^{(d)}|\leq C_{\alpha,n}t^{-1+2\varepsilon}.
\end{align}
Then, we take the term $\delta_{\underline I J}p_{\underline I}t^{-1}k_{CC}^{(d)}$ to the RHS in \eqref{k.d.eq}, multiply with $t$, integrate in $(0,T_{N_0,s,n_0}]$, differentiate in $\partial_x^\alpha$, and apply the bounds \eqref{k.d.eq.inh}-\eqref{trk.d.CM.est.imp} to conclude that
\begin{align}\label{k.d.CM.est.imp}
|\partial_x^\alpha k_{IJ}^{(d)}|\leq C_{\alpha,n}t^{-1+|p_I-p_J|+2\varepsilon},
\end{align}
for all $(t,x)\in\{U_t\}_{(0,T_{N_0,s,n}]}$ and $|\alpha|\leq M_1-2$.
We may continue iteratively improving the bounds on $e^{(d)}_{Ia},\omega^{(d)}_{bC},k_{IJ}^{(d)}$ by $t^\varepsilon$ each time, as long as $n\varepsilon\leq M-1$, hence, deriving the bounds
\begin{align}\label{e.k.d.CM.est.imp}
|\partial_x^\alpha e_{Ia}^{(d)}|\leq C_{\alpha,n} t^{M-2},\qquad
|\partial_x^\alpha \omega_{bC}^{(d)}|\leq C_{\alpha,n} t^{M-2},\qquad
|\partial_x^\alpha k_{IJ}^{(d)}|\leq C_{\alpha,n} t^{M-2}
\end{align}
Of course, for each $t^\varepsilon$ improvement we sacrifice two spatial derivatives, which is possible provided $M_1\sim M/\varepsilon$. This in turn implies that 
\begin{align}\label{gamma.d.CM.est}
|\partial_x^\alpha\gamma_{IJB}^{(d)}|\leq C_{\alpha,n}t^{M-4},
\end{align}
Thus, taking into account the volume form and shrinking $T_{N_0,s,n_0}$ if necessary, the following energy estimate is valid:
\begin{align}\label{e.k.gamma.d.H4.est}
\|e^{(d)}\|^2_{H^4(U_t)}+\|k^{(d)}\|^2_{H^4(U_t)}+\|\gamma^{(d)}\|^2_{H^4(U_t)}\leq t^{2M-7}.
\end{align}
Obviously, given condition \eqref{uniq.cond2}, the above argument applies also to the reduced variables $\widetilde{e}_{Ia},\widetilde{k}_{IJ},\widetilde{\gamma}_{IJB}$ of the solution $\widetilde{\bf g}$. Hence, the corresponding differences $\widetilde{e}^{(d)}_{Ia},\widetilde{k}^{(d)}_{IJ},\widetilde{\gamma}^{(d)}_{IJB}$ satisfy the analogous energy estimate for the same $n$: 
\begin{align}\label{tilde.e.k.gamma.d.H4.est}
\|\widetilde{e}^{(d)}\|^2_{H^4(U_t)}+\|\widetilde{k}^{(d)}\|^2_{H^4(U_t)}+\|\widetilde{\gamma}^{(d)}\|^2_{H^4(U_t)}\leq t^{2M-7}.
\end{align}
Condition \eqref{uniq.cond.b} now follows from \eqref{e.k.gamma.d.H4.est}, \eqref{tilde.e.k.gamma.d.H4.est} and the triangle inequality, by taking $n$ sufficiently large such that $2M-7\ge 2M_0$. Hence, point $(i)$ can be employed to conclude that the two solutions ${\bf g},\bf\widetilde{g}$ are equal.

\subsection{Smoothness}\label{subsec:smooth}

Let ${\bf g}$ be the solution furnished by Theorem \ref{thm:loc.exist}, for some $s_0\ge4$, defined in the domain $\{U_t\}_{t\in(0,T_{N_0,s_0,n_0}]}$, which is in turn defined relative to ${\bf g}^{[\bf n_0]}$. The corresponding variables $e_{Ia}^{(d)},k^{(d)}_{IJ},\gamma^{(d)}_{IJB}$ satisfy the estimate \eqref{loc.exist.est}, for all $t\in(0,T_{N_0,s_0,n_0}]$. 

Consider an $s_1>s_0$ and the corresponding solution $\widetilde{\bf g}$ furnished by Theorem \ref{thm:loc.exist}. Let $N_1,n_1,T_{N_1,s_1,n_1}$ be the parameters and existence time associated with $\widetilde{\bf g}$. It is evident from Proposition \ref{prop:en.est} that the constant $C_*$ in \eqref{main.en.ineq}, for an $H^{s_1}(U_t)$ energy estimate, will be larger from the corresponding constant required for $H^{s_0}(U_t)$, since $C_*$ depends on the $p_i$'s and their $s_1$ derivatives. Therefore, $n_1,N_1$ will possibly be larger than $n_0,N_0$. Nevertheless, we observe that defining the domain $\{U_t\}_{t\in(0,T_{N_1,s_1,n_1}]}$ in Section \ref{subsec:domain} relative to ${\bf g}^{[\bf n_0]}$ does not affect the existence proof. Hence, we may assume that the slicing for the two solutions ${\bf g},\widetilde{\bf g}$ is the same, albeit the times of existence differ, $T_{N_1,s_1,n_1}<T_{N_0,s_0,n_0}$.

Next, we observe that the assumptions \eqref{uniq.cond.a}-\eqref{uniq.cond.b} in point $(i)$ of Theorem \ref{thm:uniq} are satisfied by both sets of reduced variables $e_{Ia},k_{IJ},\gamma_{IJB}$ and $\widetilde{e}_{Ia},\widetilde{k}_{IJ},\widetilde{\gamma}_{IJB}$ corresponding to ${\bf g},\widetilde{\bf g}$ respectively. Indeed, this is immediate provided $N_0,N_1>M_0$. Since $M_0$ depends only on the $L^\infty(U_0)$ norms of the $p_i$'s and not their derivatives, we conclude that the two solutions must coincide in their common domain of the definition. Thus, ${\bf g}$ is $H^{s_1}(U_t)$ regular, for every $t\in(0,T_{N_1,s_1,n_1)}]$. Also, it is $H^{s_0}(U_t)$ regular in the time interval $[T_{N_1,s_1,n_1},T_{N_0,s_0,n_0}]$, satisfying the uniform bound \eqref{exist.d.est} for $s=s_0$. By standard continuation criteria for first order symmetric hyperbolic systems, the variables $e_{Ia},k_{IJ},\gamma_{IJB}$ are in fact 
$H^{s_1}(U_t)$ regular, for every $t\in[T_{N_1,s_1,n_1},T_{N_0,s_0,n_0}]$. Since $s_1>s_0$ is arbitrary, we conclude that $e_{Ia},k_{IJ},\gamma_{IJB}$, and hence ${\bf g}$, are smooth in $x$. Using the equations \eqref{e0.e}, \eqref{e0.k}, \eqref{gammaIJB.eq}, iteratively differentiating in $t$, we infer that $e_{Ia},k_{IJ},\gamma_{IJB}$ and ${\bf g}$ are also smooth in $(t,x)\in \{U_t\}_{t\in(0,T_{N_0,s_0,n_0}]}$. This completes the proof of Theorem \ref{thm:smooth}.

\end{document}